\documentclass[11pt,a4paper,leqno]{amsart}

\usepackage[latin1]{inputenc}
\usepackage[T1]{fontenc}
\usepackage{amsfonts}
\usepackage{amsmath}
\usepackage{amssymb}
\usepackage{eurosym}
\usepackage{mathrsfs}
\usepackage{palatino}
\usepackage{pxfonts}
\usepackage{color}
\usepackage{esint}
\usepackage{url}
\usepackage{verbatim}
\usepackage{mathtools}

\usepackage{enumerate}

\usepackage[pagebackref,hypertexnames=false, colorlinks, citecolor=blue, linkcolor=blue, urlcolor=red]{hyperref}

\newcommand{\C}{\mathbb{C}}
\newcommand{\Z}{\mathbb{Z}}

\newcommand{\N}{\mathbb{N}}
\newcommand{\bbP}{\mathbb{P}}
\newcommand{\Q}{\mathbb{Q}}
\newcommand{\R}{\mathbb{R}}

\newcommand{\E}{\mathbb{E}}

\newcommand{\calA}{\mathcal{A}}
\newcommand{\calB}{\mathcal{B}}
\newcommand{\calC}{\mathcal{C}}
\newcommand{\calE}{\mathcal{E}}

\newcommand{\calI}{\mathcal{I}}
\newcommand{\calJ}{\mathcal{J}}
\newcommand{\calK}{\mathcal{K}}
\newcommand{\calL}{\mathcal{L}}
\newcommand{\calM}{\mathcal{M}}
\newcommand{\calP}{\mathcal{P}}

\newcommand{\calS}{\mathcal{S}}
\newcommand{\calT}{\mathcal{T}}
\newcommand{\calR}{\mathcal{R}}
\newcommand{\calY}{\mathcal{Y}}

\newcommand{\calU}{\mathcal{U}}
\newcommand{\calV}{\mathcal{V}}

\newcommand{\scrC}{\mathscr{C}}
\newcommand{\scrD}{\mathscr{D}}

\newcommand{\size}{\operatorname{size}}

\newcommand{\weak}{\operatorname{weak}}
\newcommand{\WBP}{\operatorname{WBP}}
\newcommand{\Leb}{\operatorname{Leb}}

\newcommand{\RMF}{\operatorname{RMF}}
\newcommand{\RM}{\operatorname{RM}}
\newcommand{\CZ}{\operatorname{CZ}}
\newcommand{\hR}{\operatorname{\widehat{\calR}}}

\newcommand{\bla}{\big \langle}
\newcommand{\bra}{\big \rangle}

\newcommand{\ud}[0]{\,\mathrm{d}}

\newcommand{\pair}[2]{\langle #1,#2 \rangle}

\newcommand{\BMO}[0]{\operatorname{BMO}}
\newcommand{\supp}[0]{\operatorname{spt}}
\newcommand{\loc}[0]{\operatorname{loc}}

\newcommand{\Rad}[0]{\operatorname{Rad}}
\newcommand{\UMD}{\operatorname{UMD}}

\newcommand{\good}[0]{\operatorname{good}}
\newcommand{\bad}[0]{\operatorname{bad}}
\newcommand{\ch}[0]{\operatorname{ch}}

\newcommand{\calD}[0]{\mathcal{D}}

\newcommand{\wt}[1]{{\widetilde{#1}}}

\swapnumbers
\theoremstyle{plain}
\newtheorem{thm}[equation]{Theorem}
\newtheorem{lem}[equation]{Lemma}
\newtheorem{prop}[equation]{Proposition}

\theoremstyle{definition}
\newtheorem{defn}[equation]{Definition}

\newtheorem{exmp}[equation]{Example}

\theoremstyle{remark}
\newtheorem{rem}[equation]{Remark}

\numberwithin{equation}{section}

\author{Francesco Di Plinio}
\address[F. D. P.]{Department of Mathematics, Washington University in St. Louis, One Brookings Drive,  St. Louis,  MO 63130-4899, USA}
\email{francesco.diplinio@gmail.com}

\author{Kangwei Li}
\address[K.L.]{\curraddr{Center for Applied Mathematics, Tianjin University, Weijin Road 92, 300072 Tianjin, China}\\ BCAM (Basque Center for Applied Mathematics), Alameda de Mazarredo 14, 48009 Bilbao, Spain}
\email{kangwei.nku@gmail.com}

\author{Henri Martikainen}
\address[H.M.]{Department of Mathematics and Statistics, University of Helsinki, P.O.B. 68, FI-00014 University of Helsinki, Finland}
\email{henri.martikainen@helsinki.fi}

\author{Emil Vuorinen}
\address[E.V.]{\curraddr{Department of Mathematics and Statistics, University of Helsinki, P.O.B. 68, FI-00014 University of Helsinki, Finland} \\  
Centre for Mathematical Sciences, University of Lund, P.O.B. 118, 22100 Lund, Sweden}
\email{j.e.vuorin@gmail.com}

\title{Multilinear operator-valued Calder\'on-Zygmund theory}

\makeatletter
\@namedef{subjclassname@2010}{%
  \textup{2010} Mathematics Subject Classification}
\makeatother

\begin{document}

\begin{abstract}
We develop a general theory of  multilinear singular integrals with operator-valued kernels, acting on tuples of $\UMD$ Banach spaces.   This, in particular, involves  investigating  multilinear variants of the
$\calR$-boundedness condition naturally arising in operator-valued theory. We proceed by establishing a suitable  representation of multilinear, operator-valued singular integrals in terms of operator-valued dyadic shifts and paraproducts, and studying the boundedness of these model operators via dyadic-probabilistic Banach space-valued analysis.   In the bilinear case, we obtain a $T(1)$-type theorem without any additional assumptions on the Banach spaces other than the necessary $\UMD$.  Higher degrees of multilinearity are tackled via a  new formulation of the Rademacher maximal function (RMF) condition. In addition to the natural $\UMD$ lattice cases, our RMF condition covers suitable tuples of non-commutative $L^p$-spaces.
We employ our operator-valued theory to obtain new  multilinear, multi-parameter, operator-valued theorems in the natural setting of $\UMD$ spaces with property $\alpha$.
\end{abstract}

\maketitle

\section{Introduction}
Singular integral operators (SIOs) take the form
\begin{equation}\label{eq:SIO}
  Tf(x)=\int_{\R^d}K(x,y)f(y)\ud y, \qquad x \not \in \operatorname{spt} f,
\end{equation}
and they are abundant in classical and applied harmonic analysis. On the other hand, the UMD (unconditionality of martingale differences) property of a Banach space $X$ is a well-known necessary and sufficient condition for the boundedness of generic singular integrals on $L^p(\R^d;X)$, see Burkholder \cite{Burk1} and Bourgain \cite{Bou}, or the recent book
\cite[Sec.\ 5.2.c and the Notes to Sec.\ 5.2]{HNVW1}. Further progress on  Banach space-valued singular integrals \emph{in the linear setting} has been intertwined with applications to rather disparate areas, such as  the geometry of Banach spaces \cite{JX,KW1},  the regularity theory of elliptic and parabolic equations \cite{CH,Weis}, and the study of quasiconformal mappings \cite{GMSS}.

In the literature, classical singular integral operators with scalar-valued kernels $K$ acting on  $X$-valued functions are usually referred to as  vector-valued, or Banach-valued, singular integral operators. On the other hand,
operator-valued theory
concerns the more general case, where the kernel $K$ itself takes values in bounded linear operators between two Banach spaces $X,Y$. The systematic study of linear, operator-valued singular integrals  was first sparked by the operator-valued Fourier multiplier theorem of Weis \cite{Weis}, which is the central tool in the author's proof of maximal $L^p$-regularity for parabolic equations. In this setting, the requirement of uniform $\mathcal L(X,Y)$-bounds of H\"ormander-Mihlin type on the multiplier must be replaced with the stronger
$\calR$-boundedness condition; essentially, $\{T_1,\ldots, T_n\}$ is an $\mathcal R$-bounded set in $\mathcal L(X,Y)$ if $\{f_j:j=1,\ldots, n\mapsto T_j f_j: j=1,\ldots, n\} $ is a bounded operator from $\mathrm{Rad} X$ to $\mathrm{Rad} Y$, where $\mathrm{Rad}$ is the Rademacher space.  This approach has been later recast by Hyt\"onen and Weis \cite{HW} into a $T(1)$-type theorem for operator-valued kernels. Our broad goal is to provide an extension of \cite{HW} to the multilinear case. Therefore, a first essential difficulty we must deal with is to find a natural multilinear analog of the $\mathcal R$-boundedness condition. As we will see below, this requires additional care when dealing with linearities of degree three and higher.

We now come to a more detailed description of our main object of study. At least heuristically, we may think of an  $n$-linear singular integral operator $T$ acting on  $\R^d$ as being given by,
$$
T(f_1,\ldots, f_n)(x) = U(f_1 \otimes \cdots \otimes f_n)(x,\ldots,x), \qquad x \in \R^d,\, f_i \colon \R^d \to \C,
$$
where $U$ is a linear singular integral operator in $\R^{nd}$.
More precisely, an $n$-linear SIO $T$ has a kernel $K$ satisfying natural estimates that can be deduced from the above heuristic via the linear estimates, and
$$
T(f_1, \ldots, f_n)(x)  = \int_{\R^{nd}} K(x,y)  \prod_{i=1}^n f_i(y_i)  \ud y , \qquad x \not \in \bigcap_{i=1}^n \operatorname{spt} f_i.
$$
The study of multilinear singular multipliers and kernel operators began with the seminal articles of Coifman--Meyer \cite{CoifMey} and Christ--Journ\'e \cite{ChJo}. Motivation for this study comes from applications to elliptic and dispersive partial differential equations, ergodic theory and complex function theory, among others.
We remark that  the first general $T(1)$-type result for multilinear singular kernels, in the scalar case, is due to Grafakos--Torres \cite{GT}.

\subsection{Main results}
Until recently, vector-valued extensions of multilinear Calder\'on-Zyg\-mund operators had mostly been studied in the framework of $\ell^p$ spaces and function lattices, rather than general UMD spaces. Boundedness of  $\ell^p$ extensions is classically obtained through weighted norm inequalities, more recently in connection with localized techniques such as sparse domination: see \cite{GM} and the more recent \cite{CDPO,LMO,LMMOV,Nieraeth} for a non-exhaustive  overview of their interplay.
The paper \cite{DLMV1} finally established $L^p$
bounds for the extensions of $n$-linear SIOs to tuples of $\UMD$ spaces tied by a natural product structure --  for example, the composition of operators in the  Schatten-von Neumann  subclass of the algebra of bounded operators on a Hilbert space.

Before \cite{DLMV1}, Di Plinio and Y.\ Ou   \cite{DO} considered operator-valued bilinear {multiplier} theorems that apply to certain non-lattice $\UMD$ spaces. The results of \cite{DO} may be thought of as a first attempt of generalization of Weis' $\mathcal R$-bounded multiplier theorem \cite{Weis}; however, the treatment of \cite{DO} relies upon additional assumptions on the triple of Banach spaces involved -- some bilinear variants of the RMF conditions appearing in \cite{HMP}. We return to
the role of RMF later.
In the present article,
we develop a complete multilinear operator-valued theory in the non-translation invariant setting and our assumptions are less restrictive. Firstly,
our bilinear theory is completely free of any RMF assumptions, providing the following complete generalization of the $T(1)$-theorem of Hyt\"onen and Weis \cite{HW}, and in particular of \cite{Weis}, to the bilinear case. The key notion for our statement is the $\mathcal R$-bound of a set of trilinear forms $B\in \mathcal B$, $B:X_1\times X_2 \times X_3 \to \mathbb C$. This is  defined as the best constant $C$ such that
\begin{equation}
\label{e:rboundintro}
\sum_{k=1}^N |B_k( x_{1,k},x_{2,k},x_3)| + |B_k( x_{1},x_{2,k},x_{3,k})| + |B_k( x_{1,k},x_{2},x_{3,k})|  \leq C
\end{equation}
for all choices $\{B_1,\ldots, B_N\} \subset \mathcal B$ and integers $N$, for all sequences $\{x_{j,k}\in X_j:k=1,\ldots,N\}$ with $\|x_{j,k}\|_{\mathrm{Rad}(X_j)}\leq 1$ and  vectors $x_j\in X_j$ with $\|x_j\|_{X_j}\leq 1$, for $j=1,2,3$. A satisfactory analogy with the usual notion of $\mathcal R$-boundedness of bilinear forms (adjoint to linear operators) \cite{HW} is the following: for each  fixed $x_3 $ in the unit ball of $X_3$, the bilinear forms $B(\cdot,\cdot,x_3)$ are $\mathcal R$-bounded on $X_1\times X_2$ in the usual sense.

\begin{thm} \label{th:introbil} Let $X_1,X_2,X_3$ be $\UMD$ Banach spaces and $Y_3$ be the Banach dual of $X_3$. Let $T$ be a bilinear SIO on $\R^d$  whose kernel $K$ takes values in bounded bilinear operators from $X_1 \times X_2$ to $Y_3$ and  satisfies $\mathcal R$-boundedness versions, in the sense of \eqref{e:rboundintro} above, of the kernel smoothness and weak-boundedness properties, and some  $T(1)\in \mathrm{BMO}$ properties, see Definition \ref{defn:T1Cond}.
Then
\[
\begin{split}
&  \quad \| T(f_1,f_2) \|_{L^{q_{3}}(\R^d;Y_{3})}
\lesssim   \prod_{m=1}^2 \| f_m\|_{L^{p_m}(\R^d;X_m)},\\ &\forall 1<p_1,p_2\leq \infty, \quad \textstyle \frac{1}{2}<q_3<\infty, \quad  \frac{1}{p_1}+ \frac{1}{p_2}= \frac{1}{q_3}. \end{split}
\]
\end{thm}
Theorem \ref{th:introbil} is a particular case of Theorem \ref{cor:main}. For a detailed description of the assumptions as well as for stronger sparse bound type variants, the reader should consult these results in the main body of the article.

The RMF property of a Banach space $X$, involving $L^p$ estimates for a certain analogue of the Hardy-Littlewood maximal operator obtained by replacing uniform bounds with $\calR$-bounds, dates back  to the work of Hyt\"onen, McIntosh and Portal on the vector-valued Kato square root problem \cite{HMP}: see also \cite{HK1,Kemp1,Kemp2}. The recent multilinear vector-valued (but not operator-valued) setup of
\cite{DLMV1} avoids the use of RMF assumptions in all linearities, arguing by induction on the multilinearity index. On the other hand, the inductive argument of \cite{DLMV1} relies on an abstract assumption modeling the   H\"older type structure typical of concrete examples of Banach $n$-tuples,  such as that of non-commutative $L^p$ spaces with the exponents $p$ satisfying the natural H\"older relation.   Operator-valued analogs of the H\"older-type structures of \cite{DLMV1} is left for future work. In the present article, the $n\geq 3$ analog of Theorem \ref{th:introbil} requires   that the $(n+1)$-tuple of spaces involved obeys to a multilinear version of the RMF assumption, which is described in detail in Subsection \ref{sec:RMF}.

A precise statement of the  $T(1)$-theorem for an  $n$-linear SIO on $\R^d$ with  $\mathcal L(X_1 \times \cdots \times X_n, X_{n+1}^*)$-valued kernel (see Section \ref{sec:MultiSing}
for this notation), when $n=3 $ and higher,  is provided in Theorem \ref{cor:main}. Here, we remark that the RMF setup of Subsection \ref{sec:RMF} applies in the following cases in addition to the trivial $X_j=\mathbb C$ for all $1\leq j \leq n+1$,  see Examples   \ref{ex:LpRMF} and \ref{ex:NCLP} for details:
\begin{itemize}
\item $X_j=L^{p_j}(\Omega; Z_j)$, whenever $1<p_j<\infty$ for all $1\leq j \leq n+1$ is a H\"older tuple and $Z_j$ is a tuple of UMD Banach spaces for which $\mathrm{RMF}$ holds; by iterating this observation, $X_j$ may be any tuple of reflexive Banach mixed norm $L^p$ spaces;
\item let $\mathcal J \subset \{1,\dots,n+1\}$ be a subset of cardinality 3, and for   $j\in \mathcal J$, let $X_j=L^{p_j}(\mathcal A)$, where $1<p_j<\infty$ are as before, and $L^p(\mathcal A)$ is the noncommutative $L^p$ space associated to the von Neumann algebra $\mathcal A$ equipped with a normal, semifinite, faithful trace $\mathcal \tau$, while for $j\in \{1,\ldots, n+1\} \setminus \mathcal J $, $X_j=\mathbb C$. \end{itemize}
The restriction to having at most three non-commutative spaces in the   second example comes from the RMF assumption. Again, this is in contrast with in the recent vector-valued (but not operator-valued) setting of  \cite{DLMV1} where such restriction is unnecessary regardless of the  multilinearity index.

\subsection{Tools and techniques}
Dyadic analysis is an extremely flexible tool, for example, as shown by its role in the Banach space-valued singular integral theory \cite{Fig}, non-doubling singular integral theory \cite{NTV}, and sharp weighted inequalities \cite{Hy}. The dyadic representation theorem of Hyt\"onen \cite{Hy}, which extends the Hilbert transform case of Petermichl \cite{pet2000,Pet}, yields a decomposition of the cancellative part of a singular integral into so-called {\em dyadic shifts}. These shifts have a very natural form generalising the Haar multipliers
$$
  f=\sum_{Q\in\mathcal D}\pair{f}{h_Q}h_Q\mapsto \sum_{Q\in\mathcal D}\lambda_Q\pair{f}{h_Q}h_Q, \qquad |\lambda_Q| \le 1.
$$

H\"anninen and Hyt\"onen \cite{HH} developed the theory of operator-valued shifts in their proof of a $T1$ theorem and a representation theorem
for linear singular integrals on $\UMD$ spaces with operator-valued kernels. See also the paper by Hyt\"onen, Martikainen  and Vuorinen \cite{HMV} for further theory
and applications of operator-valued shifts in the multi-parameter setting.
As an important technical component of this article,  we prove an $n$-linear version of the operator-valued representation theorem, Theorem \ref{thm:Rep}.
Theorem \ref{thm:Rep} is in fact a multilinear, operator-valued generalization of the bilinear, scalar-valued representation theorem which appeared in \cite{LMOV} by Li, Martikainen, Y.\ Ou and Vuorinen.

The next step in our analysis is to show the boundedness of these various multilinear operator-valued dyadic model operators (Theorem \ref{thm:MultiShifts} and Theorem \ref{thm:MultiPara}).  This is quite involved, particularly when working in higher linearities, and requires the development of new abstract theory concerning e.g. the correct notions of $\calR$-boundedness, cf.\
 Section \ref{sec:RbddandRMF}. A combination of the representation theorem with bounds for the model operators yields our main result,  Theorem \ref{cor:main},  which is a boundedness criterion for operator-valued multilinear SIOs.
\subsection{Applications to multiparameter theory}
The size of the singularity of the kernel $K$ in \eqref{eq:SIO} is a fundamental classifying criteria for SIOs.
In classical SIO theory, the appearing kernels are singular exactly when $x=y$. This one-parameter
theory differs from the multi-parameter theory, where the singularities of the kernels are spread over all the hyperplanes of the form $x_i=y_i$, where $x,y\in\R^d$ are written as $x=(x_i)_{i=1}^t\in\R^{d_1}\times\cdots\times\R^{d_t}$ for a given partition $d=d_1+\ldots+d_t$.

The basic philosophy of identifying bi-parameter operators as operator-valued one-parameter operators dates back, at least, to Journ\'e \cite{Jour}. In general settings the $\calR$-boundedness plays an important role. For instance, it is required as an input to apply the abstract results on operator-valued dyadic shifts. Indeed, the $\calR$-boundedness of families of one-parameter operators is necessary for the boundedness (both with or without $\calR$-) of the bi-parameter operators.

In the multilinear setting this general idea is more involved to execute due to the nature of the multilinear $\calR$-boundedness conditions  -- an interesting difference compared to the linear theory. In fact, the notions of multilinear $\calR$-boundedness we use in our previously discussed main results are so weak that they do not appear to be sufficient to conclude the $\calR$-boundedness for families of dyadic model operators. This is why we develop stronger $\calR$-boundedness notions in Section \ref{sec:Rbound}. These can be applied in the multi-parameter context, as detailed  in Section \ref{sec:multilinmultipar}.

A somewhat loose description of our multi-parameter  results is the following. Suppose $X_1, X_2, Y_3$ are UMD spaces with Pisier's property $(\alpha)$. Then bilinear multi-parameter operator-valued shifts (see Section \ref{sec:multilinmultipar}) have the  $L^{p_1}(\R^d; X_1) \times L^{p_2}(\R^d; X_2) \to L^{q_3}(\R^d; Y_3)$ bound whenever $p_1, p_2, q_3 \in (1,\infty)$ with $1/p_1 + 1/p_2 = 1/q_3$. While we do not anymore explicitly pursue the corresponding
(paraproduct free) SIO theory in the bilinear multi-parameter operator-valued setting, this would simply follow from our result on the shifts coupled with a suitable
representation theorem.

\subsection*{Acknowledgement}
F. Di Plinio has been partially supported by the National Science Foundation under the grants   NSF-DMS-1650810 and  NSF-DMS-1800628. 

K. Li was supported by Juan de la Cierva - Formaci\'on 2015 FJCI-2015-24547, by the Basque Government through the BERC
2018-2021 program and by Spanish Ministry of Economy and Competitiveness
MINECO through BCAM Severo Ochoa excellence accreditation SEV-2017-0718
and through project MTM2017-82160-C2-1-P funded by (AEI/FEDER, UE) and
acronym ``HAQMEC''.

H. Martikainen was supported by the Academy of Finland through the grants 294840 and 306901, and by the three-year research grant 75160010 of the University of Helsinki.
He is a member of the Finnish Centre of Excellence in Analysis and Dynamics Research (Academy of Finland project No. 307333).

E. Vuorinen was supported by the Academy of Finland through the grant 306901, by the Finnish Centre of Excellence in Analysis and Dynamics Research, and by
Jenny and Antti Wihuri Foundation.

\section{Definitions and preliminaries}
\subsection{Dyadic notation}\label{sec:randomlattice}
We begin by defining the random dyadic grids that are needed for the probabilistic--dyadic techniques. These
definitions are, for example, as in Nazarov--Treil--Volberg \cite{NTV} and Hyt\"onen \cite{Hy2}.
For each $\omega \in \Omega$, where $\Omega=(\{0,1\}^d)^{\Z}$, we define the lattice
$$
\calD_\omega
= \{Q+\omega \colon Q \in \calD_0\},
$$
where $\calD_0=\{2^{-k}([0,1)^d+m) \colon k \in \Z, m \in \Z^d\}$ is the standard dyadic lattice in $\R^d$
and
$$
Q+\omega:=Q+\sum_{k \colon 2^{-k} < \ell(Q)} \omega_k2^{-k}.
$$
Here the side length of $Q$ is denoted by $\ell(Q)$.
The randomness to $\omega \mapsto \calD_{\omega}$ is induced by equipping $\Omega$ with the natural probability product measure $\bbP$.

Let $X$ be a Banach space and $\calD$ be some fixed dyadic lattice.
Let $L^p(X)=L^p(\R^d;X)$, $p \in (0, \infty]$, be the usual Bochner space of $X$-valued functions.
For a fixed $Q \in \calD$ and $f \in L^1_{\loc}(X)$ we define as follows.
\begin{itemize}
\item If $k \in \Z$, $k \ge 0$, then $Q^{(k)}$ denotes the unique cube $R \in \calD$ for which $Q \subset R$ and
$\ell(Q) = 2^{-k}\ell(R)$.
\item The dyadic children of $Q$ are denoted by $\ch (Q) = \{Q' \in \calD\colon (Q')^{(1)} = Q\}$.
\item An average over $Q$ is $\langle f \rangle_Q = \frac{1}{|Q|} \int_Q f$. We also write $E_Q f=\langle f \rangle_Q 1_Q$.
\item The martingale difference $\Delta_Q f$ is defined by $\Delta_Q f= \sum_{Q' \in \ch (Q)} E_{Q'} f - E_{Q} f$.
\item For $k \in \Z$, $k \ge 0$, we define the martingale difference and average blocks
$$
\Delta_Q^k f=\sum_{\substack{R \in \calD \\ R^{(k)}=Q}} \Delta_{R} f
\quad \text{and} \quad
E_Q^k f=\sum_{\substack{R \in \calD \\ R^{(k)}=Q}} E_{R} f.
$$
\end{itemize}

\subsubsection*{Haar functions}
Haar functions are useful for further decomposing martingale differences $\Delta_Q f$ in terms of rank-one operators.
If $I \subset \R$ is an interval, denote by $I_{l}$ and $I_{r}$ the left and right
halves of the interval $I$, respectively. We define $h_{I}^0 = |I|^{-1/2}1_{I}$ and $h_{I}^1 = |I|^{-1/2}(1_{I_{l}} - 1_{I_{r}})$.
Let now $Q = I_1 \times \cdots \times I_d \in \calD$, and define the Haar function $h_Q^{\eta}$, $\eta = (\eta_1, \ldots, \eta_d) \in \{0,1\}^d$, by setting
\begin{displaymath}
h_Q^{\eta} = h_{I_1}^{\eta_1} \otimes \cdots \otimes h_{I_d}^{\eta_d}.
\end{displaymath}
If $\eta \ne 0$ the Haar function is cancellative: $\int h_Q^{\eta} = 0$. We can now write $\Delta_Q f = \sum_{\eta \ne 0} \langle f, h_{Q}^{\eta}\rangle h_{Q}^{\eta}$,
where $\langle f, h_Q^{\eta} \rangle = \int f h_Q^{\eta}$.
Usually, we exploit notation by suppressing the presence of $\eta$, and simply write $h_Q$ for some $h_Q^{\eta}$, $\eta \ne 0$.
Similarly, we write $\Delta_Q f = \langle f, h_Q \rangle h_Q$.

\subsection{Definitions and properties related to Banach spaces}\label{sec:BanachProperties}
We present the required basics of Banach space theory now -- for an extensive treatment see the books \cite{HNVW1, HNVW2} by Hyt\"onen, van Neerven, Veraar and Weis.

We say that $\{\varepsilon_k\}_k$ is a \emph{collection of independent random signs}, if the following holds.
We have $\varepsilon_k \colon \calM \to \{-1,1\}$, where $(\calM, \rho)$ is a measure space, the collection
$\{\varepsilon_k\}_k$ is independent and
$\rho(\{\varepsilon_k=1\})=\rho(\{\varepsilon_k=-1\})=1/2$.
In what follows, $\{\varepsilon_k\}_k$ will always denote a collection of independent random signs.

Suppose $X$, equipped with the norm $| \cdot |_X$, is a Banach space.
For all $x_1,\dots, x_M \in X$ and $p,q \in (0,\infty)$ there holds that
\begin{equation}\label{eq:KK}
\Big(\E \Big | \sum_{m=1}^M \varepsilon_m x_m \Big |_X^p \Big)^{1/p}
\sim \Big(\E \Big | \sum_{m=1}^M \varepsilon_m x_m \Big |_X^q \Big)^{1/q}
\end{equation}
by the Kahane-Khintchine inequality.
Motivated by this we set
$$
\|(x_m)\|_{\Rad(X)} := \Big(\E \Big | \sum \varepsilon_m x_m \Big |_X^2 \Big)^{1/2},
$$
where the choice of the exponent is thus of no consequence.
The Kahane contraction principle tells us that if $(a_m)_{m=1}^M$ is a sequence of scalars and $p \in (0, \infty]$, then we have
\begin{equation}\label{eq:KCont}
\Big( \E \Big | \sum_{m=1}^M \varepsilon_m a_m x_m \Big |_{X}^p \Big)^{1/p}
\lesssim \max |a_m| \Big( \E \Big | \sum_{m=1}^M \varepsilon_m x_m \Big |_{X}^p \Big)^{1/p}.
\end{equation}
A minor remark is that \eqref{eq:KCont} holds with ``$\le$'' in place of ``$\lesssim$'', if $p \in [1, \infty]$ and $a_m \in \R$ (see \cite{HNVW1}).

A Banach space $X$ is said to be a $\UMD$ \emph{space}, where UMD stands for unconditional martingale differences, if for all $p \in (1,\infty)$, all
$X$-valued $L^p$-martingale difference sequences $(d_j)_{j=1}^k$ and signs $\epsilon_j \in \{-1,1\}$
we have
\begin{equation}\label{eq:UMDDef}
\Big\| \sum_{j=1}^k \epsilon_j d_j \Big \|_{L^p(X)}
\lesssim \Big\| \sum_{j=1}^k d_j \Big \|_{L^p(X)}.
\end{equation}
The $L^p(X)$-norm is with respect to the measure space where the martingale differences are defined.
In fact, a standard property of UMD spaces is that if \eqref{eq:UMDDef} holds for one $p_0\in (1,\infty)$, then it
holds for all $p \in (1, \infty)$.

Sometimes, for example in multi-parameter analysis, the following property is also needed:
for all $N$, all scalars $a_{i,j}$ and all
$e_{i,j} \in X$, $1 \le i, j \le N$, there holds
$$
\Big(\E \E' \Big| \sum_{1 \le i, j \le N} \varepsilon_i \varepsilon_j' a_{i,j} e_{i,j} \Big|_X^2\Big)^{1/2}
\lesssim \max_{i,j} |a_{i,j}| \Big(\E \E' \Big| \sum_{1 \le i, j \le N} \varepsilon_i \varepsilon_j' e_{i,j} \Big|_X^2 \Big)^{1/2}.
$$
If this holds, the Banach space $X$ is said to satisfy the property $(\alpha)$ of Pisier.

\subsubsection*{Random sums and duality}
The reader can e.g. consult the section 7 of the book \cite{HNVW2}, if he or she is unfamiliar with the notions of type and cotype of a Banach space.
What is important for us, though, is simply that all UMD spaces have non-trivial type.
The next lemma appears in Section 7.4.f of \cite{HNVW2}.
\begin{lem}\label{lem:randomduality}
Let $X$ be a Banach space with non-trivial type and let $F \subset X^*$ be a closed subspace of $X^*$ which is norming for $X$. Then
for all finite sequences $e_1, \ldots, e_N \in X$ we have
$$
\E \Big| \sum_{i=1}^N \varepsilon_i e_i \Big|_X \sim \sup \Bigg\{ \Big| \sum_{i=1}^N \langle e_i, e_i^*\rangle \Big|\Bigg\},
$$
where the supremum is taken over all choices $(e_i^*)_{i=1}^N$ in $F$ such that
$$
\E \Big| \sum_{i=1}^N \varepsilon_i e_i^* \Big|_{X^*} \le 1.
$$
\end{lem}

\subsubsection*{The decoupling inequality}
The following decoupling estimates originate from McConnell \cite{Mc}, but in their current form they essentially appear in Hyt\"onen \cite{Hy2}
and H\"anninen--Hyt\"onen \cite{HH}.
We record a special case of the decoupling estimate that is of relevance for us.

Let $\calD$ be a dyadic lattice and $Q \in \calD$.
With $\calV_Q$ we mean the probability measure space $$\calV_Q=(Q, \Leb(Q), |Q|^{-1} \ud x \lfloor Q),$$
where $|Q|^{-1} \ud x \lfloor Q$ is the normalized Lebesgue measure restricted to $Q$
and $\Leb(Q)$ stands for the Lebesgue measurable subsets of $Q$.
We define the product probability space $\calV= \prod_{Q \in \calD} \calV_Q$,
and let $\nu$ be the related measure. If $y \in \calV$, we denote the coordinate related to $Q$
by $y_Q$.

Let $k \in \{0,1,2, \dots\}$, $j \in \{0, \dots, k\}$ and
define the sub-lattice $\calD_{j,k} \subset \calD$
by setting
\begin{equation}\label{eq:SubLattice}
\calD_{j,k}=\{Q \in \calD \colon \ell(Q)=2^{m(k+1)+j} \text{ for some } m \in \Z\}.
\end{equation}
If $X$ is $\UMD$, $p \in (1, \infty)$ and $f \in L^p(X)$,
Theorem 6 in \cite{HH} implies that
\begin{equation}\label{eq:DecEst}
\int_{\R^d} \Big | \sum_{Q \in \calD_{j,k}} \Delta^l_Q f(x)  \Big |_X^p \ud x
\sim \E \int_{\R^d} \int_{\calV} \Big | \sum_{Q \in \calD_{j,k}} \varepsilon_Q 1_Q(x)\Delta^l_Q f(y_Q)  \Big |_X^p \ud \nu(y) \ud x
\end{equation}
for any $l \in \{0, 1, \ldots, k\}$.
The point of the subcollections $\calD_{j,k}$ is that now $\Delta^l_Q f$ is constant on
every $Q' \in \calD_{j,k}$ such that $Q' \subsetneq Q$. This is required by the abstract decoupling theorems.

\subsubsection*{Pythagoras' theorem}
A collection $\mathcal{S}$ of cubes in $\R^d$ is said to be \emph{$\eta$-sparse} (or just sparse), $\eta \in (0,1)$, if the following holds. For all $Q\in\mathcal{S}$ there exists a subset $E_Q\subset Q$ so that $|E_Q|>\eta|Q|$ and the sets $E_Q$ are mutually disjoint. The point of the following theorem is that
sparse collections are essentially as good as disjoint collections for some $L^p$ estimates.

Let $\calD$ be a dyadic lattice, $\calS \subset \calD$ be sparse and $X$ be a Banach space, and
suppose that for every $S \in \calS$ we have a function $f_S \colon \R^d \to X$ such that
$f_S$ is supported in $S$,  $\int f_S \ud x=0$ and $f_S$ is constant on those $S' \in \calS$ such that $S' \subsetneq S$.
Then, Lemma 4 in \cite{HH} -- Pythagoras' theorem for functions adapted to a sparse collection -- says that
\begin{equation}\label{eq:Pythagoras}
\Big \| \sum_{S \in \calS} f_S \Big \|_{L^p(X)}^p
\sim \sum_{S \in \calS} \| f_S \|_{L^p(X)}^p.
\end{equation}

\subsection{Multilinear operator-valued singular integrals}\label{sec:MultiSing}
We specify the class of operators that we study.
First, we define the operator-valued basic kernels.
Let $2 \le n \in \Z$ and let $X_1, \dots, X_n, Y_{n+1}$ be Banach spaces.
We denote by $\calL(X_1 \times \cdots \times X_n, Y_{n+1})$ the space of $n$-linear operators $B \colon X_1 \times \cdots \times X_n \to Y_{n+1}$ satisfying
$$
|B(x_1, \ldots, x_n)|_{Y_{n+1}} \le C \prod_{m=1}^n |x_m|_{X_m},
$$
and the best constant $C$ is denoted by $\|B\|_{X_1 \times \cdots \times X_n \to Y_{n+1}}$.
We will sometimes write $B$ acting on $(x_1,\dots, x_n)$ as above and sometimes like
$B[x_1, \dots, x_n]$.

Suppose $K$ is a function
$$
K \colon\R^{d(n+1)}\setminus \Delta \to \calL\Big(\prod_{m=1}^n X_m, Y_{n+1}\Big),
\quad \Delta=\{(x_1, \dots, x_{n+1}) \in\R^{d(n+1)} \colon x_1=\dots =x_{n+1}\},
$$
such that for all $e_m \in X_m$, $m \in \{1, \dots, n\}$,
the function
$$
x \mapsto K(x)[e_1, \dots, e_n] \in Y_{n+1}
$$
is strongly measurable.
Define the collection of $n$-linear operators
$$
\calC_{\size}(K)
=\Big\{ \Big(\sum_{m=2}^{n+1} |x_1-x_m|\Big)^{dn} K(x_1, \dots, x_{n+1}) \colon
(x_1, \dots, x_{n+1}) \in\R^{d(n+1)} \setminus \Delta \Big\}.
$$
For $\alpha \in (0,1]$ and $j \in \{1, \dots, n+1\}$ let $\calC_{\alpha,j}(K)$ be the collection of the operators
\begin{equation}\label{eq:Hol}
|x_j-x_j'|^{-\alpha}\Big(\sum_{m=2}^{n+1} |x_1-x_m|\Big)^{dn+\alpha}  (K(x)-K(x')),
\end{equation}
where $x=(x_1, \dots, x_{n+1}) \in\R^{d(n+1)} \setminus \Delta$ and
$x'=(x_1, \dots, x_{j-1},x_j',x_{j+1},\dots x_{n+1}) \in\R^{d(n+1)}$ satisfy
$$
|x_j-x_j'| \le 2^{-1} \max_{2 \le m \le n+1} |x_1-x_m|.
$$
We say that $K$ is an \emph{operator-valued $n$-linear basic kernel}
if there exists $\alpha \in (0,1]$ so that the families
$\calC_{\size}(K)$ and $\calC_{\alpha,m}(K)$, $m \in \{1, \dots, n+1\}$, are uniformly bounded. We also write
$\calC_{\CZ,\alpha}(K)=\calC_{\size}(K) \cap \bigcap_{m=1}^{n+1}\calC_{\alpha,m}(K)$.

If $X$ is a Banach space, we denote by $L^\infty_c(X)$ the functions in $L^\infty(X)$ with
compact support.
Let $K$ be an operator-valued basic kernel as above.
Let $T$ be an $n$-linear operator defined on tuples of functions $(f_1, \dots, f_n)$, where
$f_m \in L^\infty_c(X_m)$,
so that $T(f_1, \dots, f_n) \in L^1_{\loc}(Y_{n+1})$. We say that $T$ is an
\emph{$n$-linear operator-valued singular integral operator (SIO)} related to the kernel $K$ if
$$
\langle T(f_1, \dots, f_n), f_{n+1} \rangle=
\int_{\R^{d(n+1)}} \langle K(x_{n+1},x_1, \dots, x_n)[f_1(x_1), \dots, f_n(x_n)], f_{n+1}(x_{n+1}) \rangle \ud x
$$
whenever $f_m \in L^\infty_c (X_m)$,  $m=1, \dots, n+1$, are such that
$\supp f_i \cap \supp f_j = \emptyset$ for some $i \not= j$. Here we use the convention $X_{n+1} :=Y_{n+1}^*$ that is in force from this point on.
We are quite relaxed with the bracket notation $\langle \cdot, \cdot \rangle$ --
it means the natural duality pairing in each situation.

\subsection{$\BMO_p(X)$ and $T(1)$}\label{sec:BMOT1}
The representation theorem involves a certain $\BMO$ assumption related to ``$T1$''.
Since, as usual, $T1$ is not necessarily well defined as a function the $\BMO$ condition is stated in
terms of the pairings $\langle T1, h_Q \rangle$ (we recall below how to define these pairings). Therefore,
we define the $\BMO$ conditions for collections of elements of a Banach space.

Let $X$ be a Banach space and $\calD$ be a dyadic lattice. Suppose
$a=(a_Q)_{Q \in \calD} \subset X$ is a collection of elements of $X$ and let $\calD' \subset \calD$ be a finite subcollection.
Let $p \in (0,\infty)$. We define
$$
\|a\|_{\BMO_{\calD',p}(X)}
=\sup_{Q_0 \in \calD} \Big( \frac{1}{|Q_0|}
\E \Big\| \sum_{\substack{ Q \in \calD' \\ Q \subset Q_0}} \epsilon_Q a_Q \frac{1_Q}{|Q|^{1/2}} \Big\|_{L^p(X)}^p \Big)^{1/p},
$$
and then we define $\|a\|_{\BMO_{\calD,p}(X)}$ to be the supremum of $\|a\|_{\BMO_{\calD',p}(X)}$
over all finite subcollections $\calD' \subset \calD$.
Notice that if $\calD'$ and  $\calD''$ are two finite subcollections such that
$\calD' \subset \calD''$, then by Kahane's contraction principle \eqref{eq:KCont}
there holds that $\|a\|_{\BMO_{\calD',p}(X)} \le \|a\|_{\BMO_{\calD'',p}(X)}$.

The $X$-valued John-Nirenberg inequality for adapted sequences, Theorem 3.2.17 in \cite{HNVW1},
implies that
\begin{equation}\label{eq:JN}
\|a\|_{\BMO_{\calD,p}(X)} \sim \|a\|_{\BMO_{\calD,q}(X)}, \qquad 0 < p, q < \infty.
\end{equation}
Indeed, let $\calD' \subset \calD$ be finite.
Fix some $0<q<p< \infty$ such that $p \ge 1$.
Let $\{\varepsilon_Q \}_{Q \in \calD}$ be a collection of independent random signs on a probability space $\Omega$.
Let $Y= L^p(\Omega;X)$. From Theorem 3.2.17 in \cite{HNVW1} we deduce that
\begin{equation}\label{eq:bmowithp}
\sup_{k \in \Z} \sup_{\substack{Q_0 \in \calD \\ \ell(Q_0) \ge 2^{-k}}}  \Big ( \frac{1}{|Q_0|} \int_{Q_0}
\Big | \sum_{\substack{Q \in \calD' \\ Q \subset Q_0 \\ \ell(Q) \ge 2^{-k}}}
\varepsilon_Q a_Q \frac{1_Q}{|Q|^{1/2}} \Big|_Y^p \Big)^{1/p}
\end{equation}
is comparable to
\begin{equation}\label{eq:bmowithq}
\sup_{k \in \Z} \sup_{\substack{Q_0 \in \calD \\ \ell(Q_0) \ge 2^{-k}}}  \Big ( \frac{1}{|Q_0|} \int_{Q_0}
\Big | \sum_{\substack{Q \in \calD' \\ Q \subset Q_0 \\ \ell(Q) \ge 2^{-k}}}
\varepsilon_Q a_Q \frac{1_Q}{|Q|^{1/2}} \Big|_Y^q \Big)^{1/q}.
\end{equation}
In view of Kahane's contraction principle \eqref{eq:KCont} (it allows to remove the restriction $\ell(Q) \ge 2^{-k}$ inside the $Y$-norm)
we have that \eqref{eq:bmowithp}  is equal to $\|a\|_{\BMO_{\calD',p}(X)}$.
Likewise, \eqref{eq:bmowithq} is equal to
\begin{equation*}
 \sup_{Q_0 \in \calD }  \Big ( \frac{1}{|Q_0|} \int_{Q_0}
\Big | \sum_{\substack{Q \in \calD' \\ Q \subset Q_0}}
\varepsilon_Q a_Q \frac{1_Q}{|Q|^{1/2}} \Big|_Y^q \Big)^{1/q}
\sim \|a\|_{\BMO_{\calD',q}(X)},
\end{equation*}
where we applied the  Kahane-Khintchine inequality \eqref{eq:KK}.



Let $X_1, \dots, X_n$ and $Y_{n+1}$ be Banach spaces.
With respect to these spaces, suppose $T$ is an $n$-linear
SIO with a basic kernel $K$ as in Section \ref{sec:MultiSing}.
We turn to define the pairings $\langle T1, h_Q \rangle$ and other similar pairings.

Let $\Phi=(\phi_1, \dots, \phi_n)$ be an $n$-tuple of scalar-valued bounded functions.
 Assume that $Q \subset \R^d$ is a cube, and that
 $\varphi_Q \colon \R^d \to \C$ is a bounded function supported in
$Q$ with $\int \varphi_Q \ud x=0$.
Let $C \ge 2 \sqrt d$.
By $CQ$ we denote the cube with the same center as $Q$ and with side length $C\ell(Q)$.
If $e_m \in X_m$, $m \in \{1, \dots, n\}$, we define
\begin{equation}\label{eq:DefGlobalPairing}
\begin{split}
\big \langle T(\phi_1e_1, \dots, \phi_ne_n), \varphi_Q \big \rangle
&:=  \sum_{m=1}^n \big \langle T(\wt \phi_1^m e_1, \dots, \wt \phi^m_n e_n), \varphi_Q \big \rangle \\
&+\big \langle T(1_{CQ} \phi_1e_1, \dots, 1_{CQ} \phi_n e_n), \varphi_Q \big \rangle,
\end{split}
\end{equation}
where $\wt \phi^m_l= 1_{CQ}\phi_l$ for $1 \le l < m$, $\wt \phi^m_m= 1_{(CQ)^c} \phi_m$ and $ \wt \phi^m_l = \phi_l$ for $m<l \le n$,
and $\big \langle T(\wt \phi_1^m e_1, \dots, \wt \phi^m_n e_n), \varphi_Q \big \rangle \in Y_{n+1}$ is defined using the kernel $K$ by the formula
$$
\int_{\R^d} \int_{\R^{dn}} (K(x,y)-K(c_Q,y)) [\wt \phi_1^m(y_1)e_1, \dots, \wt \phi_n^m (y_n)e_n] \varphi_Q(x) \ud y \ud x.
$$
The uniform boundedness of the operators \eqref{eq:Hol} combined with $\supp  \wt \phi^m_m \subset (CQ)^c$
imply that this integral is absolutely convergent.
The definition of $\big \langle T(\phi_1e_1, \dots, \phi_ne_n), \varphi_Q \big \rangle $
is independent of the constant $C \ge 2 \sqrt d$.

Now, we define the $n$-linear operator
$
\langle T \Phi, \varphi_Q \rangle \colon X_1 \times \dots \times  X_n \to Y_{n+1}
$
by
\begin{equation}\label{eq:PairingOp}
\langle T \Phi, \varphi_Q \rangle [e_1, \dots, e_n]
:=\big \langle T(\phi_1e_1, \dots, \phi_ne_n), \varphi_Q \big \rangle.
\end{equation}
By $\langle T1, h_Q \rangle$ we mean the operator $\langle T\Phi, h_Q \rangle$ with
$\Phi=(1, \dots,1)$. The $\BMO$ condition related to the pairings $\langle T1, h_Q \rangle$
which appears in the representation theorem will be formulated in Definition \ref{defn:T1Cond}.

\subsection{Multilinear operator-valued shifts}
Let $2 \le n \in \Z$ and suppose $X_1, \dots, X_n,Y_{n+1}$ are Banach spaces.
Assume $k=(k_1, \dots, k_{n+1})$, $0 \le k_i \in \Z$.
Let $\calD$ be a dyadic lattice in $\R^d$.
An $n$-linear dyadic shift $S_{\calD}^k$ is an operator of the form
$$
S_{\calD}^k (f_1,\dots,f_n)
=\sum_{K \in \calD} \sum_{\substack{Q_1, \dots, Q_{n+1} \in \calD \\ Q_i^{(k_i)}=K}}
a_{K,(Q_i)}[ \langle f_1, \wt h_{Q_1} \rangle, \dots, \langle f_n, \wt h_{Q_n} \rangle] \wt h_{Q_{n+1}},
$$
where $f_m \in L^1_{\loc}(X_m)$ and $a_{K,(Q_i)}:=a_{K,Q_1,\dots, Q_{n+1}} \in \calL(\prod_{m=1}^n X_m,Y_{n+1})$.
In addition, we demand the following.
There exist two indices $j_0,j_1 \in \{1, \dots, n+1\}$, $j_0 \not =j_1$, so that
$\wt h_{Q_i}=h_{Q_i}$ if $i \in \{j_0,j_1\}$ and $\wt h_{Q_i}=h_{Q_i}^0$ if $i \not \in \{j_0,j_1\}$;
in other words there are two specified slots where the Haar functions are cancellative and in all the other slots they are
non-cancellative. One may think that only finitely many of the operators $a_{K,(Q_i)}$ are non-zero so that the shift is well defined for
locally integrable functions.
If $S^k_\calD$ is a shift as above, we denote by $\calC(S^k_\calD)$ the family of the normalized coefficient operators
\begin{equation}\label{eq:NormCoef}
\calC(S^k_\calD)=\Big \{ \frac{|K|^n}{\prod_{m=1}^{n+1}|Q_m|^{1/2}}a_{K,(Q_i)} \colon K , Q_1, \dots, Q_{n+1} \in \calD, Q_i^{(k_i)}=K\Big\}.
\end{equation}
In Section \ref{sec:ShiftBdd} we show the boundedness of shifts under certain conditions on $\calC(S^k_\calD)$ and the underlying Banach spaces.

\subsection{Multilinear operator-valued paraproducts}
Let $\calD$ be a dyadic lattice in $\R^d$.
Suppose $X_1, \dots, X_n,Y_{n+1}$ are Banach spaces and $a_Q \in \calL(\prod_{m=1}^n X_m ,Y_{n+1})$, $Q \in \calD$, are given.
Let $a = (a_Q)_{Q \in \calD}$.
An operator-valued $n$-linear paraproduct is an operator of the form
$$
\pi_{\calD, a}(f_1, \dots, f_n) :=\sum_{Q \in \calD} a_Q [\langle f_1 \rangle_Q, \dots, \langle f_n \rangle_Q ] h_Q,
$$
where $f_m \in L^1_{\loc}(X_m)$.
As with dyadic shifts, this is well defined for example if only finitely many of the operators $a_Q$
are non-zero. In Section \ref{sec:ParaBdd} we consider the boundedness of paraproducts.

\subsection{Bounding dyadic operators by sparse operators}
The boundedness of shifts and paraproducts will be considered in Sections \ref{sec:ShiftBdd} and \ref{sec:ParaBdd}.
 Here, we formulate an analogue of the sparse domination results of \cite{CDPO,CDPOMRL,LMOV} in multilinear, operator-valued setting of Theorem \ref{cor:main} below. The proof follows exactly the outline of the multilinear version of \cite{LMOV}. We refer to the above references for the, by now standard, definitions and generalities on sparse collections and forms.

\begin{lem}\label{lem:SparseModel}
Let $1 \le n \in \Z$. Let $X_1, \dots, X_n$ and $Y_{n+1}$ be Banach spaces and $X_{n+1}:= Y_{n+1}^*$.
Suppose we have functions $f_m \in L^\infty_c( X_{m})$, $m=1,\dots, n+1$.
Let $\calD$ be a dyadic grid and $\eta \in (0,1)$.
Then there exists an $\eta$-sparse collection $\mathcal{S}=\mathcal{S}((f_m), \eta)\subset \calD$ so that the following holds.

Let $k=(k_1, \dots, k_{n+1})$, $0 \le k_i \in \Z$,
and assume $p_1, \dots, p_{n+1} \in (1, \infty)$ are such that $\sum_{m=1}^{n+1} 1/p_m = 1$.
Suppose that we have operators
$a_{K,Q_1,\dots, Q_{n+1}} \in \calL(\prod_{m=1}^n X_m,Y_{n+1})$,
where $K , Q_1, \dots, Q_{n+1} \in \calD$ and  $Q_i^{(k_i)}=K$,  such that
$a_{K,(Q_i)}:=a_{K,Q_1,\dots, Q_{n+1}}$ satisfies
$$
|\langle a_{K,(Q_i)}[e_1, \ldots, e_{n}], e_{n+1} \rangle| \le A_1 \frac{\prod_{m=1}^{n+1}|Q_m|^{1/2}}{|K|^n} \prod_{m=1}^{n+1} |e_m|_{X_m}.
$$
Suppose further that for some scalar-valued functions $u_{m,Q} = \sum_{Q' \in \ch(Q)} c_{m,Q'} 1_{Q'}$
satisfying $|u_{m,Q}| \le |Q|^{-1/2}$
the operators
$$
U_{\calD'}(g_1, \ldots, g_n)
:= \sum_{K \in \calD'} \sum_{\substack{Q_1, \dots, Q_{n+1} \in \calD \\ Q_i^{(k_i)}=K}}
a_{K,(Q_i)}[ \langle g_1, u_{1,Q} \rangle, \dots, \langle g_n, u_{n,Q} \rangle] u_{n+1,Q}, \qquad \calD' \subset \calD,
$$
satisfy
$$
|\langle U_{\calD'}(g_1, \ldots, g_n), g_{n+1}\rangle| \le A_2 \prod_{m=1}^{n+1} \|g_m\|_{L^{p_i}(X_m)}, \qquad g_1, \ldots, g_{n+1} \in L^{p_i}(X_m).
$$
Then we have
\begin{equation}\label{SparseForDMOs}
|\langle U_{\calD}(f_1, \ldots, f_n), f_{n+1}\rangle|
\lesssim_{\eta} (A_1+A_1\kappa + A_2) \sum_{Q\in\mathcal{S}}|Q|\prod_{m=1}^{n+1}\bla |f_m|_{X_m} \bra_Q,
\end{equation}
where $\kappa = \max k_m$.
\end{lem}

We conclude this preliminary section by  recalling the well-known fact  that (weighted) boundedness in the full range of expected exponents  follows from a sparse estimate of the type stated in Lemma \ref{lem:SparseModel}. A proof of this exact statement is given in \cite{DLMV1}, and further consequences in the weighted setting are  formulated in \cite{CDPOBP,LMOV} and references therein.
Let $X_1, \dots, X_n$ and $Y_{n+1}$ be Banach spaces.
If $T$ is an operator such that
for all tuples  $f_m \in L^\infty_c(X_m)$,  there exists a dyadic lattice $\calD=\calD((f_m))$
and a sparse collection $\calS=\calS((f_m)) \subset \calD$
so that
$$
| \langle T(f_1, \dots, f_n),f_{n+1} \rangle |
\lesssim  \sum_{Q\in\mathcal{S}}|Q|\prod_{m=1}^{n+1}\bla |f_m|_{X_m} \bra_Q,
$$
then \begin{equation}\label{eq:consequenceofsparse}
\| T(f_1, \dots, f_n) \|_{L^{q_{n+1}}(Y_{n+1})}
\lesssim \prod_{m=1}^n \| f_m \|_{L^{p_m}(X_m)},
\end{equation}
where $p_m \in (1, \infty]$ are such that $1/{q_{n+1}}:= \sum_{m=1}^n 1/p_m>0$.

\section{Random{ize}d boundedness and the $\RMF$ property}\label{sec:RbddandRMF}

\subsection{Random{ize}d boundedness}
In this section we discuss  randomized boundedness conditions  for families of multilinear operators.
First, we recall the well-known concept of $\calR$-boundedness of linear operators.
If $X_1$ and $Y_2$ are Banach spaces and $\calT \subset \calL(X_1,Y_2)$, we say that $\calT$ is $\calR$-bounded if
there exists a constant $C$ such that for all integers $l \ge 1$, all $T_k \in \calT$
and for all $e_{k,1} \in X_1$, $e_{k,2} \in X_2 := Y_2^*$, $k = 1, \ldots, l$, the inequality
$$
\Big| \sum_{k=1}^l  \big\langle  T_k e_{k,1}, e_{k,2} \big\rangle \Big| \le C \|(e_{k,1})\|_{\Rad(X_1)} \|(e_{k,2})\|_{\Rad(X_2)}
$$
holds. The smallest constant $C$ is denoted by $\calR(\calT)$, and called the $\calR$-boundedness constant of $\calT$.
If $\calT$ is not $\calR$-bounded we set $\calR(\calT)=\infty$.
\begin{rem}\label{rem:differentR}
Suppose that $Y_2$ has non-trivial type (e.g. $Y_2$ is a UMD space).
Let $\wt \calR(\calT)$ denote the best constant $C$ such that
\begin{equation}\label{eq:StandRBdd}
\Big( \E \Big| \sum_{k=1}^l \varepsilon_k T_ke_k \Big |_{Y_2}^2  \Big)^{1/2}
\le C \Big ( \E \Big| \sum_{k=1}^l \varepsilon_k e_k \Big |_{X_1}^2 \Big)^{1/2}
\end{equation}
holds for all $e_{k} \in X_1$. Then we have by Lemma \ref{lem:randomduality} that
$$
\wt \calR(\calT) \lesssim \calR(\calT) \le \wt \calR(\calT).
$$
In fact, \eqref{eq:StandRBdd} is the most commonly appearing standard definition of $\calR$-boundedness.
\end{rem}

For a positive integer $n$, we write $\mathcal J_n$ for the discrete interval $\{1,\ldots, n+1\}$.
Throughout this section, let $X_1, \dots, X_n, X_{n+1}$ be reflexive Banach spaces and denote $Y_{j}=X_{j}^*$.
Below, we    customarily enumerate $\mathcal J\subset \mathcal J_n$ increasingly so that $1 \le j_1 < \dots < j_\ell \le n+1$.
 We use the tuple  notation
\[
(e_j)_{j \in \calJ}:= (e_{j_1}, \dots, e_{j_\ell}) \in X_{\mathcal J}\coloneqq \prod_{j\in \mathcal J} X_{j}.\]

The following discussion pertains to the case $n\geq 3$. We set some notation for the multilinear $\calR$-boundedness condition associated to an $(n+1)$-linear
 contraction
\begin{equation}\label{eq:Contr}
\varpi \colon X_{\mathcal J_n}  \to \C, \qquad
| \varpi(e_1, \dots, e_{n+1})| \le \prod_{m=1}^{n+1} |e_m|_{X_m}, \quad e_m \in X_m.
\end{equation}
 This condition will involve  suitable partitions of the set of indices $\mathcal J_n$.
 We say that $\mathcal P=( \{j_{\mathcal P}\}, \mathcal P_{\Rad}, \mathcal P_{\RM})$ is an admissible partition of $
\mathcal J_n$ if $\{j_{\mathcal P}\}, \mathcal P_{\Rad}, \mathcal P_{\RM}\subset \mathcal J_n$ are pairwise disjoint, their union is $\calJ_n$
and  $\#\mathcal P_{\RM}\leq n- 2$.
Let now  $  \mathcal J \subset \mathcal J_n $ with  $1\leq \# \mathcal J \leq n-2$ and  $v \in \mathcal J_n\setminus \mathcal J$.
For a set $A \subset X_{\mathcal J}$, define
\begin{equation}\label{eq:DefSupu}
\begin{split}
\| A \|_{\RM_v(\varpi, \calJ)}
&= \sup  \Big| \sum_{k=1}^K  \varpi (e_{1,k}, \ldots, e_v,\ldots, e_{n+1,k}) \Big|,
\end{split}
\end{equation}
where the supremum is taken over all $K \in \mathbb N$ and over all choices of
\begin{itemize}
\item tuples $(e_{j,k})_{j \in \calJ} \in A, 1\leq k \leq K$,
\item elements $e_v \in X_v$ with  $|e_v|_{X_v}\leq 1$,
\item sequences $ ( e_{j,k})_{k=1}^K \subset X_j$ with $ \| ( e_{j,k})_{k=1}^K  \|_{\Rad(X_j)} \le 1$,
where \[j \in \mathcal J_{\Rad}\coloneqq \{1,\ldots,n+1\}\setminus (\mathcal J
\cup\{v\}) .\]
\end{itemize}
Here $(e_{j,k})_{k=1}^K$ denotes the sequence $e_{j,1}, \dots, e_{j,K}$ of elements of $X_j$,
and this should not be confused with the notation $(e_{j,k})_{j \in \calJ} \in X_\calJ$ meaning a tuple of elements; later, this distinction should be clear from the context.
Let also $\| A \|_{\RM'_v( \varpi, \calJ )}$ be defined just as
$\| A \|_{\RM_v( \varpi, \calJ)}$ in \eqref{eq:DefSupu}, except that in addition there is the requirement that the tuples $(e_{j,k})_{j \in \calJ} \in A$
satisfy $(e_{j,k})_{j \in \calJ} \not= (e_{j,k'})_{j \in \calJ}$
if $k \not = k'$.

\begin{lem}\label{lem:SupDistinct}
There holds that $\| A \|_{\RM_v'(\varpi,\calJ )} = \| A \|_{\RM_v(\varpi,\calJ)}$, that is, when testing the constant $\| A \|_{\RM_v(\varpi, \calJ)}$
it is enough to consider a sequence of distinct elements of $A$.
\end{lem}

\begin{proof}
Let  $\mathcal P$ be the admissible partition with $j_{\mathcal P}=v$, $\mathcal P_{\RM}=\mathcal J$.  Fix elements $e_{j,k} \in X_j$, where $k \in \{1, \dots, K\}$, as in the definition \eqref{eq:DefSupu}.
Write $\{1, \dots, K\}$ as a disjoint union $\bigcup_{m=1}^M \calK_m$, so that
$(e_{j,k})_{j \in \calJ}=(e_{j,k'})_{j \in \calJ}$ if $k$ and $k'$ belong  to the same $\calK_m$, and
$(e_{j,k})_{j \in \calJ} \not = (e_{j,k'})_{j \in \calJ}$ if $k$ and $k'$ belong to different sets $\calK_m$.
If $j \in \calJ$ and $k \in \calK_m$, we denote $e_{j,k}=:f_{j,m}$.
Write $\mathcal P_{\Rad}=\{i_1, \dots, i_u\}$, where $i_j < i_{j+1}$. To ease the notation, for $m \in \{1, \dots, M\}$ let $\Lambda_m$ be the
$u$-linear form obtained from $\varpi$ by keeping the elements $f_{j,m}$, $j \in \calJ$,  and $e_{v}$ fixed.
We have
$$
\sum_{k=1}^K \varpi(e_{1,k}, \dots, e_{n+1,k})
= \sum_{m=1}^M \sum_{k \in \calK_m} \Lambda_m(e_{i_1,k}, \dots, e_{i_u,k}).
$$

Fix one $m$ for the moment. Let $\{\varepsilon^j_k\}_{k=1}^K$, $j \in \{1, \dots, u-1\}$, be  collections of  independent random signs.
We denote the expectation with respect to the random variables $\varepsilon^j_k$ by $\E^j$, and write $\E=\E^1 \cdots \E^{u-1}$.
Then we have the identity
\begin{equation}\label{eq:MultiRS}
\begin{split}
\sum_{k \in \calK_m} & \Lambda_m(e_{i_1,k}, \dots, e_{i_u,k}) \\
&= \E \sum_{k_1, \dots, k_u \in \calK_m} \varepsilon^1_{k_1}\varepsilon^1_{k_2}
\varepsilon^2_{k_2}\varepsilon^2_{k_3} \cdots \varepsilon_{k_{u-1}}^{u-1}\varepsilon_{k_{u}}^{u-1}
\Lambda_m(e_{i_1,k_1}, \dots, e_{i_u, k_u}) \\
& =\E \Lambda_m \Big( \sum_{k_1 \in \calK_m} \varepsilon^1_{k_1} e_{i_1,k_1}, \sum_{k_2 \in \calK_m} \varepsilon^1_{k_2}\varepsilon^2_{k_2} e_{i_2,k_2},
\dots, \sum_{k_u \in \calK_m} \varepsilon^{u-1}_{k_u}e_{i_u,k_u} \Big).
\end{split}
\end{equation}

Now, we combine the last two equations and use the definition of $\| A \|_{\RM_v'(\varpi,\calJ)}$.
This gives that
\begin{equation*}
\begin{split}
\Big | & \sum_{k=1}^K \varpi(e_{1,k}, \dots, e_{n+1,k}) \Big | \\
&\le \| A \|_{\RM_v'(\varpi,\calJ)}
\E \Big \| \Big( \sum_{k \in \calK_m} \varepsilon^1_k e_{i_1,k} \Big)_{m=1}^M \Big \|_{\Rad(X_{i_1})}
\prod_{j=2}^{u-1} \Big \| \Big( \sum_{k \in \calK_m} \varepsilon^{j-1}_k \varepsilon^j_k  e_{i_j,k} \Big)_{m=1}^M \Big \|_{\Rad(X_{i_j})} \\
& \hspace{2,4cm} \times \Big \| \Big( \sum_{k \in \calK_m} \varepsilon^{u-1}_k e_{i_u,k} \Big)_{m=1}^M \Big \|_{\Rad(X_{i_u})},
\end{split}
\end{equation*}
where the expectation is less than
\begin{equation*}
\begin{split}
\Big(\E^1 \Big \| \Big( \sum_{k \in \calK_m} \varepsilon^1_k e_{i_1,k} \Big)_{m=1}^M \Big \|_{\Rad(X_{i_1})}^2 \Big)^{1/2}
&\prod_{j=2}^{u-1} \Big( \E^{j-1} \E^j
\Big \| \Big( \sum_{k \in \calK_m} \varepsilon^{j-1}_k \varepsilon^j_k  e_{i_j,k} \Big)_{m=1}^M \Big \|_{\Rad(X_{i_j})}^2 \Big)^{1/2} \\
& \hspace{0.25cm} \times \Big(\E^{u-1} \Big \| \Big( \sum_{k \in \calK_m} \varepsilon^{u-1}_k e_{i_u,k} \Big)_{m=1}^M \Big \|_{\Rad(X_{i_u})}^2 \Big)^{1/2}.
\end{split}
\end{equation*}
Denote the expectation and the random signs related to the $\Rad$-norms by $\wt {\E}$ and $\{\wt {\varepsilon}_m \}_{m=1}^M$.
We have
\begin{equation*}
\begin{split}
\E^{j-1} \E^j \Big \| \Big( \sum_{k \in \calK_m} \varepsilon^{j-1}_k \varepsilon^j_k  e_{i_j,k} \Big)_{m=1}^M \Big \|_{\Rad(X_{i_j})}^2
&=  \tilde{\E} \E^{j-1}\E^j \Big \| \sum_{m=1}^M \sum_{k \in \calK_m} \wt{\varepsilon}_m \varepsilon^{j-1}_k \varepsilon^j_k  e_{i_j,k} \Big \|_{X_{i_j}}^2\\
&= \| (e_{i_j,k})_{k=1}^K \|_{\Rad(X_{i_j})}^2 \le 1.
\end{split}
\end{equation*}
The remaining last two terms satisfy the corresponding estimate. Thus, we have finished the proof of
$\| A \|_{\RM_v(\varpi, \calJ)} \le \| A \|_{\RM'_v(\varpi, \calJ)}$. As the reverse estimate is immediate, the conclusion of the lemma follows.
\end{proof}

\begin{rem}\label{rem:LengthK}
We record an observation based on  Lemma \ref{lem:SupDistinct},    which will later be used without explicit mention.
Let again $\calJ \subset \calJ_n$, $1 \le \# \calJ \le n-2$ and $v\in \mathcal J_n\setminus \mathcal J$. Suppose that  $K \in \N$ and we have elements
$e_{j,k} \in X_j$, $j \in \calJ$, $k \in \{1, \dots, K\}$. Then
\begin{equation}\label{eq:LengthK}
\| \{(e_{j,k})_{j \in \calJ}\}_{k=1}^K \|_{\RM_v(\varpi, \calJ)}
= \sup \Big | \sum_{k=1}^K \varpi(e_{1,k}, \ldots, e_v,\ldots, e_{n+1,k}) \Big|,
\end{equation}
where the supremum is taken  over elements
$e_{j,k} \in X_j$, $j \in \calJ_n \setminus(\calJ \cup \{v\})$, $k \in \{1, \dots, K\}$, such that $\|(e_{j,k})_{k=1}^K\|_{\Rad(X_j)} \le 1$
and  over $e_v \in X_v$ with $|e_{v}|_{X_{v}} \le 1$.
Indeed, ``$\ge$'' is clear just by definition. On the other hand,
the right hand side of \eqref{eq:LengthK}  clearly satisfies $RHS\eqref{eq:LengthK} \ge   \| \{(e_{j,k})_{j \in \calJ}\}_{k=1}^K \|_{\RM'_v(\varpi, \calJ)}$.
\end{rem}

\begin{rem}
In the same setup as in the previous remark,
assume $\calJ=\calJ_0 \cup \calJ_1$, where $\calJ_1  \not = \varnothing$ and $\calJ_0 \cap \calJ_1 = \varnothing$.
Suppose we have elements $e_{j,k} \in X_j$, $j \in \calJ$, $k \in \{1, \dots, K\}$.
Then
\begin{equation}\label{eq:SupLel2}
\| \{ (e_{j,k})_{j \in \calJ} \}_{k=1}^K \|_{\RM_v(\varpi, \calJ)}
\le \| \{ (e_{j,k})_{j \in \calJ_0} \}_{k=1}^K \|_{\RM_v(\varpi, \calJ_0)}
\prod_{j \in \calJ_1} \| (e_{j,k})_{k=1}^K \|_{\Rad(X_j)}.
\end{equation}

To see this, we use Remark \ref{rem:LengthK}.
Let $\mathcal P=\{ \{v\}, \mathcal P_{\Rad}, \calJ\}$ be the corresponding admissible partition  and $e_{j,k} \in X_j$ for $j \in \mathcal P_{\Rad} \cup \{v\}$,
$k=1, \ldots, K$, and
assume that $e_{v,k}=e_{v,k'} =: e_{v}$.

Assume first that $\calJ_0 \not=\varnothing$.
Then, by definition, we have
\begin{equation*}
\begin{split}
\Big | & \sum_{k=1}^K \varpi(e_{1,k}, \dots, e_{n+1,k}) \Big | \\
&\le \| \{ (e_{j,k})_{j \in \calJ_0}\}_{k=1}^K \|_{\RM_v(\varpi,\calJ_0)}
\prod_{j \in \calJ_1 \cup \mathcal P_{\Rad}} \| (e_{j,k})_{k=1}^K \|_{\Rad(X_j)} |e_{v}|_{X_{v}},
\end{split}
\end{equation*}
which proves the claim.

On the contrary, if $\calJ_0 =\varnothing$,
then
using random signs as in \eqref{eq:MultiRS} and boundedness of $\varpi$ (Equation \eqref{eq:Contr}) we get that
\begin{equation*}
\Big |  \sum_{k=1}^K \varpi(e_{1,k}, \dots, e_{n+1,k}) \Big |
\le \prod_{j \in \calJ \cup \mathcal P _{\Rad}} \| (e_{j,k})_{k=1}^K \|_{\Rad(X_j)} |e_v|_{X_{v}},
\end{equation*}
which gives the claim.
\end{rem}

\begin{exmp}\label{ex:RMConst}
Let $(X_1, \ldots, X_{n+1})$ be a tuple of reflexive Banach spaces and $\varpi_0 \colon X_{\mathcal J_n} \to \C$ be as in \eqref{eq:Contr}. Let $(\Omega, \mu)$ be a measure space and associate the tuple
$$
(L^{p_1}(\Omega; X_1), \ldots, L^{p_{n+1}}(\Omega; X_{n+1})), \qquad p_m \in (1,\infty), \,\, \sum_{m=1}^{n+1} 1/p_m = 1,
$$
with the $(n+1)$-linear mapping $\varpi \colon \prod_{m=1}^{n+1} L^{p_m}(\Omega; X_m)  \to \C$,
\begin{equation}\label{eq:LpXContr}
\varpi(f_1, \ldots, f_{n+1}) := \int_{\Omega} \varpi_0(f_1(\omega), \ldots, f_{n+1}(\omega))\ud \mu(\omega).
\end{equation}
We obviously have
$$
|\varpi(f_1, \ldots, f_{n+1})| \le  \prod_{m=1}^{n+1} \|f_m\|_{L^{p_m}(\Omega; X_m)}.
$$

Suppose $\calJ \subset\calJ_n$, $1 \le \#\calJ \le n-2$ and $v\in \calJ_n\setminus \calJ$.
It is not hard to see that
$$
\|\{(f_{j,k})_{j \in \calJ}\}_{k=1}^\infty \|_{\RM_v(\varpi, \calJ)} \lesssim
\big\| \omega \mapsto \| \{ (f_{j,k}(\omega))_{j \in \calJ} \}_{k=1}^\infty \|_{\RM_v(\varpi_0, \calJ)} \big\|_{L^{p(\calJ)}(\Omega)},
$$
where
$$
1/p(\calJ) = \sum_{j \in \calJ} 1/p_j.
$$
We now demonstrate that the corresponding lower bound also holds.
We will show that
\begin{equation}\label{eq:Fixedu}
\| \{ (f_{j,k})_{j \in \calJ} \}_{k=1}^K \|_{\RM_v(\varpi, \calJ)}
\gtrsim \big\| \| \{ (f_{j,k}(\omega))_{j \in \calJ} \}_{k=1}^K \|_{\RM_v(\varpi_0, \calJ)} \big\|_{L^{p(\calJ)}(\Omega)}
\end{equation}
 and the claim follows by monotone convergence.

Write $\calI= \mathcal J_n \setminus \calJ$.
Using Remark \ref{rem:LengthK}, for $\omega \in \Omega$
let $\varphi_{i,k}(\omega) \in X_i$, $i \in \calI $, $k \in \{1, \dots, K\}$,
be such that $\| (\varphi_{i,k}(\omega))_{k=1}^K \|_{\Rad(X_i)} \le 1$ for $i \not =v $ and $\varphi_{v,k}(\omega)=\varphi_{v,k'}(\omega)
=: \varphi_{v}(\omega) $ satisfies $|\varphi_{v}(\omega)|_{X_v}  \le 1$.
Furthermore, let the elements $\varphi_{i,k}(\omega)$ be such that
$\varpi_0$ acting on $f_{j,k}(\omega)$ and $\varphi_{i,k}(\omega)$ is non-negative for all $k$
and such that the sum over $k$ of these is  $\gtrsim  \| \{ (f_{j,k}(\omega))_{j \in \calJ} \}_{k=1}^K \|_{\RM_v(\varpi_0, \calJ)}$. For  $i \in \calI$ write
$$
B_i:= \big\| \| \{ (f_{j,k}(\omega))_{j \in \calJ} \}_{k=1}^K \|_{\RM_v(\varpi_0, \calJ)} \big\|_{L^{p(\calJ)}(\Omega)}^{p(\calJ)/p_i},
$$
define
$$
f_{i,k}(\omega):= B_i^{-1}\varphi_{i,k}(\omega)  \| \{ (f_{j,k}(\omega))_{j \in \calJ} \}_{k=1}^K \|_{\RM_v(\varpi_0, \calJ)}^{p(\calJ)/p_i}
$$
and write $f_v=f_{v,k}$.
For $i \in \calI \setminus \{v\}$ there holds that
$$
\| (f_{i,k})_{k=1}^K \|_{\Rad(L^{p_i}(\Omega;X))} \sim \big \| \| (f_{i,k}(\omega))_{k=1}^K \|_{\Rad(X_i)} \big \|_{L^{p_i}(\Omega)} \le1
$$
and $\| f_{v}\|_{L^{p_v}(\Omega;X_v)}\le1$.   Define the exponent $p(\calI)$ by
$
1/p(\calI)=\sum_{i \in \calI} 1/p_i.
$
We also have that
\begin{equation*}
\begin{split}
\Big | \sum_{k=1}^K \varpi (f_{1,k}, \dots, f_{n+1,k}) \Big |
&= \int_{\Omega} \sum_{k=1}^K \varpi_0(f_{1,k}(\omega), \dots, f_{n+1,k}(\omega)) \\
&\gtrsim \prod_{i \in \calI}B_i^{-1} \int_{\Omega} \| \{ (f_{j,k}(\omega))_{j \in \calJ} \}_{k=1}^K \|_{\RM_v(\varpi_0, \calJ)}^{1+p(\calJ)/p(\calI)} \\
&=\big\| \| \{ (f_{j,k}(\omega))_{j \in \calJ} \}_{k=1}^K \|_{\RM_v(\varpi_0, \calJ)} \big\|_{L^{p(\calJ)}(\Omega)}.
\end{split}
\end{equation*}
This proves \eqref{eq:Fixedu} concluding our demonstration.

In the special case that $X_1= \cdots =X_{n+1}=\C$ and $\varpi_0(e_1, \dots, e_{n+1})=\prod_{m=1}^{n+1}e_m$
it is not hard to see that
$$
\| \{(e_{j,k})_{j \in \calJ}\}_{k=1}^\infty\|_{\RM_v(\varpi_0, \calJ)}
= \sup_{k} \prod_{j \in \calJ} |e_{j,k}|.
$$
Therefore, the above gives in this case that
\begin{equation}
\label{e:latticesup}
\|\{(f_{j,k})_{j \in \calJ}\}_{k=1}^\infty \|_{\RM_v(\varpi, \calJ)}
\sim \Big\| \sup_k \prod_{j \in \calJ} |f_{j,k}(\omega)|  \Big\|_{L^{p(\calJ)}(\Omega)}.
\end{equation}
\end{exmp}

Next, we define a related multilinear $\calR_{\varpi}$-boundedness condition for families of operators. Due to the invariance under permutation of the spaces $X_j$, $j\in \mathcal J_n$ of the previous and upcoming definitions, we do not lose in generality by working with $n$-linear operators on $\prod_{j=1}^n X_j$ with range in $Y_{n+1}$. Also, it will be convenient to define the notion of \emph{tight admissible partition}  $\mathcal P  $: an   an admissible partition $\mathcal P$ of $\mathcal J_n$ is tight if   $\#\mathcal P_{\Rad}=2$ .
\begin{defn}\label{def:Rvarpi}
Let $n \ge 2$ and suppose that $\calT \subset \calL(\prod_{m=1}^n X_m,Y_{n+1})$ is a family of operators.
Assume that $\varpi$ is an $(n+1)$-linear mapping satisfying \eqref{eq:Contr} and $\mathcal P$ is a tight admissible partition. We say that $\calT$ is $\calR_{\varpi,\mathcal P}$-bounded if there exists a finite constant $C$ so that
\begin{equation*}
 \begin{split}
 \Big | & \sum_{k=1}^K \langle T_k [e_{1,k} \dots, e_{n,k}],e_{n+1,k} \rangle \Big|  \\
& \le C  \| \{(e_{j,k})_{j \in  \mathcal P_{\RM}}\}_{k=1}^K \|_{\RM_{j_{\mathcal P}}(\varpi,  \mathcal P_{\RM})}
\prod_{j \in \mathcal P_{\Rad}}\|(e_{j,k})_{k=1}^K \|_{\Rad(X_j)}
|e_{j_{\mathcal P}}|_{X_{j_{\mathcal P}}}
\end{split}
\end{equation*}
holds for
all
$K \in \N$, all choices of $T_k \in \calT$,  $\{e_{j,k}\colon k  = 1, \ldots, K\} \subset X_j$, $j \in \mathcal J_n$,
such that  $e_{j_{\mathcal P},k}=e_{j_{\mathcal P},k'}:=e_{j_{\mathcal P}}$ for all $k$ and $k'$.
The smallest possible constant $C$ is denoted by $\calR_{\varpi,\mathcal P} (\calT)$. If  $\calT$ is $\calR_{\varpi,\mathcal P}$-bounded for all tight admissible partitions $\mathcal P$, we say that  $\calT$ is   $\calR_\varpi$-bounded and write $\calR_{\varpi} (\calT)$ for the supremum over all tight admissible  partitions $\mathcal P$ of $\calR_{\varpi,\mathcal P} (\calT)$. Otherwise,  we set
$\calR_{\varpi} (\calT)=\infty$.
\end{defn}

Notice that if $n=2$, then the $\calR_\varpi$-boundedness condition does not depend on $\varpi$ at all, and it reduces to
a more simple estimate as in Equation \eqref{e:rboundintro}.
Therefore, in the case $n=2$ we will just talk about $\calR$-boundedness.
\begin{rem}\label{rem:OldRBdd}
Suppose $\calT \subset \calL(\prod_{m=1}^n X_m,Y_{n+1})$ is an $\calR_{\varpi}$-bounded family.
Suppose $\mathcal P$ is a non-tight admissible partition.
Let $T_k \in \calT$ and $e_{j,k} \in X_j$ for $k=1, \dots, K$, and as usual assume that $e_{j_{\mathcal P}}:=e_{j_{\mathcal P},k}=e_{j_{\mathcal P},k'}$.
Write $\mathcal P_{\Rad}=\calJ_0 \cup \calJ_1$, where $\# \calJ_1 =2$.
Then
\begin{equation}\label{eq:OldRBdd}
\begin{split}
&\quad \Big |   \sum_{k=1}^K \langle T_k [e_{1,k}, \dots, e_{n,k}], e_{n+1,k} \rangle \Big | \\
& \le \calR_{\varpi}(\calT) \| \{ (e_{j,k})_{j \in \mathcal P_{\RM}  \cup \calJ_0}\}_{k =1}^K \|_{\RM_{j_\calP}(\varpi, \mathcal P_{\RM} \cup \calJ_0)}
\prod_{j \in \calJ_1} \| (e_{j,k})_{k=1}^K \|_{\Rad(X_j)} | e_{j_{\mathcal P}}|_{X_{j_{\mathcal P}}} \\
& \le \calR_{\varpi}(\calT) \| \{ (e_{j,k})_{j \in \mathcal P_{\RM} }\}_{k =1}^K \|_{\RM_{j_\calP}(\varpi, \mathcal P_{\RM}  )}
\prod_{j \in \calP_{\Rad}} \| (e_{j,k})_{k=1}^K \|_{\Rad(X_j)} | e_{j_{\mathcal P}}|_{X_{j_{\mathcal P}}},
\end{split}
\end{equation}
where in the last step we applied Equation \eqref{eq:SupLel2}. We will apply this form of the $\calR_\varpi$-boundedness (where not necessarily $\mathcal P$ is a tight partition)
later.
\end{rem}

In the representation theorem we will need the fact that $\calR_{\varpi}$-boundedness is preserved under averages in the sense of the following lemma.

\begin{lem}\label{lem:AveRbdd}
Let $\calT \subset \calL(\prod_{m=1}^n X_m, Y_{n+1})$ be an $\calR_{\varpi}$-bounded family of operators.
Let $\calA(\calT) \subset \calL(\prod_{m=1}^n X_m, Y_{n+1})$ be the collection of operators of the form
$$
\int_{\R^{d(n+1)}} L(y) \lambda(y) \ud y,
$$
where $L \colon \R^{d(n+1)} \to \calL(\prod_{m=1}^n X_m, Y_{n+1})$ is such that $L(y) \in \calT$ for a.e. $y$ and
$\lambda \colon \R^{d(n+1)} \to \C$ satisfies $\int |\lambda | \le 1$. Then we have
$$
\calR_{\varpi} (\calA(\calT)) \lesssim \calR_{\varpi} (\calT).
$$
\end{lem}
\begin{rem}
The space $\R^{d(n+1)}$ plays no role here (it could be some measure space), but this is what appears later. Moreover, we have by definition that
$$
\Big(\int_{\R^{d(n+1)}} L(y) \lambda(y) \ud y\Big)[e_1, \ldots, e_n] := \int_{\R^{d(n+1)}} L(y)[e_1, \ldots, e_n] \lambda(y) \ud y,
$$
and the assumption on $L$ is that the mappings
$$
y \mapsto L(y)[e_1, \ldots, e_n]
$$
are strongly measurable.
\end{rem}

\begin{proof}[Proof of Lemma \ref{lem:AveRbdd}]
This result could be reduced to the corresponding linear result. For the linear theorem, see for example Theorem 8.5.2 in \cite{HNVW2}.
We give another self contained proof.

Suppose $T_k= \int L_k \lambda_k  \in \calA(\calT)$ for $k=1, \dots, K$. We may assume that
$L_k(y) \in \calT$ for every $y \in \R^{d(n+1)}$ and that $ \int |\lambda_k| =1$ for every $k$.
For each $k$ define the probability space $\calY_k:=(\R^{d(n+1)}, \mu_k)$, where $\mu_k:= |\lambda_k| \ud y$.
Let $(\calY,\mu)$ be the product probability space $\Big(\prod_{k=1}^K\calY_k, \prod_{k=1}^K\mu_k\Big)$.

Fix a tight admissible partition $\mathcal P$, and let $\{e_{j,k}\}_{k=1}^K \subset X_j$ be such that
$e_{j_{\mathcal P }}:= e_{j_{\mathcal P },k}=e_{j_{\mathcal P },k'}$.
Now, we have that
\begin{equation*}
\begin{split}
\sum_{k=1}^K  \langle T_k[e_{1,k}, \dots, e_{n,k}],e_{n+1,k} \rangle
&=\sum_{k=1}^K \int_{\calY_k}
\Big\langle \frac{\lambda_k(y)}{|\lambda_k(y)|} L_k(y) [e_{1,k}, \dots, e_{n,k}],e_{n+1,k} \Big\rangle \ud \mu_k(y) \\
&= \int_{\calY}  \sum_{k=1}^K \Big\langle \frac{\lambda_k(y_k)}{|\lambda_k(y_k)|} L_k(y_k) [e_{1,k}, \dots, e_{n,k}],e_{n+1,k} \Big\rangle \ud \mu(y),
\end{split}
\end{equation*}
where in the last line we denoted by $y_k$ the coordinate of $y \in \calY$ related to $\calY_k$.
For each $y \in \calY$ the absolute value of the integrand in the last line is dominated by
\begin{equation*}
\begin{split}
\calR_{\varpi,\mathcal P}(\calT)\| \{(e_{j,k})_{j \in  \mathcal P_{\RM}}\}_{k=1}^K \|_{\RM_{j_{\mathcal P}}(\varpi,  \mathcal P_{\RM})}
\prod_{j \in \mathcal P_{\Rad}}\|(e_{j,k})_{k=1}^K \|_{\Rad(X_j)}
|e_{j_{\mathcal P}}|_{X_{j_{\mathcal P}}}.
\end{split}
\end{equation*}
Since $(\calY,\mu)$ is a probability space this proves the claim.

\end{proof}

\subsection{The $\RMF_\varpi$ property}\label{sec:RMF}
Related to the $\calR_\varpi$ condition we will need a certain $\RMF_\varpi$ condition of the tuple of spaces $(X_1, \dots, X_{n+1})$. This condition is only defined when $n >2$.

Suppose $\varpi$ is an $(n+1)$-linear contraction as in \eqref{eq:Contr}.
Let $\calJ  \subset \mathcal J_n$   be such that $1\leq \# \calJ \le n-2$.
Suppose $f_j \in  L^1_{\loc}(X_j)$ for all $j \in \calJ$. Denote by $(f_j)_{j \in \calJ}$ the tuple of functions $(f_{j_1}, \dots, f_{j_\ell})$.
Let $\calD$ be a dyadic lattice in $\R^d$. For $v\in \calJ_n\setminus \calJ$,
we define the multilinear Rademacher maximal function $\RM_{\calD, \varpi,\calJ,v}[(f_j)_{j \in \calJ}]$ by
$$
\RM_{\calD, \varpi,\calJ,v}[(f_j)_{j \in \calJ}](x) = \Big\| \Big\{ (\langle f_j \rangle_Q)_{j \in \calJ} \colon x \in Q \in \calD\Big\} \Big\|_{\RM_v(\varpi, \calJ)}.
$$
Let $p_j \in (1, \infty)$ for $j \in \calJ_n$ and  define for all $\calJ \subset \calJ_n$ with $1 \le \# \calJ \le n-2$ the exponent
$p(\calJ)$ by $1/p(\calJ)= \sum_{j \in \calJ} 1/p_j$.
We say that $(X_1, \dots, X_{n+1})$ has the $\RMF_\varpi$ property relatively to a given dyadic lattice $\mathcal D$ in $\R^d$
and the tuple of exponents $(p_j)$ if
\begin{equation}\label{eq:RMAssump}
\max_{\substack{j_1,j_2 \in \calJ_n \\ j_1 \not= j_2}}
\min_{v\in \calJ_n\setminus \{j_1, j_2\}}
\max_{ \calJ\subset \calJ_n\setminus \{j_1, j_2, v\}}
\Big\| \RM_{\calD, \varpi,\calJ,v}: \prod_{ j \in \calJ} L^{p_j}(X_j) \to L^{{p(\mathcal J)}}(\R^d) \Big\| <\infty.
\end{equation}
The number defined in  \eqref{eq:RMAssump} will be referred to as the $\RMF_\varpi(\calD,(p_j))$ constant of the tuple $(X_1, \dots, X_{n+1})$.

Independence of the $\RMF_\varpi$ property   \eqref{eq:RMAssump}  on the dimension $d$ and
the lattice $\mathcal D$ can be proved with the same procedure used for the linear case by  Kemppainen  \cite{Kemp2}.
In particular, this means that if we have two lattices $\calD$ and $\calD'$ in $\R^d$, then
$\RMF_\varpi(\calD,(p_j))=\RMF_\varpi(\calD',(p_j))$.
On the other hand,
Lemma \ref{l:RMFeq} implies that  the $\RMF_\varpi$ condition is independent of the tuple of exponents in the sense that if
$(q_j)$ is another set of exponents then $\RMF_\varpi(\calD,(p_j))\sim\RMF_\varpi(\calD,(q_j))$.
Lemma \ref{l:RMFeq} also shows that it is not important to have a fixed tuple $(p_j)$ in \eqref{eq:RMAssump}; for each $\calJ$ appearing in
\eqref{eq:RMAssump} we could have related exponents $p^\calJ_j \in (1, \infty]$, $j \in \calJ$, such that the corresponding target exponent is finite.
Henceforth, we are authorized to not mention the dimension $d$, the choice of the lattice $\mathcal D$ and of the exponents, and  refer to a tuple   $(X_1, \dots, X_{n+1})$ enjoying \eqref{eq:RMAssump} as a tuple of spaces with the $\RMF_\varpi$ property.

\begin{rem}\label{rem:RMFClassic}
The original definition of an $\RMF$ property of a Banach space $X$ is in Hyt\"onen-McIntosh-Portal \cite{HMP}.
If $A \subset X$, then define
$$
\| A \|_{\RM}:= \sup \E \| \sum_{k=1}^K \varepsilon_k  \lambda_k a_k \|_X,
$$
where the supremum is over $K \in \N$, $a_k \in A$ and over scalars $\lambda_k$ such that $\sum_{k=1}^K |\lambda_k |^2 \le 1$.
In \cite{HMP} the Rademacher maximal function $M_R$ was defined by
$$
M_{R,\calD_0}f(x) :=M_R f(x) := \|\{\langle f \rangle_Q  \colon x \in Q \in \calD_0\}\|_{\RM},
$$
where $\calD_0$ is the standard lattice in $\R^d$ and $f \in L^1_{\loc} (X)$, and it was defined that $X$ has the $\RMF$ property if $M_R \colon L^2(X) \to L^2$ is bounded.
The Rademacher maximal function has further been studied for example by Kemppainen \cite{Kemp1,Kemp2}.
The boundedness of $M_R \colon L^p(X) \to L^p$ is independent of the dimension $d$ and the lattice $\mathcal D $ used in the definition, as well as of the exponent $p \in (1, \infty)$. A  definition akin to the one given in this article was previously given in \cite{DO}.
\end{rem}

The proof of the next lemma is a twist on a sparse domination argument presented in  \cite{CDPOBP} by  Culiuc, Di Plinio and Ou.

\begin{lem} \label{l:RMFeq}
Let $\calD$ be a dyadic lattice in $\R^d$.
Let $n \ge 3$ and let $\calJ\subset \mathcal J_n$ with
$1\leq \# \calJ\leq n-2$. Suppose $v \in \calJ_n \setminus \calJ$.
Assume that for some $q_j\in (1,\infty]$, $j\in \calJ$,  such that
$
q:=\big(\sum_{j\in \calJ} 1/q_j\big)^{-1}<\infty
$
the estimate
\[
 \| \RM_{\calD, \varpi,\calJ,v}[(f_j)_{j \in \calJ}] \|_{L^{q}}
\lesssim \prod_{j \in \calJ} \| f_j\|_{L^{q_j}(X_j)}
\]
holds.
Then, for all $p_j \in (1, \infty]$, $j \in \calJ$, such that $p:=\big(\sum_{j\in \calJ} 1/{p_j}\big)^{-1}<\infty$ we have the estimate
\[
 \| \RM_{\calD, \varpi,\calJ,v}[(f_j)_{j \in \calJ}] \|_{L^{p}}
\lesssim \prod_{j \in \calJ} \| f_j\|_{L^{p_j}(X_j)}.
\]
\end{lem}

\begin{proof}
We abbreviate $\RM:=\RM_{\calD, \varpi,\calJ,v}$.
First, we prove the weak type boundedness
$ \RM \colon \prod_{j \in \calJ} L^1(X_j) \to L^{\frac{1}{\ell}, \infty}$, where $\ell:=\#\calJ$. Then, we show that it implies a suitable pointwise sparse domination for certain finite
maximal functions, from which the claim follows.

We turn to the weak type estimate. Let $f_j \in L^1(X_j)$, $j \in \calJ$, and fix some $\lambda>0$.
It is enough to assume that $\| f_j \|_{L^1(X_j)}=1$ and show that
$$
| \{ \RM [(f_j)_{j \in \calJ}] > \lambda \} |  \lesssim \lambda^{-\frac1\ell}.
$$

We perform the usual Calder\'on-Zygmund decomposition.
Let $\calD_{j,\lambda}$ denote the collection of the maximal cubes $Q \in \calD$ such that
$
\langle | f_j |_{X_j} \rangle_Q > \lambda^{\frac{1}{\ell}}.
$
As usual, we write
$
f_j=g_j+b_j,
$
where
$$
g_j=1_{\bigcup \calD_{j,\lambda}}f_j+\sum_{Q \in \calD_{j,\lambda}} \langle f_j \rangle_Q1_Q
\qquad
\text{and}
\qquad
b_j=\sum_{Q \in \calD_{j,\lambda}}(f_j-\langle f_j \rangle_Q)1_Q.
$$

Write $\calJ=\{j_1, \dots, j_\ell\}$. We have
\begin{equation}\label{eq:GoodAndBad}
\RM [(f_j)_{j \in \calJ}]
\le \sum_{i=1}^\ell \RM [(g_{j_1}, \dots, g_{j_{i-1}}, b_{j_i}, f_{j_{i+1}}, \dots, f_{j_\ell})]
+\RM [(g_{j_1}, \dots, g_{j_\ell})].
\end{equation}
The assumed boundedness gives that
$$
\Big |\Big\{ \RM [(g_{j_1}, \dots, g_{j_\ell})] > \frac{\lambda}{\ell+1}\Big\}\Big|
\lesssim \lambda^{-q} \prod_{j=1}^\ell \| g_j \|_{L^{q_j}(X_j)}^q
\lesssim \lambda^{-q} \prod_{j=1}^\ell \lambda^{\frac{q(q_j-1)}{lq_j}} \|f_j \|_{L^{1}(X_j)}^{q/q_j}
\le \lambda^{-\frac{1}{\ell}},
$$
where we used the facts that $\| g_j \|_{L^\infty(X_j)} \lesssim \lambda^{\frac{1}{\ell}}$ and
$\|g_j \|_{L^{1}(X_j)} \le \|f_j \|_{L^{1}(X_j)}=1$.

Fix some $i \in \{1, \dots, l\}$ and consider the corresponding function from the sum in the right hand side of
\eqref{eq:GoodAndBad}.
Since
$| \bigcup \calD_{j_i,\lambda}| \le \lambda^{-\frac{1}{\ell}}$,
it is enough to consider
\begin{equation}\label{eq:BadInvolved}
\Big\{ x \not \in \bigcup \calD_{j_i, \lambda} \colon
\RM [(g_{j_1}, \dots, g_{j_{i-1}}, b_{j_i}, f_{j_{i+1}}, \dots, f_{j_\ell})](x) > \frac{\lambda}{\ell+1}\Big\}.
\end{equation}
Suppose $x \not \in \bigcup \calD_{j_i, \lambda}$. If $Q \in \calD$ and $Q' \in \calD_{j_i, \lambda}$
are such that $x \in Q$ and $Q \cap Q' \not= \varnothing$, then $ Q' \subset Q$. Using this and the zero integral of $b_{j_i}$
in the cubes $Q' \in \calD_{j_i, \lambda}$ we see that $\RM [(g_{j_1}, \dots, g_{j_{i-1}}, b_{j_i}, f_{j_{i+1}}, \dots, f_{j_\ell})](x)=0$.
Thus, the set in \eqref{eq:BadInvolved} is actually empty. This finishes the proof of the weak type estimate.

We move on to prove the sparse domination.
Let $\mathscr{C} \subset \calD$ be any finite collection such that there exists a cube $Q_0 \in \calD$ so
that $Q \subset Q_0$ for every $Q \in \scrC$.
We consider the related maximal function
$$
\RM_{\scrC}[(f_j)_{j \in \calJ}](x)
= \Big\| \big\{ (\langle f_j \rangle_Q)_{j \in \calJ} \colon x \in Q \in \scrC\big\} \Big\|_{\RM_v(\varpi, \calJ)}.
$$
Define the related truncated versions for every $Q' \in \calD$ by
$$
\RM_{\scrC, Q'}[(f_j)_{j \in \calJ}](x)
= \Big\| \big\{ (\langle f_j \rangle_{Q})_{j \in \calJ} \colon x \in Q \in \scrC, Q \subset Q'\big\} \Big\|_{\RM_v(\varpi, \calJ)},
$$
and also the numbers
$$
\RM_{\scrC}^{Q'}[(f_j)_{j \in \calJ}]
= \Big\| \big\{ (\langle f_j \rangle_Q)_{j \in \calJ} \colon Q \in \scrC, Q \supset Q'\big\} \Big\|_{\RM_v(\varpi, \calJ)}.
$$
Notice that $\RM_{\scrC, Q_0}=\RM_{\scrC}$.

Fix some functions $f_j \in L^1_{\loc}(X_j)$. Let $\calB$ denote the weak type norm of $\RM$.
We show that there exists a sparse collection $\calS=\calS( (f_j)_{j \in \calJ}) \subset \calD$ of subcubes
of $Q_0$ so that
\begin{equation}\label{eq:MaximalSparse}
\RM_{\scrC}[(f_j)_{j \in \calJ}]
\le 2^l\calB \sum_{Q \in \calS} \prod_{j \in \calJ}\langle | f_j |_{X_j} \rangle_{Q}1_{Q}.
\end{equation}
This follows via iteration from the estimate
\begin{equation}\label{eq:1IterSparse}
\RM_{\scrC}[(f_j)_{j \in \calJ}]
\le 2^l \calB \prod_{j \in \calJ} \langle | f_j |_{X_j} \rangle_{Q_0}1_{Q_0}
+ \sum_{Q \in \calE(Q_0)} \RM_{\scrC, Q} [(f_j)_{j \in \calJ}],
\end{equation}
where $\calE(Q_0)$ is a collection of pairwise disjoint cubes $Q \subset Q_0$ such that $\sum_{Q \in \calE(Q_0)} |Q| \le |Q_0|/2$.

We prove \eqref{eq:1IterSparse}.
Define the collection $\calE(Q_0)$ to be the set of the maximal cubes
$Q \in \calD$, $Q \subset Q_0$, such that
\begin{equation*}
\RM_{\scrC}^{Q}[(f_j)_{j \in \calJ}] > 2^l\calB \prod_{j \in \calJ} \langle | f_j |_{X_j} \rangle_{Q_0}.
\end{equation*}
If $Q \in \calE(Q_0)$, then
$
\RM_\scrC [(f_j1_{Q_0})_{j \in \calJ}](x)=\RM_\scrC [(f_j)_{j \in \calJ}](x) >2^l\calB \prod_{j \in \calJ} \langle | f_j |_{X_j} \rangle_{Q_0}
$
for all $x \in Q$. Therefore, applying the weak type boundedness,
we have that
$$
\sum_{Q \in \calE(Q_0)} |Q|
\le 2^{-1}  \Big( \prod_{j \in \calJ} \langle | f_j |_{X_j} \rangle_{Q_0}\Big)^{-\frac{1}{\ell}}
\prod_{j \in \calJ} \| f_j 1_{Q_0}\|_{L^1(X_j)}^{\frac{1}{\ell}}
=  2^{-1} |Q|.
$$

If $x \in Q_0 \setminus \bigcup \calE(Q_0)$, then
$\RM_\scrC [(f_j)_{j \in \calJ}](x) \le 2^l\calB \prod_{j \in \calJ} \langle | f_j |_{X_j} \rangle_{Q_0}$.
On the other hand, if $x \in Q \in \calE(Q_0)$, then
$$
\RM_\scrC [(f_j)_{j \in \calJ}](x)
\le \RM_{\scrC}^{Q^{(1)}}[(f_j)_{j \in \calJ}]
+\RM_{\scrC,Q}[(f_j)_{j \in \calJ}](x),
$$
where $\RM_{\scrC}^{Q^{(1)}}[(f_j)_{j \in \calJ}] \le 2^l\calB \prod_{j \in \calJ} \langle | f_j |_{X_j} \rangle_{Q_0}$
by the stopping condition. Thus, we have proved \eqref{eq:1IterSparse}, and therefore also \eqref{eq:MaximalSparse}.

From \eqref{eq:MaximalSparse} it follows that each $\RM_\scrC \colon \prod_{j \in \calJ} L^{p_j}(X_j) \to L^p$ is bounded
for all $p_j$ as in the statement. There exist cubes $Q_{i,N}  \in \calD$,
where $1 \le i \le  m$ for some $m \le 2^d$ and $N \in \N$,
so that $\ell(Q_{i,N})=2^N$, $Q_{i,N} \subset Q_{i,N+1}$, $Q_{i,N} \cap Q_{i',N}=\emptyset$ if $i \not= i'$, and
$\bigcup_i \bigcup_N Q_{i,N}=\R^d$.
What the number $m$ is depends on the lattice $\calD$.
Let $\scrC_{i,N}:=\{Q \in \calD \colon Q \subset Q_{i,N} , \ell(Q) \ge 2^{-N}\}$. Then,
$$
\sum_{i=1}^{m} \RM_{\scrC_{i,N}} [(f_j)_{j \in \calJ}](x) \nearrow \RM[(f_j)_{j \in \calJ}](x), \quad N \to \infty,
$$
for every $x$. Thus, by monotone convergence, we have that $\RM$ is also bounded.
\end{proof}

\begin{exmp}[The $\RMF_\varpi$ property in the function lattice case]\label{ex:LpRMF}
We continue with the setting and notation of Example \ref{ex:RMConst}. We also assume that
$(X_1, \dots, X_{n+1})$ has the $\RMF_{\varpi_0}$ property.
Suppose $\calJ \subset \mathcal J_n$, $1 \le \#\calJ \le n-2$ and $v\in \calJ_n\setminus \calJ$.
Suppose $f_j \in  L^{p_j}(\R^d;L^{p_j}(\Omega; X_j)) = L^{p_j}(\R^d \times \Omega; X_j)$ for all $j \in \calJ$, and fix a dyadic lattice $\calD$.
We have by Example \ref{ex:RMConst} that
\begin{align*}
\RM_{\calD, \varpi,\calJ,v}[(f_j)_{j \in \calJ}](x) &\lesssim \big\| \omega \mapsto \RM_{\calD, \varpi_0,\calJ,v}[(f_j(\cdot, \omega))_{j \in \calJ}](x)  \big\|_{L^{p(\calJ)}(\Omega)},
\end{align*}
so that
\begin{align*}
\big\|\RM_{\calD, \varpi,\calJ,v}[(f_j)_{j \in \calJ}](x) \big\|_{L^{p(\calJ)}(\R^d)}
&\lesssim \big\|  \big\| \RM_{\calD, \varpi_0,\calJ,v}[(f_j(\cdot, \omega))_{j \in \calJ}](x) \big\|_{L^{p(\calJ)}(\R^d)}\big\|_{L^{p(\calJ)}(\Omega)} \\
&\lesssim \Big\| \prod_{j \in \calJ} \|f_j(x,\omega)\|_{L^{p_j}(\R^d;X_j)}  \Big\|_{L^{p(\calJ)}(\Omega)} \\
&\le \prod_{j \in \calJ} \|f_j\|_{L^{p_j}(\R^d \times \Omega ;X_j)}.
\end{align*}
This shows that $(L^{p_1}(\Omega; X_1), \ldots, L^{p_{n+1}}(\Omega; X_{n+1}))$ has the $\RMF_{\varpi}$ property. In particular, by iterating the previous we have obtained that any   H\"older tuple of iterated  Banach function lattice spaces  $
(L^{p_1^1}_{\mu_1}\cdots L^{p_1^m}_{\mu_m}  , \ldots, L^{p_{n+1}^1}_{\mu_1}\cdots L^{p_{n+1}^m}_{\mu_1} )$ enjoys the  $\RMF_{\varpi}$ property with respect to
\[
\varpi\left( f_1,\ldots, f_{n+1}\right) = \int \prod_{j=1}^{n+1} f_j(t_1,\ldots, t_{m}) \, \mathrm{d}\mu_1(t_1) \cdots  \mathrm{d}\mu_m(t_m).
\]
\end{exmp}
\begin{exmp}[Noncommutative RMF property] \label{ex:NCLP} We are interested in operator valued, multilinear singular integrals acting on products of  noncommutative $L^p$-spaces; to this purpose, we need to study the corresponding Rademacher maximal function theory. We begin with a quick summary of the relevant definitions. For comprehensive background material on noncommutative $L^p$ spaces and their role in noncommutative probability and operator algebras we refer to the classical survey by Pisier and Xu \cite{PX}, to the recent monograph \cite{PB} by Pisier and references therein.

   Consider a von Neumann algebra $\mathcal A$ equipped with a normal, semifinite, faithful trace $\mathcal \tau$. For $1\leq p< \infty$, the corresponding noncommutative space $L^p(\mathcal A)$ is defined by the norm
\[
\|\xi\|_{L^p(\mathcal A)} = \left[\tau \left(\left(\xi^\star\xi \right)^{\frac p2}\right) \right]^{\frac1p}.
\]
Notice that $L^p(\mathcal A)$ is a $\UMD$ space for all $1<p<\infty$.
An enlightening example is obtained by choosing $\mathcal A=\mathcal B(H)$, the space of bounded linear operators on a complex separable  Hilbert space $H$ with orthonormal basis $\{e_j:j\in \mathbb N\}$ equipped with the usual trace
\[
\tau (\xi)=\sum_{j=1}^\infty \langle   \xi e_i,  e_i\rangle
\]
In this case $L^p(\mathcal A)$ is usually referred to as the $p$-th Schatten class and denoted by $S^p$.

Let now $(\Omega, \mu)$ be a $\sigma$-finite measure space and $\mathcal A$ a von Neumann algebra as above. We conveniently recall that $\mathcal M=L^\infty(\Omega) \otimes \mathcal A$ is also a von Neumann algebra equipped with normal semifinite faithful trace
\[
\nu (f \otimes \xi) = \left(\int_{\Omega} f\, \mathrm{d} \mu\right)\tau ( \xi),
\]
and that  we have the isometrically isomorphic identification $L^p(\mathcal M)\sim L^p(\Omega; L^p(\mathcal A))$.

We turn to the study of   noncommutative multilinear Rademacher maximal functions. This concept was first explored in the bilinear setting in \cite{DO}. Let $2\leq \kappa\leq n+1$, and  $p_1,\ldots, p_{\kappa}\in (1,\infty)$ be a H\"older tuple of exponents. We are interested in the Rademacher maximal functions associated to the tuple of spaces $X_{1},\ldots,  X_{n+1}$, where
 \begin{itemize}
 \item   $X_j=L^{p_j}(\mathcal A)$ for $1\leq j \leq \kappa$,
 \item $X_{j}=\mathbb C$ for $\kappa<j\leq n+1$,
 \end{itemize}
 equipped with the $(n+1)$-linear contraction
\[
\varpi(\xi_1,\ldots,\xi_{n+1}) = \tau \left(\prod_{j=1}^{\kappa}\xi_j  \right) \prod_{j=\kappa+1}^{n+1} \xi_j
\]
  We are able to establish a satisfactory multilinear   Rademacher maximal function estimate for the above tuple of spaces when  $\kappa \leq 3$. This is an improvement over the results of \cite{DO}, where the restriction $\kappa=2$ was imposed. Notice that the analysis of \cite{DO} concerned the nontangential version of the Rademacher maximal function, but the arguments therein can be recast in the dyadic setting as well.
\begin{prop} Suppose $\kappa=2$ or $\kappa=3$. Then the above defined tuple of spaces $X_{1},\ldots,  X_{n+1}$ has the $\mathrm{RMF}_\varpi$ property.
\end{prop}
\begin{proof} We work with the dyadic lattice $\mathcal D_0$ in dimension $d$ and therefore make use of the identification $L^p(\mathcal M)\sim L^p(\R^d; L^p(\mathcal A))$, where $\mathcal M=L^\infty(\R^d) \otimes   \mathcal A.$ The proof is split into several cases, all of which will make use of the following celebrated result of Junge \cite{JUN}: if $f\in L^p(\mathcal M)$, we may find
 $a_f,b_f\in L^{2p}(\mathcal M)$ and contractions $y_{\ell,f}\in \mathcal M$ such that
\begin{equation}
\label{e:0}
\|a_f \|_{L^{2p}(\mathcal M)}\|b_f\|_{L^{2p}(\mathcal M)}\lesssim \|f \|_{L^{p}(\mathcal M)}
\end{equation}
and
\begin{equation}
\label{e:1}
E_\ell f\coloneqq \sum_{\substack{Q\in \mathcal D_0 \\ \ell(Q)=2^{-\ell}}} E_Q f= a_f y_{\ell,f} b_f.
\end{equation}
By the definition of $\RMF_{\varpi}$ \eqref{eq:RMAssump}, we will prove the result according to the following cases:
\begin{enumerate}
\item[(i)] $\{j_1, j_2\}\cap \{1,\ldots,\kappa\}=\emptyset$;
\item[(ii)] $\#\{j_1, j_2\}\cap \{1,\ldots,\kappa\}=1$;
\item[(iii)] $\#\{j_1, j_2\}\cap \{1,\ldots,\kappa\}=2$
\end{enumerate}
In case (i), we will take $j_\calP=1$, so that $\#\calJ\cap \{1,\ldots,\kappa\}$ can be $0$, $1$ or $2$(if $\kappa=3$). In case (ii), without loss of generality, we can assume that $1\notin \{j_1, j_2\}$ so that we can still take $j_\calP=1$, then $\#\calJ\cap \{1,\ldots,\kappa\}$ can be $0$ or $1$(if $\kappa=3$). In case (iii), if $\kappa=3$, again we can assume $1\notin \{j_1, j_2\}$ and take  $j_\calP=1$, then $\#\calJ\cap \{1,\ldots,\kappa\}=0$; if $\kappa=2$ we will take $j_\calP=3$ and in this case $\#\calJ\cap \{1,\ldots,\kappa\}=0$. Let us consider the last situation first. We have
\[\begin{split}
&\quad \RM_{\calD_0, \varpi,\calJ,3}[(f_j)_{j \in \calJ}](x)   = \Big\| \big\{ (\langle f_j \rangle_Q)_{j \in \calJ} \colon x \in Q \in \calD_0\big\} \Big\|_{\RM_3(\varpi, \calJ)}
\\&\leq \sup_{\substack{\|\xi_{\ell,u}\|_{\Rad(L^{p_u}(\mathcal A))}=1\\ u=1,2}}  \sup_{\substack{i\in \calP_{\Rad}\setminus \{1,2\}\\ \|\xi_i\|_{\ell^2}=1}}  \Big| \sum_{\ell}     \tau \left(
\xi_{\ell,1}\xi_{\ell,2}  \right)  \prod_{i\in \calP_{\Rad}\setminus \{1,2\}}\xi_{\ell, i}\prod_{j\in \mathcal J} E_\ell f_j(x)\Big| \\
&=\sup_{\substack{\|\xi_{\ell,u}\|_{\Rad(L^{p_u}(\mathcal A))}=1\\ u=1,2}}  \sup_{\substack{i\in \calP_{\Rad}\setminus \{1,2\}\\ \|\xi_i\|_{\ell^2}=1}}  \mathbb E \Big| \tau\Big(  \sum_{k}    \varepsilon_k
\xi_{k,1}\sum_\ell \varepsilon_\ell\xi_{\ell,2}   \prod_{i\in \calP_{\Rad}\setminus \{1,2\}}\xi_{\ell, i}\prod_{j\in \mathcal J} E_\ell f_j(x)\Big) \Big|\\
\\ & \leq  \sup_{\ell}\prod_{j\in \mathcal J} |E_\ell f_j(x)|\le \prod_{j\in \mathcal J} M f_j(x),
 \end{split}\] where we have used H\"older's inequality and Kahane contraction principle. We conclude this case by using the boundedness of $\prod_{j\in \mathcal J} M f_j(x)$. Now we are left to deal with the remaining cases, keep in mind that we always have $j_\calP=1$.

\subsubsection*{Case $\#\mathcal J \cap \{1,\ldots,\kappa\}=2$}  This forces that $\kappa=3$.
We are then allowed to estimate \[\begin{split}
\RM_{\calD_0, \varpi,\calJ,1}[(f_j)_{j \in \calJ}](x) & = \Big\| \big\{ (\langle f_j \rangle_Q)_{j \in \calJ} \colon x \in Q \in \calD_0\big\} \Big\|_{\RM_1(\varpi, \calJ)}\\&\leq \sup_{|\xi_1|_{L^{p_1}(\mathcal A)}=1} \sup_{\ell}  \left| \tau \left(  f_1(x) E_\ell f_2(x)  E_\ell\xi_3\right)\right| \prod_{j\in \mathcal J\setminus \{2,3\} } Mf_j
 \\ & \lesssim \prod_{j=2}^{3} \|a_{f_j}(x)\|_{L^{2p_j}(\mathcal A)}\|b_{f_j}(x)\|_{L^{2p_j}(\mathcal A)}\prod_{j\in \mathcal J\setminus \{2,3\} } Mf_j
 \end{split}\]
where the first inequality follows from the definition and the second by \eqref{e:1} and H\"older's inequality. For $q_1=(p_1)'$ Holder's inequality yields
\[
\begin{split}
\|\RM_{\calD_0, \varpi,\calJ,1}[(f_j)_{j \in \calJ}]  \|_{L^{q_1}(\R^d)}& \lesssim \prod_{j=2}^3   \|a_{f_j}\|_{L^{2p_j}(\mathcal M)}\|b_{f_j}\|_{L^{2p_j}(\mathcal M)}\prod_{j\in \mathcal J\setminus \{2,3\} } \| Mf_j\|_{L^\infty(\R^d)} \\ &\lesssim
 \prod_{j=2}^3    \| f_j \|_{L^{ p_j}(\mathcal M)} \prod_{j\in \mathcal J\setminus \{2,3\} }\| f_j\|_{L^\infty(\R^d)}.
 \end{split}
\]
In view of \eqref{e:0}, we obtain the bound
\[
\RM_{\calD_0, \varpi,\calJ}:\left(\prod_{j=2}^3 L^{p_j}(\R^d; L^{p_j}(\mathcal A) ) \right) \times \left(  \prod_{j\in \mathcal J\setminus \{2,3\} } {L^\infty(\R^d)} \right) \to L^{q_1}(\R^d)
\]
 and conclude the required condition for this class of $\mathcal J$ by means of Lemma \ref{l:RMFeq}.
\subsubsection*{Case $\#\mathcal J \cap \{1,\ldots,\kappa\}=1$} This is actually the most complex case.
If $\kappa=2$, we may proceed similar as in the previous case, details are omitted. We are therefore left with treating the case where $\kappa=3$ and without loss of generality we assume $2\in \mathcal J_n\setminus \mathcal J$. We may estimate the Rademacher maximal function corresponding to the $
{\RM_1(\varpi, \calJ)}$ norm
\[\begin{split}
&\quad \RM_{\calD_0, \varpi,\calJ,1}[(f_j)_{j \in \calJ}](x)   = \Big\| \big\{ (\langle f_j \rangle_Q)_{j \in \calJ} \colon x \in Q \in \calD_0\big\} \Big\|_{\RM_1(\varpi, \calJ)}\\&\leq \sup_{\substack{\|\xi_{\ell,u}\|_{\Rad(L^{p_u}(\mathcal A))}=1\\ u=1,2}}  \sup_{\substack{i\in \calP_{\Rad}\setminus \{2\}\\ \|\lambda_{\ell,i}\|_{\Rad  }=1}} \left|\sum_{\ell}    \tau \left(   \xi_1(x) \xi_{\ell,2} E_\ell f_3 \prod_{i\in \calP_{\Rad}\setminus \{2\}}\lambda_{\ell,i}\right)  \prod_{j\in \mathcal J\setminus \{3\}} E_\ell f_j(x)\right|
\\ & \leq  \left( \sup_{\|\lambda_{\ell }\|_{\Rad  }=1} \left\| \lambda_\ell E_\ell f_3(x)  \right\|_{\Rad(L^{p_3}(\mathcal A))} \right) \prod_{j\in \mathcal J\setminus \{3\} } Mf_j,
 \end{split}\]
where the second step is obtained via H\"older's inequality and the  Kahane contraction principle. Now using \eqref{e:1}, we obtain
\[
\begin{split}
&\quad \left\| \lambda_\ell E_\ell f_3(x)  \right\|_{\Rad(L^{p_3}(\mathcal A))}
 \lesssim \|a_{f_3}(x)\|_{L^{2p_3}(\mathcal A)}
\left\| \lambda_\ell y_{\ell,f_3}(x) b_{f_3}(x) \right\|_{\Rad(L^{2p_3}(\mathcal A))}
\\
& \lesssim \|a_{f_3}(x)\|_{L^{2p_3}(\mathcal A)} \left(\sum_{\ell} |\lambda_\ell^2| \|y_{\ell,f_3}(x) b_{f_3}(x)  \|_{ L^{2p_3}(\mathcal A))}^2  \right)^{\frac12}
 \lesssim
  \|a_{f_3}(x)\|_{L^{2p_3}(\mathcal A)}\|b_{f_3}(x)\|_{L^{2p_3}(\mathcal A)}.
\end{split}
\]
We have crucially used type $2$ of $L^{2p_3}(\mathcal A)$  when passing to the second line.
Taking $L^{p_3}(\R^d)$ norm and using \eqref{e:0}, we realize  that we have proved the estimate
\[
\RM_{\calD_0, \varpi,\calJ,1}:  L^{p_3}(\R^d; L^{p_3}(\mathcal A) ) \times   \prod_{j\in \mathcal J\setminus \{3 \} } {L^\infty(\R^d)} \to L^{p_3}(\R^d)
\]
which completes the proof of this case.
\subsubsection*{Case $\mathcal J \cap \{1,\ldots, \kappa\}=\varnothing$} This is the easiest case. If $\kappa=2$ it can be proved similarly as the the previous case (even easier as we don't need to deal with $E_\ell f_3$). And if $\kappa=3$, it can be proved similarly as at the very beginning.
\end{proof}

\end{exmp}

\section{Operator-valued multilinear shifts}\label{sec:ShiftBdd}
Our basic result concerning the boundedness of n-linear operator-valued dyadic shifts is summarized by the following sparse domination principle.
\begin{thm}\label{thm:MultiShifts}
Let  $n \ge 2$, suppose $X_1, \dots, X_n, Y_{n+1}$ are $\UMD$ spaces and denote $X_{n+1}=Y^*_{n+1}$.
Assume that $(X_1, \dots, X_{n+1})$ has the $\RMF_\varpi$ property with some $\varpi$ as described in Section \ref{sec:RMF}.
Let $f_m \in L^\infty_c(X_m)$ for $m=1,\ldots, n+1$.
Suppose $\calD$ is a dyadic grid and $\eta \in (0,1)$.
Then there exists an $\eta$-sparse collection $\mathcal{S}=\mathcal{S}((f_m), \eta)\subset \calD$ so that the following holds.

Suppose $S^k:=S^k_{\calD}$, $k=(k_1,\dots, k_{n+1})$, $0 \le k_i \in \Z$, is an $n$-linear dyadic shift with complexity $k$
and with coefficient operators $a_{K,(Q_i)} \in  \calL(\prod_{m=1}^n X_m, Y_{n+1})$.
Recall the collection $\calC(S^k)$ of normalized coefficients  from \eqref{eq:NormCoef}.
We have
\begin{equation}\label{SparseForShifts}
|\langle S^k(f_1, \ldots, f_n), f_{n+1}\rangle| \lesssim_{\eta} (1+\kappa)^{n-1} \calR_{\varpi}(\calC(S^k))  \sum_{Q\in\mathcal{S}}|Q|\prod_{m=1}^{n+1}\bla |f_m|_{X_m} \bra_Q,
\end{equation}
where $\kappa = \max k_m$.

In particular, if $p_1, \dots, p_n \in (1, \infty]$, $q_{n+1} \in (1/n, \infty)$ and $\sum_{m=1}^n 1/p_m=1/q_{n+1}$,
then
\begin{equation}\label{eq:MultiShiftBdd}
\| S^k(f_1, \dots, f_n) \|_{L^{q_{n+1}}(Y_{n+1})}
\lesssim (1+\kappa)^{n-1}\calR_{\varpi}(\calC(S^k)) \prod_{m=1}^n \| f_m\|_{L^{p_m}(X_m)}.
\end{equation}
\end{thm}

\begin{proof}
The estimate
$$
|\langle a_{K,(Q_i)}[e_1, \ldots, e_{n}], e_{n+1} \rangle| \le \calR_{\varpi}(\calC(S^k)) \frac{\prod_{m=1}^{n+1}|Q_m|^{1/2}}{|K|^n} \prod_{m=1}^{n+1} |e_m|_{X_m}
$$
follows directly from the definition.
We are considering a shift
$$
S^k (f_1,\dots,f_n)=\sum_{K \in \calD} A_K(f_1, \dots, f_n),
$$
where
$$
 A_K(f_1, \dots, f_n)
 = \sum_{\substack{Q_1, \dots, Q_{n+1} \in \calD \\ Q_i^{(k_i)}=K}}
a_{K,(Q_i)}[ \langle f_1, \wt h_{Q_1} \rangle, \dots, \langle f_n, \wt h_{Q_n} \rangle] \wt h_{Q_{n+1}}.
$$
By Lemma \ref{lem:SparseModel} it remains to prove that
\begin{equation}\label{eq:ShiftSomeExpo}
|\langle S^k(f_1, \ldots, f_n), f_{n+1}\rangle|
\lesssim (1+\kappa)^{n-1} \calR_{\varpi}(\calC(S^k)) \prod_{m=1}^{n+1} \|f_m\|_{L^{p_m}(X_m)}
\end{equation}
for some $p_m \in (1,\infty)$ satisfying $\sum_{m=1}^{n+1} 1/p_m = 1$. Notice that
by proving this, we also prove the corresponding estimate for all subshifts
$S^k_{\calD'}=\sum_{K \in \calD'} A_K$, where $\calD' \subset \calD$.
This is required in the assumptions of Lemma \ref{lem:SparseModel}.

If $n \ge 3$ we assume the following. Let $j_0,j_1 \in \calJ_n$ be the indices such that the corresponding Haar functions of the shift $S^k$ are cancellative.
Then, since $(X_1, \dots, X_{n+1})$ is assumed to have the $\RMF_\varpi$ property there exists a $v \in \calJ_n \setminus \{j_0,j_1\}$
so that the maximal functions $\RM_{\calD, \varpi, \calJ,v}$ are bounded for all $\calJ \subset \calJ_n \setminus (\calJ \cup \{j_0,j_1\})$, see \eqref{eq:RMAssump}.
Notice that $\wt h_{Q_{v}}=h_{Q_{v}}^0$.
For convenience of notation we assume that $v =n+1$ but the general case is handled similarly. If $n=2$ there are no $\RMF_\varpi$ assumptions involved and we assume
for convenience that $\wt h_{Q_{3}}=h_{Q_{3}}^0$.

Having made the initial assumptions we proceed with an arbitrary $n \ge 2$.
We let $\{ \{n+1\}, \calJ_{\Rad}^0, \calJ_{\RM}^0\}$ be the admissible partition such that $\#\calJ_{\Rad}^0=2$ and
$\wt h_{Q_j}=h_{Q_j}$ for $j \in \calJ_{\Rad}^0$; for $j \in \calJ_{\RM}^0$ we have $\wt h_{Q_j}=h_{Q_j}^0$.
If $m \in \calJ^0_{\Rad}$, then $\langle f_m, \wt h_{Q_m} \rangle=\langle \Delta^{k_m}_K f_m,  h_{Q_m} \rangle$,
and if $m \in \calJ_{\RM}^0$, then $\langle f_m, \wt h_{Q_m} \rangle=\langle E^{k_m}_K f_m,  h_{Q_m}^0 \rangle$.
By writing for $m \in \calJ_{\RM}^0$ that
$$
E^{k_m}_K f_m= \sum_{l=0}^{k_m-1} \Delta^{l}_K f_m + E_K f_m,
$$
we see that the shift $S^k$ can be split into at most $(1+\kappa)^{n-2}$
operators of the form
$$
\sum_{K \in \calD} A_K(P_K^{l_1}f_1, \dots, P^{l_n}_K f_n ),
$$
where the following holds. We have $0 \le l_m \in \Z$ and $l_m \le k_m$. If $l_m \not=0$ then $P^{l_m}_K =\Delta^{l_m}_K$, and  if $l_m=0$ then
$P^{l_m}_K=P_K$ can be either $\Delta_K$ or $E_K$. For $m \in \calJ^0_{\Rad}$ we have $l_m=k_m$ and $P^{l_m}_K = \Delta^{l_m}_K$.
Now we fix one such operator and show that it is bounded as desired. Let $\calJ'_{\RM} \subset \calJ_{\RM}^0$ be the subset of those indices
$m$ such that $P^{l_m}_K=E_K$, and set $\calJ'_{\Rad}=\{1, \dots, n\} \setminus \calJ_{\RM}'$.
Recall the lattices $\calD_{j,\kappa}$, $j \in \{0, \dots, \kappa\}$, from \eqref{eq:SubLattice}.
We fix one $j$ and start to consider the term
\begin{equation}\label{eq:SplitShift}
\Big\langle\sum_{K \in \calD_{j,\kappa}} A_K(P_K^{l_1}f_1, \dots, P^{l_n}_K f_n ), f_{n+1} \Big\rangle.
\end{equation}
We prove an estimate that is independent of $j$.

Let $K \in \calD_{j,\kappa}$.
Define an operator-valued kernel $a_K \colon \R^{d(n+1)} \to \calL(\prod_{m=1}^n X_m, Y_{n+1})$ by
$$
a_K(x,y_1,\dots, y_n)
= \sum_{\substack{Q_1, \dots, Q_{n+1} \in \calD \\ Q_i^{(k_i)}=K}}
|K|^n a_{K,(Q_i)}\wt h_{Q_1}(y_1)\cdots \wt h_{Q_{n}}(y_n)\wt h_{Q_{n+1}}(x).
$$
Notice that $a_K$ is supported in $K^{n+1}$ and that if $x,y_1, \dots, y_n \in K$, then for some $\epsilon\in \{-1,1\}$ we have
$\epsilon a_K(x,y_1,\dots, y_n) \in \calC(S^k)$ .
Thus, there holds
$$
\calR_{\varpi}(\{a_K(x,y_1, \dots,y_n)\colon  K \in \calD_{j,\kappa}, x,y_1, \dots, y_n \in \R^d\}) \le \calR_{\varpi}(\calC(S^k)).
$$
Below we write $a_K(x,y)=a_K(x, y_1, \dots, y_n)$ for $y=(y_1, \dots, y_n) \in \R^{dn}$.
We have that
\begin{equation}\label{eq:ShiftAsIntOp}
\begin{split}
A_K(P_K^{l_1}f_1, \dots, P^{l_n}_K f_n )(x)
=\frac{1}{|K|^n} \int_{K^n} a_K(x, y) [ P_K^{l_1}f_1(y_1),\dots, P^{l_n}_K f_n(y_n)] \ud y.
\end{split}
\end{equation}

Let $(\calV, \nu)$ be the space related to decoupling, see Section \ref{sec:BanachProperties}.
Let also $\calV^n$ be the $n$-fold product of these, and let $\nu_n$ be the related measure. If $y=(y_1, \dots, y_n) \in \calV^n$, then
$y_{m,K}$ for $m \in \{1, \dots, n\}$ and $K \in \calD$ denotes the coordinate of $y_m$ related to $K$. If $y \in \calV^n$, we write
$y_K$ to mean the tuple $(y_{1,K}, \dots, y_{n,K})$.
If $K \in \calD_{j,\kappa}$ then we can rewrite \eqref{eq:ShiftAsIntOp} as
\begin{equation}\label{eq:IntroV}
A_K(P_K^{l_1}f_1, \dots, P^{l_n}_K f_n )(x) =\int_{\calV^n} a_K(x, y_K) [ P_K^{l_1}f_1(y_{1,K}),\dots, P^{l_n}_K f_n(y_{n,K})] \ud \nu_n (y) .
\end{equation}

Applying \eqref{eq:IntroV} it is seen that \eqref{eq:SplitShift} equals
\begin{equation*}
\begin{split}
\int_{\R^d}\int_{\calV^n}\sum_{K \in \calD_{j,\kappa}} \langle a_K(x, y_K)
[ P_K^{l_1}f_1(y_{1,K}),\dots, P^{l_n}_K f_n(y_{n,K})], f_{n+1}(x) \rangle \ud \nu_n (y) \ud x.
\end{split}
\end{equation*}
Notice that here we may multiply each of the functions inside $a_K(x,y_K)$ by $1_K(x)$, since $a_K(x, \cdot) = 0$ unless $x \in K$.
Applying the $\calR_\varpi$-boundedness (see Remark \ref{rem:OldRBdd}) of the kernels
gives that the absolute value of \eqref{eq:SplitShift} is dominated by
\begin{equation} \label{eq:AfterR}
\begin{split}
&\int_{\R^d}  \int_{\calV^n}  \prod_{m \in \calJ_{\Rad}'} \| (1_K(x)\Delta^{l_m}_K f_m(y_{m,K}))_{K \in \calD_{j,\kappa}} \|_{\Rad(X_m)} \\
& \times \Big\| \Big\{\Big( 1_K(x)E_K f_m(y_{m,K})\Big)_{m \in \calJ'_{\RM}} \colon K \in \calD_{j,\kappa} \Big\} \Big\|_{\RM(\varpi, \calJ'_{\RM}, n+1)} |f_{n+1}(x)|_{X_{n+1}} \ud \nu_n (y) \ud x.
\end{split}
\end{equation}

Let $F_{\calJ'_{\RM}}$ be the tuple of the functions $f_m$ with indices $m \in \calJ'_{\RM}$.
We notice that for all $x \in \R^d$ and $y \in \calV^n$ there holds that
$$
\Big\| \Big\{\Big( 1_K(x)E_K f_m(y_{m,K})\Big)_{m \in \calJ'_{\RM}} \colon K \in \calD_{j,\kappa} \Big\} \Big\|_{\RM(\varpi, \calJ'_{\RM}, n+1)}
\le \RM_{\calD,\varpi, \calJ'_{\RM}, n+1} (F_{\calJ'_{\RM}})(x),
$$
and by assumption we have that
\begin{equation}\label{eq:RMAssumption}
\begin{split}
\Big(\int_{\R^d} \int_{\calV^n} \RM_{\calD,\varpi, \calJ'_{\RM},n+1} (F_{\calJ'_{\RM}})(x)^r \ud \nu_n (y) \ud x \Big)^{1/r}
\lesssim \prod_{m \in \calJ'_{\RM}} \|f_m \|_{L^{p_m}(X_m)}.
\end{split}
\end{equation}
Here $r$ is the exponent defined by $1/r=\sum_{m \in \calJ'_{\RM}}1/p_m$.

Let $m \in \calJ'_{\Rad}$. Then
Kahane-Khintchine inequality \eqref{eq:KK} shows that
\begin{equation}\label{eq:RadTerm}
\begin{split}
\int_{\R^d} \int_{\calV^n} & \| (1_K(x)\Delta_K^{l_m}f_m(y_{m,K}))_{K \in \calD_{j,\kappa}}\|_{\Rad(X_m)}^{p_m} \ud \nu_n (y) \ud x \\
& \sim \E \int_{\R^d} \int_{\calV} \Big | \sum_{K \in \calD_{j,\kappa}} \varepsilon_K 1_K(x)\Delta_K^{l_m}f_m(y_{m,K}) \Big |_{X_m}^{p_m} \ud x \ud \nu (y)
\lesssim \| f \|_{L^{p_m}(X_m)}^{p_m}.
\end{split}
\end{equation}
The last step is based on the decoupling estimate \eqref{eq:DecEst}.

Now, if we use H\"older's inequality in \eqref{eq:AfterR} and combine the result with \eqref{eq:RMAssumption} and \eqref{eq:RadTerm}, we finally see that
$$
|\eqref{eq:SplitShift}| \lesssim \prod_{m=1}^{n+1} \|f_m \|_{L^{p_m}(X_m)}.
$$
This concludes the proof.
\end{proof}

\section{Operator-valued multilinear paraproducts}\label{sec:ParaBdd}

We begin with some definitions.
Suppose $X_1, \dots, X_n, Y_{n+1}$ are Banach spaces.
Suppose  $T \in \calL(\prod_{m=1}^nX_m,Y_{n+1})$.
Let $k \in \{2, \dots, n\}$ and $e_m \in X_m$ for $m \in \{k, \dots, n\}$.
Define the operator $T[e_k, \dots, e_n] \in \calL(\prod_{m=1}^{k-1}X_m,Y_{n+1})$ by
$$
(T[e_k, \dots, e_n])[e_1, \dots, e_{k-1}]:=
T[e_1, \dots, e_n],
$$
where $ e_m \in X_m$, $m \in \{1, \dots, k-1\}$.
We see that for $k \in \{2, \dots, n-1\}$ there holds that
\begin{equation}\label{eq:Iteration1}
 \| T[e_k, \dots, e_n]  \|_{\calL(\prod_{m=1}^{k-1}X_m,Y_{n+1})}
\le |e_k|_{X_k}  \| T[e_{k+1}, \dots, e_n]  \|_{\calL(\prod_{m=1}^{k}X_m,Y_{n+1})}
\end{equation}
and that
\begin{equation}\label{eq:Iteration2}
\big \|T[e_n]\big\|_{\calL(\prod_{m=1}^{n-1}X_m,Y_{n+1})}
\le |e_n|_{X_n} \| T \|_{\calL(\prod_{m=1}^{n}X_m,Y_{n+1})}.
\end{equation}

Let $a = (a_Q)_{Q \in \calD}$, where
$$
a_Q \in \calL\Big(\prod_{m=1}^nX_m ,Y_{n+1}\Big).
$$
Assume that there exists a $\UMD$ subspace $\calT_k \subset \calL(\prod_{m=1}^kX_m,Y_{n+1}) $
for every $k \in \mathcal{I}_{n-1}$ so that
$$
a_Q  \in \calT_n
$$
and for $k < n$
$$
a_Q[e_{k+1}, \dots, e_n] \in \calT_k, \quad \text{for every } e_{k+1} \in X_{k+1}, \dots, e_n \in X_n.
$$
If these conditions are satisfied, we say that $a$ satisfies the \emph{$\UMD$ subspace condition}.

The next theorem generalises the result about boundedness of operator-valued paraproducts from \cite{HH} to the multilinear context.
\begin{thm}\label{thm:MultiPara}
Let  $n \ge 2$, $X_1, \dots, X_n$ be Banach spaces, $Y_{n+1}$ be a $\UMD$ space and $X_{n+1} := Y_{n+1}^*$.
Let $f_m \in L^\infty_c( X_{m})$ for $m=1,\ldots, n+1$.
Suppose $\calD$ is a dyadic grid and $\eta \in (0,1)$.
Then there exists an $\eta$-sparse collection $\mathcal{S}=\mathcal{S}((f_m), \eta)\subset \calD$ so that the following holds.

Suppose $a = (a_Q)_{Q \in \calD}$ satisfies the $\UMD$ subspace condition as above.
For a paraproduct $\pi := \pi_{\calD, a}$ we have for all $r \in (0,\infty)$ that
\begin{equation}\label{SparseForPara}
|\langle \pi(f_1, \ldots, f_n), f_{n+1}\rangle| \lesssim_{\eta} \| a \|_{\BMO_{\calD, r}(\calT_n)}  \sum_{Q\in\mathcal{S}}|Q|\prod_{m=1}^{n+1}\bla |f_m|_{X_m} \bra_Q.
\end{equation}
In particular, if $p_1, \dots, p_n \in (1, \infty]$, $q_{n+1} \in (1/n, \infty)$ and $\sum_{m=1}^n 1/p_m=1/q_{n+1}$,
then
\begin{equation}\label{eq:MultiShiftBdd}
\| \pi(f_1, \dots, f_n) \|_{L^{q_{n+1}}(Y_{n+1})}
\lesssim \| a \|_{\BMO_{\calD, r}(\calT_n)} \prod_{m=1}^n \| f_m\|_{L^{p_m}(X_m)}.
\end{equation}
\end{thm}

\begin{proof}
We first fix $r \in (1,\infty)$ and assume that $\| a \|_{\BMO_{\calD, r}(\calT_n)}= 1$. We will use Lemma \ref{lem:SparseModel} again. First, the correctly normalized
estimate
$$
\Big| \frac{1}{|Q|^{n/2}} a_Q [e_1, \ldots, e_n], e_{n+1}\rangle\Big| \le
\frac{|Q|^{n/2 + 1/2}}{|Q|^n} \prod_{m=1}^{n+1} |e_m|_{X_m}
$$
follows directly from the $\BMO$ assumption.

Choose $p_m \in (1,\infty)$ so that $\sum_{m=1}^n 1/p_m = 1/r$.
We show that
\begin{equation}\label{eq:bddParBanach}
\Big \| \sum_{\substack{Q \in \calD \\Q \subset Q_0}}
a_Q [ \langle F \rangle_Q]h_Q \Big\|_{L^{r}(Y_{n+1})}
 \lesssim  \prod_{m=1}^n \| f_m \|_{L^{p_m}(X_m)},
\end{equation}
where $Q_0 \in \calD$ is arbitrary and $\langle F \rangle_Q := (\langle f_1 \rangle_Q, \ldots, \langle f_n \rangle_Q)$.
We denote by $\calD(Q_0)$ the set of cubes $Q \in \calD$ such that $Q \subset Q_0$.

We begin by constructing a collection of stopping cubes.
Set $\calS_0:= \{Q_0\}$, and suppose that $\calS_0, \dots, \calS_k$ are defined for some $k$.
If $S \in \calS_k$, we define $\ch_\calS(S)$ to be the collection of the maximal cubes $Q \in \calD$ such that
$Q \subset S$ and
\begin{equation}\label{eq:Stopping}
\max\Big( \frac{\langle | f_1|_{X_1} \rangle_Q}{\langle | f_1|_{X_1} \rangle_S}, \dots,
\frac{\langle | f_n|_{X_n} \rangle_Q}{\langle | f_n|_{X_n} \rangle_S}\Big)> 2n.
\end{equation}
Then, we define
$
\calS_{k+1}:= \bigcup_{S \in \calS_k} \ch_\calS(S)
$
and finally
$
\calS:= \bigcup_{k=0}^\infty \calS_k.
$
If $Q \in \calD$ and $Q \subset Q_0$, then the unique minimal cube $S \in \calS$ containing $Q$ is denoted by $\pi_\calS(Q)$.
If $S \in \calS$ we define $E(S):= S \setminus \bigcup_{S' \in \ch_\calS(S)} S'$.

It follows from the construction that $\calS$ is a sparse collection.
For $S \in \calS$ define
$
F_S:=(f_{1,S}, \dots, f_{n,S}),
$
where
\begin{equation}\label{eq:DefBlock}
f_{m,S}:= \sum_{S' \in \ch_\calS(S)} \langle f_m \rangle_{S'}1_{S'} + 1_{E(S)} f_m.
\end{equation}
From the stopping condition \eqref{eq:Stopping} it is deduced that
\begin{equation}\label{eq:LInftyProp}
\| f_{m,S} \|_{L^\infty(X_m)} \lesssim \langle |f_m|_{X_m} \rangle_{S}.
\end{equation}
Equation \eqref{eq:DefBlock} implies that
\begin{equation}\label{eq:SameAve}
\langle F \rangle_Q= \langle F_{S } \rangle_Q
\end{equation}
for all $Q \in \calD(Q_0)$ such that  $\pi_\calS Q=S$.

Pythagoras' theorem \eqref{eq:Pythagoras} and Equation \eqref{eq:SameAve} give that
\begin{equation}\label{eq:ApplyPyth}
\begin{split}
\Big \| \sum_{Q \in \calD(Q_0) } & a_Q [\langle F \rangle_Q] h_Q \Big \|_{L^{r}(Y_{n+1})} \\
& \sim \Big( \sum_{S \in \calS} \Big \|  \sum_{\substack{Q \in \calD(Q_0) \\ \pi_\calS Q =S}}
a_Q [\langle F_S \rangle_Q] h_Q \Big \|_{L^{r}(Y_{n+1})}^{r} \Big)^{1/r}.
\end{split}
\end{equation}
It will be shown that for all $S \in \calS$ there holds that
\begin{equation}\label{eq:InftyTest}
\Big \|  \sum_{\substack{Q \in \calD(Q_0) \\ \pi_\calS Q =S}}
a_Q [\langle g_1 \rangle_Q, \dots, \langle g_n \rangle_Q] h_Q \Big \|_{L^{r}(Y_{n+1})}
\lesssim \prod_{m=1}^n \| g_m \|_{L^\infty(X_m)} | S |^{1/r}
\end{equation}
for all $g_m \in L^\infty(X_m)$, $m \in \{1, \dots, n\}$.
This combined with   \eqref{eq:LInftyProp} and \eqref{eq:ApplyPyth} implies that the left hand side of \eqref{eq:ApplyPyth} satisfies
\begin{equation*}
\begin{split}
LHS \eqref{eq:ApplyPyth}\lesssim  \Big( \sum_{S \in \calS } \prod_{m=1}^n \langle |f_m|_{X_m} \rangle_{S}^{r} | S | \Big)^{1/r}
& \le \prod_{m=1}^n \Big( \sum_{S \in \calS } \langle |f_m|_{X_m} \rangle_{S}^{p_m} | S | \Big)^{1/p_m} \\
& \lesssim \prod_{m=1}^n \| f_m \|_{L^{p_m}(X_m)},
\end{split}
\end{equation*}
where the last step followed from an application of the Carleson embedding theorem based on the sparseness of $\calS$.

Fix $S \in \calS$ and suppose $g_m \in L^\infty(X_m)$, $m \in \{1, \dots, n\}$.  Let $\{\varepsilon_Q\}_{Q \in \calD}$
be a collection of independent random signs.
For all $x \in \R^d$ the collections of random variables $\{\varepsilon_Q h_Q(x)\}_{Q \in \calD}$ and
$\{\varepsilon_Q |h_Q(x)|\}_{Q \in \calD}$ are identically distributed. This gives, using the $\UMD$ property of $Y_{n+1}$, that
\begin{equation}\label{eq:ReadyForStein}
\begin{split}
\Big \| &  \sum_{\substack{Q \in \calD(Q_0) \\ \pi_\calS Q =S}}
a_Q [\langle g_1 \rangle_Q, \dots, \langle g_n \rangle_Q] h_Q \Big \|_{L^{r}(Y_{n+1})} \\
&\sim \E\Big \|   \sum_{\substack{Q \in \calD(Q_0) \\ \pi_\calS Q =S}}
\varepsilon_Qa_Q [\langle g_1 \rangle_Q, \dots, \langle g_n \rangle_Q] |h_Q| \Big \|_{L^{r}(Y_{n+1})}.
\end{split}
\end{equation}
Notice that $|h_Q|=1_Q/|Q|^{1/2}$. This allows us to use Stein's inequality in the next estimate.
The UMD-valued version of Stein's inequality is due to Bourgain, for a proof see e.g. Theorem 4.2.23 in the book \cite{HNVW1}.

The right hand side of \eqref{eq:ReadyForStein}
satisfies
\begin{equation*}
\begin{split}
RHS\eqref{eq:ReadyForStein}
&=\E\Big \|   \sum_{\substack{Q \in \calD(Q_0) \\ \pi_\calS Q =S}}
\varepsilon_Q\Big\langle a_Q [ g_1(\cdot) , \dots, \langle g_n \rangle_Q] \Big\rangle_Q |h_Q| \Big \|_{L^{r}(Y_{n+1})} \\
&\lesssim \E\Big \|   \sum_{\substack{Q \in \calD(Q_0) \\ \pi_\calS Q =S}}
\varepsilon_Q a_Q [ g_1(\cdot) , \dots, \langle g_n \rangle_Q]  |h_Q(\cdot)| \Big \|_{L^{r}(Y_{n+1})} \\
& \le \E\Big \| |g_1|_{X_1}
\Big \|\sum_{\substack{Q \in \calD(Q_0) \\ \pi_\calS Q =S}}
\varepsilon_Q a_Q [ \langle g_2 \rangle_Q, \dots, \langle g_n \rangle_Q]  |h_Q| \Big \|_{\calT_1} \Big \|_{L^{r}} \\
& \le \| g_1 \|_{L^\infty(X_1)}
\E \Big \| \sum_{\substack{Q \in \calD(Q_0) \\ \pi_\calS Q =S}}
\varepsilon_Q a_Q [ \langle g_2 \rangle_Q, \dots, \langle g_n \rangle_Q]  |h_Q|  \Big \|_{L^{r}(\calT_1)}.
\end{split}
\end{equation*}

Because $\calT_1$  is assumed to be a UMD space we can repeat the above estimate with $g_2$ in place of $g_1$.
In course of doing so,  after applying Stein's inequality, we use the estimate
\begin{equation*}
\begin{split}
\Big\|  \sum_{\substack{Q \in \calD(Q_0) \\ \pi_\calS Q =S}}
& \varepsilon_Q a_Q [ g_2 (x), \langle g_3 \rangle_Q, \dots, \langle g_n \rangle_Q]  |h_Q(x)| \Big \|_{\calT_1} \\
& \le| g_2(x)|_{X_2} \Big \| \sum_{\substack{Q \in \calD(Q_0) \\ \pi_\calS Q =S}}
\varepsilon_Q a_Q [ \langle g_3 \rangle_Q, \dots, \langle g_n \rangle_Q]  |h_Q(x)| \Big \|_{\calT_2},
\end{split}
\end{equation*}
see \eqref{eq:Iteration1} and \eqref{eq:Iteration2}.
Iterating this we arrive at
\begin{equation*}
RHS\eqref{eq:ReadyForStein}
\lesssim \prod_{m=1}^n \| g_m \|_{L^\infty(X_m)}
\E \Big \| \sum_{\substack{Q \in \calD(Q_0) \\ \pi_\calS Q =S}}
\varepsilon_Q a_Q  |h_Q|  \Big \|_{L^{r}(\calT_n)}.
\end{equation*}
Finally, the  $\BMO$ assumption gives that
$$
\E \Big \| \sum_{\substack{Q \in \calD(Q_0) \\ \pi_\calS Q =S}}
\varepsilon_Q a_Q  |h_Q|  \Big \|_{L^{r}(\calT_n)}
\lesssim  \| a \|_{\BMO_{\calD, r}(\calT_n)} |S|^{1/{r}} = |S|^{1/{r}}.
$$
This concludes the proof of \eqref{eq:InftyTest}, and hence of \eqref{eq:bddParBanach}.
Lemma \ref{lem:SparseModel} now gives \eqref{SparseForPara}. It remains to recall the John--Nirenberg inequality \eqref{eq:JN}.
\end{proof}

\section{The representation theorem}\label{sec:RepThm}

Let $n \ge 2$ and let $X_1, \dots, X_n,Y_{n+1}$ be $\UMD$-spaces. With respect to these spaces, suppose that
$T$ is an $n$-linear operator-valued SIO with a basic kernel $K$.
In addition, $T^{m*}$ denotes the $m$th adjoint of $T$ for $m \in \{1,\ldots,n\}$ -- that is, there are $n$-linear operator-valued SIOs  $T^{m*}$ satisfying the identity
$$
\langle T(f_1, \dots, f_n), f_{n+1} \rangle
= \langle T^{m*}(f_1, \dots, f_{m-1}, f_{n+1}, f_{m+1}, \dots, f_n), f_m \rangle.
$$

Assume that there exist $\UMD$ subspaces $\calT_k \subset \calL(\prod_{m=1}^kX_m,Y_{n+1})$
so that for all dyadic lattices $\calD$ the sequence $[T1]_{\calD} := (\langle T1, h_Q\rangle)_{Q \in \calD}$
satisfies the $\UMD$ subspace condition
(see Section \ref{sec:ParaBdd}) with these spaces. Recall that the operators $\langle T1, h_Q\rangle$ were defined in Section
\ref{sec:BMOT1}.
We abbreviate this by saying that $T1$ satisfies the $\UMD$ subspace condition.
Let $m \in \{1, \dots, n\}$. Likewise, we say that $T^{m*}1$ satisfies the $\UMD$ subspace condition
if the corresponding $\UMD$ subspaces $\calT_k^m$ exist. Notice that here the spaces $X_m$ change places,
so that for instance
$$
\calT_n^m \subset
\calL\Big(\prod_{k=1}^{m-1} X_k \times X_{n+1} \times \prod_{k=m+1}^{n} X_k , X_m^*\Big).
$$

If $\phi_m \colon \R^d \to \C$, $m = 1, \ldots, n+1$, are bounded and compactly supported and $\Phi = (\phi_1, \ldots, \phi_n)$, then
we can define the $n$-linear operator
$$
\langle T\Phi, \phi_{n+1}\rangle \colon \prod_{m=1}^n X_m \to Y_{n+1}
$$
by setting
\begin{equation}\label{eq:defnTpair}
\langle T \Phi, \phi_{n+1} \rangle [e_1, \dots, e_n]
:=\big \langle T(\phi_1e_1, \dots, \phi_ne_n), \phi_{n+1} \big \rangle, \qquad e_m \in X_m.
\end{equation}
We define the following collection of $n$-linear operators $\prod_{m=1}^n X_m \to Y_{n+1}$ by setting
$$
\calC_{\weak}(T) = \{ |Q|^{-1}\langle T(1_Q, \ldots, 1_Q), 1_Q \rangle \colon Q \subset \R^d \textup{ is a cube}\}.
$$
\begin{defn}\label{defn:T1Cond}
Let $n \ge 2$, $X_1, \dots, X_n,Y_{n+1}$ be $\UMD$ spaces and $X_{n+1}=Y_{n+1}^*$.
With respect to these spaces, suppose that
$T$ is an $n$-linear operator-valued SIO with a basic kernel $K$.
Suppose that $ \varpi \colon X_1\times  \dots \times X_{n+1} \to \C$ is an $n+1$-linear contraction  as in \eqref{eq:Contr}.
We say that $T$ satisfies $T1$ type testing conditions if:
\begin{enumerate}
\item We have
$$
\|K\|_{\CZ_{\alpha}, \varpi} := \calR_{\varpi}(\calC_{\CZ,\alpha}(K)) < \infty.
$$
\item We have
$$
\|T\|_{\WBP, \varpi} :=  \calR_{\varpi}(\calC_{\weak}(T)) < \infty.
$$
\item For all $m \in \{0,1, \dots, n\}$ the $\UMD$ subspace condition for $T^{m*}1$ holds, and
there exist exponents $r_0,r_1, \dots, r_n \in (0, \infty)$
so that
$$
\| T^{m*}1\|_{\BMO_{r_m}(\calT^m_n)}
:=\sup_{\calD} \| (\langle T^{m*}1, h_Q\rangle)_{Q \in \calD} \|_{\BMO_{\calD,r_m}(\calT^m_n)}<\infty,
$$
where the supremum is over dyadic lattices on $\R^d$.
\end{enumerate}
\end{defn}

\subsubsection*{Good and bad cubes}
Recall the random dyadic lattices from Section \ref{sec:randomlattice}.
We introduce the good and bad dyadic cubes of Nazarov--Treil--Volberg \cite{NTV}.
Let $\gamma \in (0,1)$ and $r \in \{1,2,3,\dots\}$.
We say that a cube $Q \in \calD$, where $\calD$ is a dyadic lattice, is $(\gamma,r)$-\emph{good} if for all $R \in \calD$ with $\ell(R) \ge 2^r \ell(Q)$ there holds that
$$
d(Q,\partial R) > \ell(Q)^\gamma \ell(R)^{1-\gamma},
$$
where $\partial R$ is the boundary of $R$. If $Q$ is not $(\gamma,r)$-good, we say that it is $(\gamma,r)$-\emph{bad}.
Let $Q \in \calD_0$ (where $\calD_0$ is the standard dyadic grid) and define the probability
$$
\calP_{\bad}(\gamma,r)=\bbP(\{ \omega \in \Omega \colon Q+\omega \text{ is } (\gamma,r)\text{-bad}\}).
$$
This probability is independent of the cube $Q \in \calD_0$ and $\calP_{\bad}(\gamma,r) \to 0$, as $r \to \infty$.
In what follows we make the explicit choice $\gamma(dn+\alpha)=\alpha/2$ and then fix $r$ large enough -- at least
so large that $\calP_{\good} = \calP_{\good}(\gamma, r) = 1 - \calP_{\bad}(\gamma, r) > 0$, and that certain calculations
below are legitimate. Now that $\gamma$ and $r$ are fixed we simply write $\calD_{\good}$ and $\calD_{\bad}$ for the good and bad cubes of a given lattice $\calD$.

\begin{thm}\label{thm:Rep}
Let $n \ge 2$ and let  $X_1, \dots, X_n, Y_{n+1}$ be $\UMD$ spaces. Denote $X_{n+1}=Y_{n+1}^*$.
Suppose that $(X_1, \dots, X_{n+1})$ satisfies the $\RMF_{\varpi}$ property with some $\varpi$,
as in Section \ref{sec:RMF}.
With respect to these spaces, suppose that
$T$ is an $n$-linear operator-valued SIO with a basic kernel $K$
as in Section \ref{sec:MultiSing}. Suppose that $T$ satisfies the $T1$ type testing conditions of Definition \ref{defn:T1Cond}.

Let $f_m\colon \R^d \to X_{m}$, $m=1,\ldots, n+1$, be compactly supported and bounded functions. Then we have the representation
\begin{equation*}
\begin{split}
\langle &T(f_1,  \dots, f_n),f_{n+1} \rangle \\
&= C\Big[\|K\|_{\CZ_{\alpha}, \varpi} + \|T\|_{\WBP, \varpi} + \sum_{m=0}^n \|T^{m*}1\|_{\BMO_{r_m}(\calT^m_n)}\Big] \\
 &\times \bigg[  \E_\omega \sum_{k \in \Z^{n+1}_+} 2^{-\alpha\max k_m/2} \sum_u
\langle S^k_{\omega,u} (f_1, \dots, f_n),f_{n+1} \rangle
+\sum_{m=0}^n \langle \pi^{m*}_{\omega, T^{m*}1} (f_1, \dots, f_n),f_{n+1} \rangle \bigg].
\end{split}
\end{equation*}
Here  $C \lesssim 1$ and the sum over $u$ is a finite summation.
The operators $S^k_{\omega,u}$ are dyadic shifts of complexity $k$ defined in the grid $\calD_{\omega}$ and satisfy
$\calR_{\varpi}(\calC(S^k_{\omega,u})) \lesssim 1$.
The operator $\pi_{\omega, T^{m*}1}$ is the paraproduct $\pi_{\calD_{\omega}, [T^{m*}1]_{\calD_{\omega, \good}}}$
related to $[T^{m*}1]_{\calD_{\omega, \good}}$.
\end{thm}

Using Theorem \ref{thm:MultiShifts}, Theorem \ref{thm:MultiPara} and \eqref{eq:consequenceofsparse} we get:
\begin{thm} \label{cor:main}
There exist dyadic grids $\mathcal{D}_i$, $i = 1, \ldots, 3^d$, with the following property. Let $\eta \in (0,1)$.
Let $X_1, \dots, X_n$ and $Y_{n+1}$ be Banach spaces and $f_m\colon \R^d \to X_{m}$, $m=1,\ldots, n+1$, be compactly supported and bounded functions. Then for some $i$ there exists an $\eta$-sparse collection $\calS = \calS((f_m), \eta)\subset \calD_i$ with the following property.

If $T$ is an operator-valued $n$-linear singular integral satisfying the $T1$ type testing conditions, then
\begin{equation}\label{SparseForT}
\begin{split}
|\langle T(f_1, \ldots, f_n), f_{n+1}\rangle| \lesssim_{\eta}
\Big[\|K\|_{\CZ_{\alpha}, \varpi} + \|T\|_{\WBP, \varpi} + \sum_{m=0}^n& \|T^{m*}1\|_{\BMO_{r_m}(\calT^m_n)}\Big] \\
&\times \sum_{Q\in\mathcal{S}}|Q|\prod_{m=1}^{n+1}\bla |f_m|_{X_m} \bra_Q.
\end{split}
\end{equation}
In particular, if $p_1, \dots, p_n \in (1, \infty]$, $q_{n+1} \in (1/n, \infty)$ and $\sum_{m=1}^n 1/p_m=1/q_{n+1}$,
then
\begin{equation}\label{eq:TBdd}
\begin{split}
\| T(f_1, \dots, f_n) \|_{L^{q_{n+1}}(Y_{n+1})}
\lesssim \Big[\|K\|_{\CZ_{\alpha}, \varpi} + \|T\|_{\WBP, \varpi} + \sum_{m=0}^n \|T^{m*}&1\|_{\BMO_{r_m}(\calT^m_n)}\Big] \\
&\times \prod_{m=1}^n \| f_m\|_{L^{p_m}(X_m)}.
\end{split}
\end{equation}

\end{thm}

\begin{proof}[Proof of Theorem \ref{thm:Rep}]
We show the proof under the additional assumption that $T$
is a priori bounded, say from $\prod_{m=1}^n L^{n+1}(X_m)$ to $L^{(n+1)/n}(Y_{n+1})$.
At the end we comment why this is enough. With this assumption, all the steps in this subsection can be made rigorous.

For every $m \in \{1, \dots, n+1\}$ suppose $f_m$ is a $X_m$-valued, bounded and compactly supported strongly measurable function.
Write $F =(f_1, \dots, f_n)$. We start considering the pairing $\langle T[F], f_{n+1}\rangle$.

\subsection*{Multilinear reduction to good cubes}
Fix for the moment a random parameter $\omega \in \Omega$.
Writing the functions as $f_m= \sum_{Q_m \in \calD_{\omega}}\Delta_{Q_m} f_m$ we have
\begin{equation*}
\begin{split}
 \langle T(f_1, & \dots, f_n), f_{n+1} \rangle \\
& = \sum_{Q_1, \dots, Q_{n+1} \in \calD_{\omega}}
 \big\langle T(\Delta_{Q_1}f_1, \dots, \Delta_{Q_n}f_n), \Delta_{Q_{n+1}}f_{n+1} \big\rangle \\
 & =\sum_{m=1}^{n+1}
 \sum_{\substack{Q_1, \dots, Q_{n+1} \in \calD_{\omega} \\ \ell(Q_1), \dots,\ell(Q_{m-1})> \ell(Q_m) \\ \ell(Q_{m+1}), \dots, \ell(Q_{n+1}) \ge \ell(Q_m)}}
 \big \langle T(\Delta_{Q_1}f_1, \dots, \Delta_{Q_n}f_n), \Delta_{Q_{n+1}}f_{n+1} \big\rangle.
 \end{split}
 \end{equation*}
For $k \in \Z$ there holds that
 $$
 \sum_{\substack{Q \in \calD_{\omega} \\ \ell(Q)> 2^k}} \Delta_Q g
 =E_{\omega, 2^k} g, \qquad E_{\omega, 2^k}g := \sum_{\substack{Q \in \calD_{\omega} \\ \ell(Q)=2^k}} \langle g \rangle_Q 1_Q.
 $$
Thus, for $m \in \{1, \dots,n\}$ we have
\begin{equation*}
\begin{split}
& \sum_{\substack{Q_1, \dots, Q_{n+1} \in \calD_{\omega} \\ \ell(Q_1), \dots,\ell(Q_{m-1})> \ell(Q_m) \\ \ell(Q_{m+1}), \dots, \ell(Q_{n+1}) \ge \ell(Q_m)}}
 \big \langle T(\Delta_{Q_1}f_1, \dots, \Delta_{Q_n}f_n), \Delta_{Q_{n+1}}f_{n+1} \big\rangle
= \sum_{Q \in \calD_{\omega}} \tilde{\Lambda}_m(Q),
\end{split}
\end{equation*}
where $\tilde{\Lambda}_m(Q)$ is defined, for $Q \in \calD_{\omega}$, to be
\begin{equation*}
\begin{split}
\Big\langle T^{m*}\big(E_{\omega, \ell(Q)}f_1, \dots, E_{\omega, \ell(Q)} f_{m-1},
 E_{\omega,\frac{\ell(Q)}{2}} f_{n+1}, E_{\omega,\frac{\ell(Q)}{2}} f_{m+1}, \dots, E_{\omega,\frac{\ell(Q)}{2}} f_{n} \big), \Delta_{Q}f_{m}\Big\rangle.
\end{split}
 \end{equation*}
Similarly, for $m=n+1$ we have
 \begin{equation*}
\sum_{\substack{Q_1, \dots, Q_{n+1} \in \calD_{\omega} \\ \ell(Q_1), \dots,\ell(Q_{n})> \ell(Q_{n+1})}}
 \big \langle T(\Delta_{Q_1}f_1, \dots, \Delta_{Q_n}f_n), \Delta_{Q_{n+1}}f_{n+1} \big\rangle
  =: \sum_{Q \in \calD_{\omega}}  \Lambda_{n+1}(Q),
\end{equation*}
where
\begin{equation}\label{eq:Lambdan}
\Lambda_{n+1}(Q)
:=\big\langle T\big(E_{\omega, \ell(Q)}f_1, \dots, E_{\omega,\ell(Q)} f_{n}\big), \Delta_{Q}f_{n+1} \big\rangle, \qquad Q \in \calD_{\omega}.
\end{equation}

Recall that $\calD_{\omega}$ is the lattice $\{Q + \omega \colon Q \in \calD_0\}$.
Hence, the average over $\omega \in \Omega$ of $\sum_{Q \in \calD_{\omega}}  \Lambda_{n+1}(Q)$ can be written as
\begin{equation*}
\begin{split}
\E_\omega \sum_{Q \in \calD_0}  \Lambda_{n+1}(Q+ \omega)
&= \frac{1}{\calP_{\good}}\sum_{Q \in \calD_0} \E_\omega [1_{\good} (Q + \omega)]
\E_\omega \Lambda_{n+1}(Q+ \omega) \\
&=\frac{1}{\calP_{\good}}\E_\omega \sum_{Q \in \calD_0} 1_{\good} (Q + \omega)
\Lambda_{n+1}(Q+ \omega) \\
&= \frac{1}{\calP_{\good}}\E_\omega \sum_{Q \in \calD_{\omega, \good}}  \Lambda_{n+1}(Q),
\end{split}
\end{equation*}
where independence of the functions
$
\omega \mapsto 1_{\good} (Q + \omega)
$
and
$
\omega \mapsto \Lambda_{n+1}(Q + \omega)
$
was used in the second identity. Since the same argument can be clearly made for every $ \sum_{Q \in \calD_{\omega}} \tilde{\Lambda}_m(Q)$, we have shown that
\begin{equation}\label{eq:MainSymTogether}
\langle T[F], f_{n+1} \rangle
= \frac{1}{\calP_{\good}}\E_\omega \Big[\sum_{m=1}^{n} \sum_{Q \in \calD_{\omega, \good}} \tilde{\Lambda}_m(Q)
+\sum_{Q \in \calD_{\omega, \good}} \Lambda_{n+1}(Q)\Big].
\end{equation}

\subsection*{Expansion back to martingale differences}\label{sec:BackToDifferences}
Now that the probabilistic reduction is done, we fix one $\omega \in \Omega$ and suppress $\omega$ from the notation;
all the dyadic concepts are with respect to the lattice $\calD := \calD_{\omega}$.
Let first $m \in \{1, \dots, n\}$ and $Q \in \calD$, and consider the pairing
$\tilde{\Lambda}_m(Q)$. Just for notational  convenience define for the moment
\begin{equation}\label{eq:ReIndexed}
g_j^m:=
\begin{cases}
f_j, \quad j \in \{1, \dots, n\} \setminus \{m\}, \\
f_{n+1}, \quad j=m, \\
f_m, \quad j=n+1.
\end{cases}
\end{equation}
By writing $E_{\frac{\ell(Q)}{2}} g_{m}^m =D_{\ell(Q)} g^m_{m}+E_{\ell(Q)} g_{m}^m$, where $D_{2^k}g = \sum_{Q\colon \ell(Q) = 2^k} \Delta_Q g$,
we have
\begin{equation*}
\begin{split}
\tilde{\Lambda}_m(Q)
&=\Big\langle T^{m*}\big(E_{\ell(Q)}g^m_1, \dots, E_{\ell(Q)} g^m_{m-1}, D_{\ell(Q)} g^m_{m}, E_{\frac{\ell(Q)}{2}} g^m_{m+1}, \dots, E_{\frac{\ell(Q)}{2}} g^m_{n} \big), \Delta_{Q}g_{n+1}^m\Big\rangle \\
&+\Big\langle T^{m*}\big(E_{\ell(Q)}g_1^m, \dots, E_{\ell(Q)} g_{m}^m, E_{\frac{\ell(Q)}{2}} g_{m+1}^m, \dots, E_{\frac{\ell(Q)}{2}} g_{n}^m \big), \Delta_{Q}g_{n+1}^m\Big\rangle.
\end{split}
\end{equation*}
Continuing in the same way with the second term on the right hand side,
it is seen that
\begin{equation*}
\begin{split}
\tilde{\Lambda}_m(Q)
&=\sum_{j=m}^n
\Big\langle T^{m*}\big(E_{\ell(Q)}g^m_1, \dots, E_{\ell(Q)} g^m_{j-1}, D_{\ell(Q)} g^m_{j}, E_{\frac{\ell(Q)}{2}} g^m_{j+1}, \dots, E_{\frac{\ell(Q)}{2}} g^m_{n} \big), \Delta_{Q}g_{n+1}^m\Big\rangle \\
& +\Lambda_m(Q),
\end{split}
\end{equation*}
where we have defined
\begin{equation}\label{eq:OtherMainSym}
\Lambda_m(Q)
:= \Big\langle T^{m*}\big(E_{\ell(Q)}g^m_1, \dots, E_{\ell(Q)} g^m_{n}\big), \Delta_{Q}g_{n+1}^m\Big\rangle.
\end{equation}

The terms $\Lambda_m(Q)$, $m \in \{1, \dots, n\}$, are completely symmetric with the term $\Lambda_{n+1}(Q)$ in \eqref{eq:Lambdan}.
Hence,  we will concentrate on finding the model operator structure for the sum
\begin{equation}\label{eq:MainTerm}
\sum_{Q \in \calD_{\good}} \Lambda_{n+1}(Q).
\end{equation}

The terms
\begin{equation}\label{eq:ContainsDiag}
\sum_{Q \in \calD_{\good}}
\Big\langle T^{m*}\big(E_{\ell(Q)}g^m_1, \dots, E_{\ell(Q)} g^m_{j-1}, D_{\ell(Q)} g^m_{j}, E_{\frac{\ell(Q)}{2}} g^m_{j+1}, \dots, E_{\frac{\ell(Q)}{2}} g^m_{n} \big), \Delta_{Q}g_{n+1}^m\Big\rangle,
\end{equation}
where $m \in \{1, \dots, n\}$ and $j \in \{m, \dots, n\}$, will be handled separately in Section \ref{sec:Diagonal}. Apart from a certain diagonal part,
the shift structure for these will follow from our arguments concerning \eqref{eq:MainTerm}.

Now we start to consider the term \eqref{eq:MainTerm}.
Fix the cube $Q \in \calD_{\good}$ for the moment.
Since by a priori boundedness there holds
$$
T\big(E_{2^k}f_1, \dots, E_{2^k} f_{n}\big) \to 0, \quad \text{ as } k \to \infty,
$$
we have
\begin{equation*}
\begin{split}
\big\langle T\big(E_{\ell(Q)}f_1, \dots, E_{\ell(Q)} f_{n}\big), \Delta_{Q}f_{n+1} \big\rangle
= \sum_{\substack{k \in \Z \\ 2^k \ge \ell(Q)}}&\Big[\big\langle T\big(E_{2^k}f_1, \dots, E_{2^k} f_{n}\big), \Delta_{Q}f_{n+1} \big\rangle \\
&-\big\langle T\big(E_{2^{k+1}}f_1, \dots, E_{2^{k+1}} f_{n}\big), \Delta_{Q}f_{n+1} \big\rangle\Big].
\end{split}
\end{equation*}
Let $k \in \Z$ be such that $2^k \ge \ell(Q)$. Then
\begin{equation*}
\begin{split}
\big\langle & T\big(E_{2^k}f_1, \dots, E_{2^k} f_{n}\big), \Delta_{Q}f_{n+1} \big\rangle
-\big\langle T\big(E_{2^{k+1}}f_1, \dots, E_{2^{k+1}} f_{n}\big), \Delta_{Q}f_{n+1} \big\rangle \\
& = \sum_{i=1}^n \big\langle T\big(E_{2^{k+1}}f_1, \dots, E_{2^{k+1}} f_{i-1},D_{2^{k+1}} f_{i},E_{2^{k}} f_{i+1}, \dots, E_{2^{k}} f_{n}\big), \Delta_{Q}f_{n+1} \big\rangle.
\end{split}
\end{equation*}
For $i \in \{1, \dots, n\}$ the corresponding term in this sum can further be written as
$$
\mathop{\mathop{\sum_{Q_1, \dots,Q_n \in \calD}}_{\ell(Q_1) = \cdots = \ell(Q_i) = 2^{k+1}}}_{\ell(Q_{i+1}) = \cdots = \ell(Q_{n}) = 2^k}
\big\langle T\big(E_{Q_1}f_1, \dots, E_{Q_{i-1}} f_{i-1},\Delta_{Q_i} f_{i},E_{Q_{i+1}} f_{i+1}, \dots, E_{Q_n} f_{n}\big), \Delta_{Q}f_{n+1} \big\rangle.
$$

Let us agree on the following conventions. Let $\scrD_i$, $i = 1, \ldots, n$, be defined by
\begin{align*}
\scrD_i := \big\{R:=Q_1 \times \dots \times Q_n &\colon Q_1, \dots, Q_n \in \calD, \\
&\ell(Q_1)= \cdots = \ell(Q_i) = 2\ell(Q_{i+1}) = \cdots = 2\ell(Q_n)\big\}.
\end{align*}
Suppose $R=Q_1\times \dots \times Q_n \in \scrD_i$ and $Q \in \calD$.
Define $\ell(R) := \ell(Q_1)$ and
$$
d(Q,R) = \max_m d(Q, Q_m).
$$
Set $V^i_R F$ to be the $n$-tuple of functions
$$
V^i_R F := (E_{Q_1}f_1, \dots, E_{Q_{i-1}} f_{i-1},\Delta_{Q_i} f_{i},E_{Q_{i+1}} f_{i+1}, \dots, E_{Q_n} f_{n}).
$$

Using the above splitting we have
\begin{equation}\label{eq:ReadyToSplit}
\sum_{Q \in \calD_{\good}} \Lambda_{n+1}(Q)
= \sum_{i=1}^n \sum_{Q \in \calD_{\good}} \mathop{\sum_{R \in \scrD_i}}_{\ell(R) > \ell(Q)}
\langle T[V^i_R F], \Delta_Q f_{n+1} \rangle.
\end{equation}
Most of the time we will consider each $i$ separately.
However, related to every $i$ there will be a paraproduct type term.
These will be summed together in Section \ref{sec:StepIV} giving one simple paraproduct.

Let us rewrite the pairings $\langle T[V^i_R F], \Delta_Q f_{n+1} \rangle$
using Haar functions. Expanding
$$
\Delta_Q f_{n+1}=  \langle f_{n+1}, h_Q \rangle h_Q
$$
there holds
$$
\langle T[V^i_R F], \Delta_Q f_{n+1} \rangle
= \langle \langle T[V^i_R F], h_Q \rangle, \langle f_{n+1}, h_Q \rangle  \rangle.
$$
Related to the set $R=Q_1 \times \dots \times Q_n \in \scrD_i$, define the $n$-tuple
\begin{equation}\label{eq:h_RAsTuple}
h_{R, i} = (h_{Q_1}^0,  \dots, h_{Q_{i-1}}^0, h_{Q_i}, h_{Q_{i+1}}^0, \dots, h_{Q_{n}}^0).
\end{equation}
Using $h_{R,i}$ we can write
$$
\langle T[V^i_R F], h_Q \rangle
=\langle Th_{R,i}, h_Q \rangle [\langle F,h_{R,i} \rangle],
$$
where $\langle Th_{R,i}, h_Q \rangle$ is the natural operator defined using \eqref{eq:defnTpair} and
$$
\langle F,h_{R,i} \rangle:=(\langle f_1, h_{Q_1}^0 \rangle,  \dots, \langle f_{i-1}, h_{Q_{i-1}}^0\rangle,
\langle f_i, h_{Q_i} \rangle, \langle f_{i+1},h_{Q_{i+1}}^0\rangle, \dots, \langle f_n, h_{Q_{n}}^0\rangle).
$$
Altogether, we have
$$
\langle T[V^i_R F], \Delta_Q f_{n+1} \rangle
= \langle \langle Th_{R,i}, h_Q \rangle [\langle F,h_{R,i} \rangle],\langle f_{n+1}, h_Q \rangle \rangle.
$$

We now fix one $i$ in \eqref{eq:ReadyToSplit} and start to study the related term
\begin{equation}\label{eq:Fixedi}
\sum_{Q \in \calD_{\good}} \mathop{\sum_{R \in \scrD_i}}_{\ell(R) > \ell(Q)}
\langle \langle Th_{R,i}, h_Q \rangle [\langle F,h_{R,i} \rangle],\langle f_{n+1}, h_Q \rangle \rangle.
\end{equation}

\subsection{Step I: separated cubes}
Here the part
\begin{align*}
\sigma_1^i &=  \sum_{Q \in \calD_{\good}} \mathop{\mathop{\sum_{R \in \scrD_i}}_{\ell(R) > \ell(Q)}}_{d(Q, R) > \ell(Q)^{\gamma}(\ell(R)/2)^{1-\gamma}}
\langle \langle Th_{R,i}, h_Q \rangle [\langle F,h_{R,i} \rangle],\langle f_{n+1}, h_Q \rangle \rangle
\end{align*}
of \eqref{eq:Fixedi} is considered.
We begin with the following lemma on the existence of nice common parents. A short proof is given, albeit it is morally the same as in the case $n=1$ in \cite{Hy3}.
If $R= Q_1 \times \dots \times Q_n \in \scrD_i$ and $Q \in \calD$ are such that there exists a cube $K \in \calD$ so that
$Q,Q_1, \dots,Q_n \subset K$, then the minimal such $K$ is denoted by $Q \vee R $.

\begin{lem}\label{lem:ComPar}
Suppose $R = Q_1 \times \cdots \times Q_n \in \scrD_i$ and $Q \in \calD_{\good}$ are such that there holds $d(Q, R) > \ell(Q)^{\gamma}(\ell(R)/2)^{1-\gamma}$.
Then there exists  $K \in \calD$ so that $Q_1 \cup \cdots \cup Q_n \cup Q \subset K$ and
$$
d(Q,R) \gtrsim \ell(Q)^{\gamma}\ell(K)^{1-\gamma}.
$$
\end{lem}

\begin{proof}
Let $K \in \calD$ be the minimal parent of $Q$ for which both of the following two conditions hold:
\begin{itemize}
\item $\ell(K) \ge 2^r \ell(Q)$;
\item $d(Q,R) \le \ell(Q)^{\gamma}\ell(K)^{1-\gamma}$.
\end{itemize}
If we had that $Q_m \subset K^c$ for some $m$, we would get by the goodness of the cube $Q$ that
$$
\ell(Q)^{\gamma}\ell(K)^{1-\gamma} < d(Q,R) \le \ell(Q)^{\gamma}\ell(K)^{1-\gamma},
$$
which is a contradiction. Moreover, we have
$$
\ell(Q)^{\gamma}(\ell(R)/2)^{1-\gamma} <  d(Q,R) \le \ell(Q)^{\gamma}\ell(K)^{1-\gamma}
$$
implying that $\ell(K) \ge \ell(R)$. Thus, there holds $Q_1 \cup \cdots \cup Q_n \cup Q \subset K$.

It remains to note that the estimate $d(Q,R) \gtrsim \ell(Q)^{\gamma}\ell(K)^{1-\gamma}$ is a trivial consequence of the minimality of $K$. There is something
to check only if $\ell(K) = 2^r\ell(Q)$. But then $\ell(K) \lesssim \ell(R)$ and so
$$
d(Q,R) \gtrsim \ell(Q)^{\gamma}\ell(R)^{1-\gamma} \gtrsim \ell(Q)^{\gamma}\ell(K)^{1-\gamma}.
$$
\end{proof}

We use the common parents to organize $\sigma^i_1$ as
\begin{equation}\label{eq:SepOrga}
\sum_{Q \in \calD_{\good}} \mathop{\mathop{\sum_{R \in \scrD_i}}_{\ell(R) > \ell(Q)}}_{d(Q, R) > \ell(Q)^{\gamma}(\ell(R)/2)^{1-\gamma}}
= \sum_{j_1=2}^{\infty} \sum_{j_2 = 1}^{j_1-1}  \sum_{K \in \calD} \mathop{\mathop{\mathop{\sum_{Q \in \calD_{\good},R \in \scrD_i}}_{d(Q, R) > \ell(Q)^{\gamma}(\ell(R)/2)^{1-\gamma}}}_{2^{j_1}\ell(Q) = 2^{j_2}\ell(R) = \ell(K)}}_{Q \vee R =K}.
\end{equation}
Suppose $j_1$, $j_2$, $K$, $Q$ and $R$ are as in the above sum.
We will show that for some large enough constant $C$ there holds that
\begin{equation}\label{eq:SepShiftEst}
\frac{2^{\alpha j_1/2}}{C} \frac{|K|^n}{\prod_{m=1}^n|Q_m|^{1/2}|Q|^{1/2}} \langle Th_{R,i}, h_Q \rangle
\in \calA(\calC_{\CZ,\alpha}(K)).
\end{equation}
Here we are using the notation defined in Lemma \ref{lem:AveRbdd}.
This then implies that for fixed $j_1$ and $j_2$ the inner double sum in \eqref{eq:SepOrga} multiplied by
$$
\frac{2^{\alpha j_1/2}}{C\Big[\|K\|_{\CZ_{\alpha}, \varpi} + \|T\|_{\WBP, \varpi} + \sum_{m=0}^n \|T^{m*}1\|_{\BMO_{r_m}(\calT^m_n)}\Big]}
$$
is an operator-valued $n$-linear shift as in the representation theorem acting on $f_1, \dots, f_n$ and paired with $f_{n+1}$.
The complexity of this shift is $k=(k_1, \dots, k_{n+1})$, where
$k_1=\dots =k_i=2k_{i+1}=\dots=2k_{n}=j_2$ and $k_{n+1}=j_1$.
Therefore,
\eqref{eq:SepOrga} is of the right form for the representation theorem.

We turn to prove \eqref{eq:SepShiftEst}. We write $R=Q_1 \times \cdots \times Q_n$,
and suppose first that $d(Q,R) \le C_d \ell(Q)$, where $C_d$ is a dimensional constant.
From the proof of Lemma \ref{lem:ComPar} it is clear that, if $r$ is fixed large enough, then $K \subset Q^{(r)}$ implying that $\ell(K) \sim \ell(Q)$ and $j_1 \le r \lesssim 1$.

 If $x \in \R^d$ and $y=(y_1, \dots, y_n) \in \R^{dn}$ are such that $(x,y_1,\dots,y_n) \in \R^{d(n+1)} \setminus \Delta$,
where $\Delta$ consists of the diagonal points $(x,\dots, x)$,  define
\begin{equation}\label{eq:lambda}
\lambda(x,y)
= \Big(\sum_{m=1}^n |x-y_m|\Big)^{-dn}.
\end{equation}
By slight abuse of notation we write
$$
h_{R, i}(y) = h_{Q_1}^0(y_1)  \cdots h_{Q_{i-1}}^0(y_{i-1}) h_{Q_i}(y_i) h_{Q_{i+1}}^0(y_{i+1}) \cdots h_{Q_{n}}^0(y_n).
$$
In \eqref{eq:h_RAsTuple} we gave a different definition for $h_{R,i}$, but it should be clear from the context which one we use (see the next equation, for instance).

We  write out $\langle Th_{R,i}, h_Q \rangle$ as
\begin{equation}\label{eq:WriteOut}
\begin{split}
\langle Th_{R,i}, h_Q \rangle
= \int_{\R^d} \int_{\R^{dn}} \frac{K(x,y)}{\lambda(x,y)} \lambda(x,y) h_{R,i}(y) h_Q(x) \ud y \ud x.
\end{split}
\end{equation}
Notice that this integral does not make sense as such, but has to be interpreted as in \eqref{eq:PairingOp}.
There holds that $K(x,y)/\lambda(x,y) \in \calC_{\CZ,\alpha}(K)$. Because $d(Q,R) > \ell(Q)$,  there also holds that
\begin{equation*}
\begin{split}
\int_{\R^d} \int_{\R^{dn}}  \frac{|h_{R,i}(y) h_Q(x)|}{\big(\sum_{m=1}^n |x-y_m|\big)^{dn}} \ud y \ud x
\le \frac{|R|^{1/2} |Q|^{1/2}}{\ell(Q)^{dn}} \sim 2^{-\alpha j_1/2} \frac{\prod_{m=1}^n |Q_m|^{1/2} |Q|^{1/2}}{|K|^n}.
\end{split}
\end{equation*}
This together with \eqref{eq:WriteOut} proves \eqref{eq:SepShiftEst} in the case $d(Q,R) \le C_d \ell(Q)$.

Consider then the case $d(Q,R) > C_d \ell(Q)$. If $(c_Q,y_1, \dots, y_n) \not \in \Delta$, define
\begin{equation}\label{eq:lambdaQ}
\lambda_Q(x,y)
= |x-c_Q|^\alpha\Big(\sum_{m=1}^n |c_Q-y_m|\Big)^{-(dn+\alpha)}.
\end{equation}
The zero average of $h_Q$ allows us to write
$$
\langle Th_{R,i}, h_Q \rangle
=\int_{\R^d} \int_{\R^{dn}} \frac{K(x,y)-K(c_Q,y)}{\lambda_Q(x,y)} \lambda_Q(x,y) h_{R,i}(y) h_Q(x) \ud y \ud x.
$$
Since $d(Q,R) > C_d \ell(Q)$,  it holds that here $(K(x,y)-K(c_Q,y))/\lambda_Q(x,y) \in \calC_{\CZ,\alpha}(K)$.
In addition, using  the fact $d(Q,R) \gtrsim \ell(Q)^\gamma \ell(K)^{1-\gamma}$ from Lemma \ref{lem:ComPar} we have the estimate
\begin{equation*}
\int_{\R^d} \int_{\R^{dn}} \frac{|x-c_Q|^\alpha |h_{R,i}(y) h_Q(x)|}{\big(\sum_{m=1}^n |c_Q-y_m|\big)^{dn+\alpha}} \ud y \ud x
\lesssim  \frac{\ell(Q)^\alpha |R|^{1/2}|Q|^{1/2}}{(\ell(Q)^\gamma \ell(K)^{1-\gamma})^{dn+\alpha}}
=  \frac{ |R|^{1/2}|Q|^{1/2}}{2^{\alpha j_1/2}|K|^n},
\end{equation*}
where we recalled that $\gamma(dn+\alpha)=\alpha/2$. This proves \eqref{eq:SepShiftEst} in the case $d(Q,R)> C_d \ell(Q)$.
We are done with Step I.

\subsection{Step II: nearby cubes}
We now look at the sum
$$
\sigma^i_2 = \sum_{Q \in \calD_{\good}} \mathop{\mathop{\mathop{\sum_{R=Q_1 \times \cdots \times Q_n \in \scrD_i}}_{\ell(R) > \ell(Q)}}_{d(Q, R) \le \ell(Q)^{\gamma}(\ell(R)/2)^{1-\gamma}}}_{Q_m \cap Q = \emptyset \textup{ for some } m}
\langle \langle Th_{R,i}, h_Q \rangle [\langle F,h_{R,i} \rangle],\langle f_{n+1}, h_Q \rangle \rangle.
$$
Let $Q \in \calD$ and $R=Q_1 \times \dots \times Q_n \in \scrD_i$ be as in $\sigma^i_2$. Suppose $Q_{m_0}$ is a cube such that
$Q_{m_0} \cap Q = \emptyset$. If $\ell(Q_{m_0}) \ge \ell(Q^{(r)})$, then the goodness of the cube $Q$ implies that
$$
d(Q,R) \ge d(Q,Q_{m_0}) > \ell(Q)^\gamma \ell(Q_{m_0})^{1-\gamma} \ge \ell(Q)^\gamma (\ell(R)/2)^{1-\gamma},
$$
which is a contradiction. Thus, we have $\ell(R) \le 2\ell(Q_{m_0}) \le \ell(Q^{(r)})$.
Suppose $Q_m \cap  Q^{(r)}= \emptyset$ for some $m$. Then
$$
d(Q,R) \ge d(Q,(Q^{(r)})^c) > \ell(Q)^\gamma \ell(Q^{(r)})^{1-\gamma}
> \ell(Q)^\gamma (\ell(R)/2)^{1-\gamma},
$$
which is again a contradiction. We conclude that $Q \vee R \subset Q^{(r)}$.

These observations show that
$$
\sigma^i_2
=
\sum_{j_1 = 1}^r \sum_{j_2 = 0}^{j_1-1} \sum_{K \in \calD} \sum_{\substack{
Q \in \calD_{\good},R \in \scrD_i \\
d(Q, R) \le \ell(Q)^{\gamma}(\ell(R)/2)^{1-\gamma} \\
Q_m \cap Q = \emptyset \textup{ for some } m \\
2^{j_1}\ell(Q) = 2^{j_2}\ell(R) = \ell(K) \\
Q \vee R = K}}
\langle \langle Th_{R,i}, h_Q \rangle [\langle F,h_{R,i} \rangle],\langle f_{n+1}, h_Q \rangle \rangle.
$$
Similarly as in Step I, we need to show that if $j_1$, $j_2$, $Q$, $R$ and $K$ are as in $\sigma^i_2$, then
\begin{equation*}
\frac{2^{\alpha j_1/2}}{C} \frac{|K|^n}{\prod_{m=1}^n|Q_m|^{1/2}|Q|^{1/2}} \langle Th_{R,i}, h_Q \rangle
\in  \calA(\calC_{\CZ,\alpha}(K)).
\end{equation*}

Recall the function $\lambda(x,y)=\big(\sum_{m=1}^n |x-y_m|\big)^{-dn}$ and write
$$\langle Th_{R,i}, h_Q \rangle
= \int_{\R^d} \int_{\R^{dn}} \frac{K(x,y)}{\lambda(x,y)} \lambda(x,y) h_{R,i}(y) h_Q(x) \ud y \ud x.
$$
Here we have $K(x,y)/\lambda(x,y) \in \calC_{\CZ,\alpha}(K)$. Let $m_0 \in \{1, \dots, n\}$ be such that $Q_{m_0} \cap Q =\emptyset$.
Then, using the estimate
$
\int_{\R^d} (c+|x-y|)^{-d-\beta} \ud y
\lesssim_{d,\beta} c^{-\beta},
$
where $c, \beta >0$, we have
\begin{equation}\label{eq:NearbyEst}
\begin{split}
\int_{\R^d} \int_{\R^{dn}}\frac{|h_{R,i}(y) h_Q(x)|}{\big(\sum_{m=1}^n |x-y_m|\big)^{dn}}\ud y \ud x
& \lesssim |R|^{-1/2} |Q|^{-1/2} \int_Q \int_{Q^{(r)} \setminus Q} \frac{1}{|x-y_{m_0}|^d} \ud y_{m_0} \ud x \\
& \lesssim |R|^{-1/2} |Q|^{-1/2} |Q| \sim 2^{-\alpha j_1/2} \frac{|R|^{1/2}|Q|^{1/2}}{|K|^n}.
\end{split}
\end{equation}
This concludes Step II.

\subsection{Step III: error terms}
We start working with the sum
$$
\sigma^i_3 = \sum_{Q \in \calD_{\good}} \mathop{\sum_{R \in \scrD_i}}_{R = (Q^{(k)})^{i} \times (Q^{(k-1)})^{n-i}  \textup{ for some } k \ge 1}
\langle \langle Th_{R,i}, h_Q \rangle [\langle F,h_{R,i} \rangle],\langle f_{n+1}, h_Q \rangle \rangle,
$$
which is what is left after Step I and Step II. Here and in what follows $Q^i = Q \times \cdots \times Q$, where there are $i$ members in the Cartesian product.
First, we define abbreviations related to certain error terms.
Let $Q \in \calD_{\good}$ and $1 \le k \in \Z$.
We will define the function tuple $\Phi_{Q,k,i,j}=(\phi_{Q,k,i,j}^1, \dots, \phi_{Q,k,i,j}^n)$ for every $j = 1, \ldots, n$.
If $j \le i-1$ we set
$$
\phi_{Q,k,i,j}^1 = \cdots = \phi_{Q,k,i,j}^{j-1} \equiv |Q^{(k)}|^{-1/2}, \, \phi_{Q,k,i,j}^j = |Q^{(k)}|^{-1/2} 1_{(Q^{(k)})^c},
$$
$$
\phi_{Q,k,i,j}^{j+1} = \cdots = \phi_{Q,k,i,j}^{i-1} =  |Q^{(k)}|^{-1/2} 1_{Q^{(k)}}, \, \phi_{Q,k,i,j}^{i} = h_{Q^{(k)}},
$$
$$
\phi_{Q,k,i,j}^{i+1} = \cdots = \phi_{Q,k,i,j}^{n} =  |Q^{(k-1)}|^{-1/2} 1_{Q^{(k-1)}}.
$$
If $j=i$ we set
$$
\phi_{Q,k,i,i}^1 = \cdots = \phi_{Q,k,i,i}^{i-1} \equiv |Q^{(k)}|^{-1/2}, \, \phi_{Q,k,i,i}^i = 1_{(Q^{(k-1)})^c}[-h_{Q^{(k)}} + \langle h_{Q^{(k)}} \rangle_{Q^{(k-1)}}],
$$
$$
\phi_{Q,k,i,i}^{i+1} = \cdots = \phi_{Q,k,i,i}^{n} =  |Q^{(k-1)}|^{-1/2} 1_{Q^{(k-1)}}.
$$
Finally, for $j \in \{i+1, \dots,n\}$ we set
$$
\phi_{Q,k,i,j}^1 = \cdots = \phi_{Q,k,i,j}^{i-1} \equiv |Q^{(k)}|^{-1/2}, \, \phi_{Q,k,i,j}^i \equiv \langle h_{Q^{(k)}} \rangle_{Q^{(k-1)}},
$$
$$
\phi_{Q,k,i,j}^{i+1} = \cdots = \phi_{Q,k,i,j}^{j-1} \equiv |Q^{(k-1)}|^{-1/2}, \, \phi_{Q,k,i,j}^{j} =  |Q^{(k-1)}|^{-1/2}  1_{(Q^{(k-1)})^c},
$$
$$
\phi_{Q,k,i,j}^{j+1} = \cdots = \phi_{Q,k,i,j}^{n} =   |Q^{(k-1)}|^{-1/2} 1_{Q^{(k-1)}}.
$$
With the same abuse of notation as with $h_{R,i}$ we set for
$y =(y_1, \dots, y_n) \in \R^{dn}$ that
$$
\Phi_{Q,k,i,j}(y)
=\prod_{m=1}^n\phi_{Q,k,i,j}^m(y_m).
$$

We are ready to move forward. Suppose $Q$ and $R$ are as in $\sigma^i_3$ with $\ell(R)=2^k\ell(Q)$.
We will denote the function $h_{R,i}$ also by $u_{Q,k,i}$. Note that if $y=(y_1, \dots, y_n)$ with $y_m \in Q^{(k-1)}$ for all $m$,
then
$u_{Q,k,i}(y)= \langle u_{Q,k,i}\rangle_{Q^n}$.
With the previous definitions we have the identity
$$
\langle Th_{R,i}, h_Q  \rangle
=\langle u_{Q,k,i} \rangle_{Q^n}\langle T1, h_Q  \rangle
-\sum_{j=1}^n \langle T\Phi_{Q,k,i,j}, h_Q  \rangle,
$$
where we recall that the operators
$\langle T1, h_Q  \rangle$ and $ \langle T\Phi_{Q,k,i,j}, h_Q  \rangle$ are interpreted as in \eqref{eq:PairingOp}.
This gives that $\sigma^i_3= \sigma^i_{3,\pi}-\sum_{j=1}^n\sigma^i_{3,e,j}$, where
\begin{equation} \label{eq:Parai}
\sigma^i_{3,\pi}
:=\sum_{Q \in \calD_{\good} } \sum_{k=1}^\infty
\langle u_{Q,k,i} \rangle_{Q^n} \langle \langle T1, h_Q  \rangle  [\langle F, u_{Q,k,i} \rangle],\langle f_{n+1}, h_Q \rangle \rangle
\end{equation}
and
\begin{equation*}
\sigma^i_{3,e,j}:=
\sum_{Q \in \calD_{\good} } \sum_{k=1}^\infty
\langle \langle T\Phi_{Q,k,i,j}, h_Q  \rangle    [\langle F, u_{Q,k,i} \rangle],\langle f_{n+1}, h_Q \rangle \rangle.
\end{equation*}

The term $\sigma^i_{3,\pi}$ will become part of the paraproduct that is considered in the
next section.
Now we look at the error terms $\sigma^i_{3,e,j}$.
We consider each of them separately, so we fix
$j \in \{1, \dots, n\}$.

The term $\sigma^i_{3,e,j}$ can be written as
$$
\sigma^i_{3,e,j}
= \sum_{k=1}^\infty \sum_{K \in \calD} \sum_{\substack{Q \in \calD_{\good} \\ Q^{(k)}=K}}
\langle \langle T\Phi_{Q,k,i,j}, h_Q  \rangle    [\langle F, u_{Q,k,i} \rangle],\langle f_{n+1}, h_Q \rangle \rangle.
$$
From here it is seen that this produces a series of shifts that satisfy the requirements of the representation theorem
once we have shown that if $k$, $K$ and $Q$ are as in $\sigma^i_{3,e,j}$, then
\begin{equation}\label{eq:ErrShiftEst}
C^{-1} 2^{\alpha k/2} |K|^{n/2}|Q|^{-1/2}\langle T\Phi_{Q,k,i,j}, h_Q  \rangle  \in \calA(\calC_{\CZ,\alpha}(K)).
\end{equation}
Notice that we have the right normalisation since
$$
|K|^{n/2}|Q|^{-1/2} \sim |Q^{(k)}|^{-i/2}|Q^{(k-1)}|^{-(n-i)/2} |Q|^{-1/2}|K|^n.
$$

Recall the functions $\lambda$ and $\lambda_Q$ from \eqref{eq:lambda} and \eqref{eq:lambdaQ}.
Notice that the $j$-coordinate of $\Phi_{Q,k,i,j}$ is supported in the complement of $Q$.
Therefore, using the definition \eqref{eq:PairingOp} of $\langle T\Phi_{Q,k,i,j}, h_Q  \rangle$,
we have that $\langle T\Phi_{Q,k,i,j}, h_Q  \rangle$ is the sum of
\begin{equation}\label{eq:ErrNear}
\int_{\R^d}\int_{(C_dQ)^n}\frac{K(x,y)}{\lambda(x,y)} \lambda(x,y)\Phi_{Q,k,i,j}(y) h_Q(x) \ud y \ud x
\end{equation}
and
\begin{equation}\label{eq:ErrFar}
\int_{\R^d}\int_{((C_dQ)^n)^c}\frac{K(x,y)-K(c_Q,y)}{\lambda_Q(x,y)} \lambda_Q(x,y)\Phi_{Q,k,i,j}(y) h_Q(x) \ud y \ud x.
\end{equation}

Let us first consider the case $k \le r$. Using the pointwise normalisation $|\Phi_{Q,k,i,j}(y)| \lesssim |Q^{(k)}|^{-n/2}
=|K|^{-n/2}$ we get, similarly as in \eqref{eq:NearbyEst}, that
\begin{equation*}
\begin{split}
\int_{\R^d}\int_{(C_dQ)^n}\frac{|\Phi_{Q,k,i,j}(y) h_Q(x)|}{\big(\sum_{m=1}^n |x-y_m|\big)^{dn}}  \ud y \ud x
& \lesssim |K|^{-n/2}|Q|^{-1/2} \int_Q \int_{C_dQ \setminus Q} \frac{1}{ |x-y_j|^{d}} \ud y_j \ud x \\
&\lesssim |K|^{-n/2}|Q|^{1/2}.
\end{split}
\end{equation*}
In the same way, there holds that
\begin{equation*}
\begin{split}
\int_{\R^d}\int_{((C_dQ)^n)^c} \frac{|x-c_Q|^\alpha |\Phi_{Q,k,i,j}(y) h_Q(x)| }{\big(\sum_{m=1}^n |c_Q-y_m|\big)^{dn+\alpha}} \ud y \ud x
 &\lesssim \frac{|Q|^{1/2}}{|K|^{n/2}}  \int_{ Q^c} \frac{\ell(Q)^\alpha}{ |c_Q-y_j|^{d+\alpha}} \ud y_j  \\
&\lesssim |K|^{-n/2}|Q|^{1/2}.
\end{split}
\end{equation*}
These estimates prove \eqref{eq:ErrShiftEst} in the case $k \le r$.

Assume then that $k > r$. The $j$-coordinate of $\Phi_{Q,k,i,j}$ is supported in $(Q^{(k-1)})^c$.
Because $Q$ is a good cube, we have
$Q^{(k-1)} \supset Q^{(r)} \supset C_d Q$, which uses the fact that $r$ is large enough.
Thus, the integral \eqref{eq:ErrNear} is zero. Related to \eqref{eq:ErrFar}  there holds
by the goodness of the cube $Q$ that
\begin{equation*}
\begin{split}
\int_{\R^d}\int_{\R^{dn}} \frac{|x-c_Q|^\alpha |\Phi_{Q,k,i,j}(y) h_Q(x)| }{\big(\sum_{m=1}^n |c_Q-y_m|\big)^{dn+\alpha}} \ud y \ud x
 &\lesssim \frac{|Q|^{1/2}}{|K|^{n/2}}  \int_{(Q^{(k-1)})^c} \frac{\ell(Q)^\alpha}{ |c_Q-y_j|^{d+\alpha}} \ud y_j  \\
&\lesssim \frac{|Q|^{1/2}}{|K|^{n/2}} \frac{\ell(Q)^\alpha}{(\ell(Q)^\gamma \ell(K)^{1-\gamma})^\alpha}.
\end{split}
\end{equation*}
Since $1-\gamma \ge 1/2$, we also get the right geometric decay.
We have proved  \eqref{eq:ErrShiftEst} also in the case $k > r$.

 \subsection{Step IV: paraproduct}\label{sec:StepIV}
We consider  the term
\begin{equation*}
\sigma^i_{3,\pi}
:=\sum_{Q \in \calD_{\good} } \sum_{k=1}^\infty
\langle u_{Q,k,i} \rangle_{Q^n} \langle \langle T1, h_Q  \rangle  [\langle F, u_{Q,k,i} \rangle],\langle f_{n+1}, h_Q \rangle \rangle
\end{equation*}
from \eqref{eq:Parai}. Here we will sum up the corresponding terms $\sigma^i_{3,\pi}$, $i \in \{1, \dots, n\}$, to get one paraproduct.

Recalling the implicit summation over the cancellative Haar functions we have that
\begin{equation*}
\begin{split}
&\quad \langle u_{Q,k,i} \rangle_{Q^n}\langle   T1, h_Q  \rangle[  \langle F, u_{Q,k,i} \rangle] \\
&=\langle T1, h_Q  \rangle [(\langle f_1 \rangle_{Q^{(k)}}, \dots, \langle f_{i-1} \rangle_{Q^{(k)}}, \langle \Delta_{Q^{(k)}} f_i \rangle_{Q^{(k-1)}}
,\langle f_{i+1} \rangle_{Q^{(k-1)}}, \dots, \langle f_{n} \rangle_{Q^{(k-1)}} )].
\end{split}
\end{equation*}
Summing these together over $i \in \{1, \dots, n\}$ yields
$$
\sum_{i=1}^n
\langle u_{Q,k,i} \rangle_{Q^n}\langle   T1, h_Q  \rangle[  \langle F, u_{Q,k,i} \rangle]
= \langle   T1, h_Q  \rangle[ \langle F \rangle_{Q^{(k-1)}}]
-\langle   T1, h_Q  \rangle[ \langle F \rangle_{Q^{(k)}}],
$$
where $\langle F \rangle_Q$ denotes $(\langle f_1 \rangle_Q, \dots, \langle f_n \rangle_Q)$.
Finally, we get the desired paraproduct:
\begin{equation*}
\begin{split}
\sum_{i=1}^n \sigma^i_{3,\pi}
&= \sum_{Q \in \calD_{\good} } \sum_{k=1}^\infty \big\langle
\langle   T1, h_Q  \rangle[ \langle F \rangle_{Q^{(k-1)}}]
-\langle   T1, h_Q  \rangle[ \langle F \rangle_{Q^{(k)}}], \langle f_{n+1}, h_Q \rangle\big\rangle \\
&= \sum_{Q \in \calD_{\good} }\langle \langle   T1, h_Q  \rangle[ \langle F \rangle_{Q}],\langle f_{n+1}, h_Q \rangle \rangle
= \langle \pi_{\calD , [T1]_{\calD_{\good}}}[F], f_{n+1} \rangle.
\end{split}
\end{equation*}

\subsection{Synthesis of the steps I-IV}\label{sec:Synth}
We summarize what we have done so far. We have shown that the terms
$\sigma_1^i$, $\sigma_2^i$ and  $\sigma_{3,e,j}^i$, where $i,j \in \{1, \dots, n\}$, can be represented in terms of shifts.
Also, we proved that the sum $\sum_{i=1}^n \sigma^i_{3,\pi}$ produces a paraproduct. Therefore, one of the main terms
$$
\sum_{Q \in \calD_{\good}} \Lambda_{n+1}(Q)
=\sum_{i=1}^{n} \Big[\sigma_1^i+\sigma_2^i+\sum_{j=1}^n\sigma_{3,e,j}^i+\sigma^i_{3,\pi}\Big]
$$
satisfies the required identity for the representation theorem.
By symmetry, this gives the corresponding identity for the terms $\sum_{Q \in \calD_{\good}} \Lambda_{m}(Q)$, $m \in \{1, \dots, n\}$.

\subsection{Step V: diagonal}\label{sec:Diagonal}
To finish the proof of Theorem \ref{thm:Rep} it remains to
consider the term
\begin{equation*}
\begin{split}
 \sum_{Q \in \calD_{\good}}
\Big\langle T^{m*}\big(E_{\ell(Q)}g^m_1,  \dots,
&E_{\ell(Q)} g^m_{j-1}, D_{\ell(Q)} g^m_{j}, E_{\frac{\ell(Q)}{2}} g^m_{j+1}, \dots, E_{\frac{\ell(Q)}{2}} g^m_{n} \big), \Delta_{Q}g^m_{n+1}\Big\rangle,
\end{split}
\end{equation*}
where $m \in \{1, \dots, n\}$ and $j \in \{m, \dots, n\}$. This is the term from \eqref{eq:ContainsDiag}.
For notational convenience write $g_j:=g^m_j$ and   $G=(g_1, \dots, g_n)$.
The term under consideration can be written as
\begin{equation*}
\begin{split}
 \sum_{Q \in \calD_{\good}} &\sum_{\substack{R \in \scrD_j \\ \ell(R)=\ell(Q)}}
\big \langle T^{m*}[V^j_{R} G], \Delta_Q g_{n+1} \big \rangle \\
&=\sum_{j=0}^\infty \sum_{K \in \calD} \sum_{\substack{Q \in \calD_{\good}, R \in \scrD_j \\ 2^j\ell(Q)=2^j \ell(R)= \ell(K) \\ Q \vee R =K}}
\langle\langle T^{m*} h_{R,j}, h_Q \rangle [\langle G, h_{R,j} \rangle],  \langle g_{n+1}, h_Q \rangle \rangle.
\end{split}
\end{equation*}
Notice that the common parents exist since the cubes $Q$ are good.
Those pairs $(Q,R)$,  $R=Q_1 \times \dots \times Q_n$, where $Q \cap Q_u= \emptyset$ for some $u$, can
be handled with the arguments presented above: either $Q$ and $R$ are separated as in Step I, or then
$Q$ and $R$ are close to each other as in Step II. So the new part  here is
\begin{equation}\label{eq:AlmDiag}
\begin{split}
&\sum_{Q \in \calD_{\good}}  \sum_{\substack{R =Q_1 \times \cdots \times  Q_n \in \scrD_j \\
Q_1= \cdots= Q_j = Q  \\ Q_{j+1},\dots ,Q_n  \in \ch(Q)}}
\langle\langle T^{m*} h_{R,j}, h_Q \rangle [\langle G, h_{R,j} \rangle],  \langle g_{n+1}, h_Q \rangle \rangle \\
&=\sum_{Q \in \calD} \sum_{\substack{Q_1, \dots, Q_n \in \calD \\ Q_1=\cdots= Q_j = Q_{n+1} =Q \\ Q_{j+1}, \dots, Q_{n} \in \ch(Q)}}
\langle a_{Q, Q_1 \dots, Q_{n+1}}
[ \langle f_1, \wt h_{Q_1} \rangle, \dots, \langle f_n ,\wt h_{Q_n} \rangle ],
 \langle f_{n+1}, \wt h_{Q_{n+1}} \rangle \rangle,
\end{split}
\end{equation}
where $a_{Q,(Q_i)}:=a_{Q, Q_1 \dots, Q_{n+1}}= 1_{\good}(Q)\langle T[\wt h_{Q_1}, \dots, \wt h_{Q_n}], \wt h_{Q_{n+1}} \rangle$.
We defined for $Q \in \calD$ that $1_{\good}(Q)=1$ if $Q$ is good  and $1_{\good}(Q)=0$ if $Q$ is not good.
If $j=m$, then $\wt h_{Q_i}=h^0_{Q_j}$ for $i \in \{1, \dots, n\} \setminus \{m\}$ and $\wt h_{Q_i}= h_{Q_i}$ for $i \in \{m,n+1\}$.
If $j \in \{m+1, \dots, n\}$, then $\wt h_{Q_i}=h^0_{Q_j}$ for $i \in \{1, \dots, n+1\} \setminus \{m, j\}$ and
$\wt h_{Q_i}= h_{Q_i}$ for $i \in \{m,j\}$.
We divide \eqref{eq:AlmDiag} into two by splitting the coefficients as $a_{Q,(Q_i)}=a_{Q,(Q_i),1}+a_{Q,(Q_i),2}$,
where
\begin{equation*}
\begin{split}
a_{Q,(Q_i),1}
&=1_{\good}(Q)\sum_{\substack{Q'_1, \dots, Q'_{n+1} \in \ch (Q) \\
Q'_u \not= Q'_l \text{ for some } u \text{ and } l}}
   \langle T[1_{Q'_1} \wt h_{Q_1}, \dots, 1_{Q'_{n}} \wt h_{Q_{n}}],
1_{Q_{n+1}'}\wt h_{Q_{n+1}} \rangle
\end{split}
\end{equation*}
and
\begin{equation*}
\begin{split}
a_{Q,(Q_i),2}
&= 1_{\good}(Q) \sum_{Q' \in \ch (Q)}
\langle T[1_{Q'} \wt h_{Q_1}, \dots,  1_{Q'} \wt h_{Q_n}],
1_{Q'}\wt h_{Q_{n+1}} \rangle.
\end{split}
\end{equation*}

Let $Q_1= \dots = Q_j = Q_{n+1} = Q \in \calD$ and $Q_{j+1}, \dots, Q_n \in \ch(Q)$.
Consider first the coefficient $a_{Q,(Q_i),1}$.
Suppose $Q'_1, \dots, Q'_{n+1} \in \ch (Q)$ are such that $Q'_u \not=Q'_l$
for some $u$ and $l$. Recall the function $\lambda(x,y)=\big(\sum_{m=1}^n |x-y_m|\big)^{-dn}$. Then
\begin{equation}\label{eq:DiagAdjCoef}
\langle T[1_{Q'_1}  \wt h_{Q_1}, \dots, 1_{Q'_{n}} \wt h_{Q_{n}}],
1_{Q_{n+1}'}\wt h_{Q_{n+1}} \rangle
= \int_{\R^d} \int_{\R^{dn}} \frac{K(x,y)}{\lambda(x,y)} \varphi(x,y) \ud y \ud x,
\end{equation}
where
\begin{equation*}
\varphi(x,y)=
\lambda(x,y)\prod_{m=1}^n 1_{Q'_m}(y_m) \wt h_{Q_m}(y_m) 1_{Q'_{n+1}}(x)\wt h_{Q_{n+1}}(x).
\end{equation*}
Similarly as in \eqref{eq:NearbyEst} we see that $\iint | \varphi(x,y) | \ud x \ud y \lesssim  |Q|^{-(n+1)/2} |Q|$.
Since $a_{Q,(Q_i),1}$ is a finite sum of terms of the form \eqref{eq:DiagAdjCoef}, this shows that
$$
C^{-1}\prod_{m=1}^{n+1} |Q_m|^{-1/2} |Q|^n a_{Q,(Q_i),1} \in \calA(\calC_{\CZ,\alpha}(K)).
$$
Consider then the coefficient $a_{Q,(Q_i),2}$. Suppose $Q' \in \ch(Q)$. We may obviously suppose that
$Q_{j+1} =  \cdots =  Q_n =Q'$.
Then, we have
$$
\langle T[1_{Q'} \wt h_{Q_1}, \dots,  1_{Q'} \wt h_{Q_n}],1_{Q'}\wt h_{Q_{n+1}} \rangle
= \pm \frac{|Q'|}{\prod_{m=1}^{n+1} |Q_m|^{1/2}} \frac{\langle T[1_{Q'} , \dots,  1_{Q'} ],1_{Q'} \rangle}{|Q'|},
$$
where $\langle T[1_{Q'} , \dots,  1_{Q'} ],1_{Q'} \rangle/|Q'| \in \calC_{\weak}(T)$.
Since $a_{Q,(Q_i),2}$ is a finite sum of operators of this type, we are done with Step V.

\subsection{Step VI: $T$ is not a priori bounded.}
It is possible to first prove a representation theorem in a
certain finite set up, where no a priori boundedness is needed to make the calculations legitimate (as all the sums are finite to begin with).
The proof is similar to the above except for some initial probabilistic preparations related to the finite setup inside a given fixed cube.
Reductions of this type appear e.g. in \cite{GH} and \cite{Hy4}.
We omit the technical details in our setting as they are similar.  A corollary of such a special representation is the boundedness of
$T$, say from $\prod_{m=1}^n L^{n+1}(X_m)$ to $L^{(n+1)/n}(Y_{n+1})$. After this, we can run the above argument. We are done with the proof.
\end{proof}

\begin{rem}
We make a remark here about the WBP in the linear setting, and describe a seemingly weaker condition in that setting, which can still
be used to estimate the ``diagonal'' in the $T1$ argument.

Suppose $X_1, X_2$ are $\UMD$ spaces and $1<p<\infty$.
Let $T \colon L^p(X_1) \to L^p(X_2^*)$ be a bounded linear operator.
Then, for $\{e_Q\}_{Q \in \calD} \subset X_1$, we have
\begin{equation}\label{eq:NewWeak}
\E \Big \| \sum_{Q \in \calD} \varepsilon_Q \langle T(e_Q1_Q) \rangle_Q1_Q \Big \|_{L^p(X_2^*)}
\lesssim \| T \|_{L^p(X_1) \to L^p(X_2^*)}
\E \Big \| \sum_{Q \in \calD} \varepsilon_Q  e_Q 1_Q \Big \|_{L^p(X_1)}.
\end{equation}
Indeed, the left hand side is dominated by
$$
\E \Big \| \sum_{Q \in \calD} \varepsilon_Q T(e_Q1_Q) \Big \|_{L^p(X_2^*)}
$$
by Stein's inequality, from which the claim follows using linearity.
We denote the smallest possible constant in
\eqref{eq:NewWeak} by $\calR_{\operatorname{weak}}$ (which may depend on the exponent $p$).

We recall that our usual WBP means the $\calR$-boundedness of the operators
\[|Q|^{-1} \langle T1_Q, 1_Q \rangle = \langle T1_Q \rangle_Q,\]
 which means
the estimate
\begin{equation}\label{eq:RWBP}
\E \Big\| \sum_{Q \in \calD} \varepsilon_Q \langle T(e_Q1_Q) \rangle_Q  \Big \|_{X_2^*}
\le C_{\WBP} \E \Big\| \sum_{Q \in \calD} \varepsilon_Q e_Q \Big\|_{X_1}.
\end{equation}
We show that $\calR_{\operatorname{weak}} \lesssim C_{\WBP}$. Actually,  there holds that
$$
C_{\WBP} \ge \sup_{x \in \R^d} \calR(\{ \langle T1_Q \rangle_Q \colon x \in Q \in \calD\})=: \wt C_{\WBP},
$$
and we show that $\calR_{\operatorname{weak}} \lesssim \wt C_{\WBP}$.
The proof is quite immediate.
Raising the left hand side of \eqref{eq:NewWeak} to power $p$ and using Kahane-Khintchine inequality we are left with
\begin{equation*}
\begin{split}
\E \int_{\R^d} \Big \|
&\sum_{Q \in \calD} \varepsilon_Q \langle T(e_Q1_Q) \rangle_Q1_Q(x) \Big \|_{X_2^*}^p \ud x \\
&\lesssim  \int_{\R^d} \E \calR(\{ \langle T1_Q \rangle_Q \colon x \in Q \in \calD\})^p
\Big \| \sum_{Q \in \calD} \varepsilon_Q e_Q1_Q(x) \Big \|_{X_1}^p \ud x \\
& \le  \wt C_{\WBP}^p \E \int_{\R^d}
\Big \| \sum_{Q \in \calD} \varepsilon_Q e_Q1_Q(x) \Big \|_{X_1}^p \ud x,
\end{split}
\end{equation*}
which gives the proof.

We then look at how the $\calR_{\operatorname{weak}}$-condition handles the diagonal in the $T1$ argument.
We consider a zero complexity shift whose coefficient operators $\{a_Q\}_{Q \in \calD}$ satisfy the estimate \eqref{eq:NewWeak}
(so $a_Q$ in place of $\langle T(1_Q)\rangle_Q$). Then we simply have:
\begin{equation}
\begin{split}
\Big\| \sum_{Q \in \calD} a_Q \langle f, h_Q \rangle h_Q \Big \|_{L^p(X_2^*)}
& \sim \E  \Big\| \sum_{Q \in \calD} \varepsilon_Q a_Q \langle f, h_Q \rangle 1_Q/|Q|^{1/2} \Big \|_{L^p(X_2^*)} \\
&\lesssim \E  \Big\| \sum_{Q \in \calD} \varepsilon_Q \langle f, h_Q \rangle 1_Q/|Q|^{1/2} \Big \|_{L^p(X_1)}
 \lesssim \| f \|_{L^p(X)}.
\end{split}
\end{equation}

Finally, we remark that we do not know how to formulate a similar weak boundedness condition in the multilinear setting.
\end{rem}

\section{$\calR$-boundedness of bilinear shifts}\label{sec:Rbound}
In this section we consider the $\calR$-boundedness properties of families of shifts.
Let $X_1 \dots,  X_n, Y_{n+1}$ be $\UMD$ spaces, $X_{n+1}=Y_{n+1}^*$
and let $\varpi_0 \colon \prod_m X_m \to \C$ be a contraction as in \eqref{eq:Contr}.
Fix some exponents $p_m \in (1, \infty)$ such that $\sum_{m=1}^{n+1}1/p_m=1$ and let  $\varpi \colon \prod_m L^{p_m}(X_m) \to \C$
be as in \eqref{eq:LpXContr}.

Suppose $\{S_j\}_{j}$ is a family of shifts, all of them defined with respect to a grid $\calD$ and having a fixed complexity $k$.
Recall the families $\calC(S_j)$ from \eqref{eq:NormCoef} that consist of the normalized coefficients.
It seems that the $\calR_\varpi$-boundedness condition from Definition \ref{def:Rvarpi} is not suitable for proving
that
$$
\calR_\varpi \Big(\Big\{ S_j \colon \prod_{m=1}^n L^{p_m}(X_m) \to L^{p_{n+1}'}(Y_{n+1})\Big\}_j\Big)
$$
is dominated by $\calR_{\varpi_0}(\bigcup_j \calC(S_j))$, that is,
we can not prove that if the family of coefficients $\bigcup_j \calC(S_j)$ is  $\calR_{\varpi_0}$-bounded, then $\{S_j\}_{j}$ is $\calR_{\varpi}$-bounded.

Currently, we are only able to come up with a suitable $\calR$-boundedness condition that works for families of shifts in the bilinear case. But even here we need to modify the one we used previously:
we need a somewhat stronger condition, but then also the conclusion is stronger -- i.e., the families of shits will satisfy the said stronger condition.
We now start considering the bilinear case, and here, as usual, no contraction $\varpi$ is needed.

We now introduce this stronger bilinear $\calR$-boundedness condition, which we call $\widehat{\calR}$-boundedness.
After this we will show that if the spaces $X_m$ have Pisier's property $(\alpha)$, then $\hR(\{S_j\}_{j}) \lesssim \hR (\bigcup_j \calC(S_j))$.

We denote by $\Rad_2(X)$ the space of those doubly indexed sequences $(e_{l,m})_{l,m=1}^\infty$ of elements of $X$ such that
$$
\| (e_{l,m})_{l,m=1}^\infty \|_{\Rad_2(X)}:= \Big ( \E  \Big \| \sum_{l,m=1}^\infty \varepsilon_{l,m} e_{l,m} \Big \|_X^2\Big)^{1/2}<\infty.
$$

\begin{defn}\label{def:hR}
Let $X_1$, $X_2$ and $Y_3$ be Banach spaces and write $X_3=Y_3^*$.
Suppose $\calT \subset \calL(X_1 \times X_2, Y_3)$ is a family of operators.
We say that $\calT$ is $\hR$-bounded if there exists a constant $C$ such that
for all $N \in \N$, $T_{t,u,v} \in \calT$, $e^1_{t,u} \in X_1$, $e^2_{u,v} \in X_2$ and $e^3_{t,v} \in X_3$, where $t,u,v \in \{1, \dots, N\}$, there holds that
\begin{equation*}
 \sum_{t,u,v=1}^N | \langle T_{t,u,v} [e^1_{t,u}, e^2_{u,v}], e^3_{t,v} \rangle  |
\le C \prod_{i=1}^3\| (e^i_{l,m})_{l,m=1}^N  \|_{\Rad_2(X_i)}.
\end{equation*}
The smallest possible constant $C$ is denoted by $\hR(\calT)$.
\end{defn}

\begin{thm}\label{thm:hRShift}
Let $X_1$, $X_2$ and $Y_3$ be $\UMD$ spaces with Pisier's property $(\alpha)$ and
suppose $p_1, p_2,p_3 \in (1, \infty)$ satisfy $\sum_m 1/p_m=1$.
Let $\calD$ be a dyadic lattice in $\R^d$ and fix some complexity $k=(k_1,k_2,k_3)$, $0 \le k_i \in \Z$.
Suppose $\{S_j\}_{j \in \calJ}$ is a family of operator-valued bilinear dyadic shifts with respect to the spaces $X_1$, $X_2$ and $Y_3$,
where each $S_j=S_{\calD,j}^k$ is a shift of complexity $k$ with respect to the lattice $\calD$.
Then
$$
\hR(\{S_j \colon L^{p_1}(X_1) \times L^{p_2}(X_2) \to L^{p_3'}(Y_3) \colon j \in \calJ\} )
\lesssim (1+\max_{j}k_j)\hR\Big(\bigcup_j \calC(S_j)\Big).
$$
\end{thm}

\begin{proof}
We divide the collection $\{S_j\}_{j \in \calJ}$ into three subcollections according to the type of the shifts, that is,
according to the place of the non-cancellative Haar function. We show that each of these subcollections satisfies the required estimate,
and therefore their union satisfies it also. Thus, we assume that each $S_j$ is of the form
$$
S_j(f_1,f_2)
= \sum_{K \in \calD} \sum_{\substack{I_i \in \calD \\ I_i^{(k_i)}=K}}
a^j_{K,(I_i)} [ \langle f_1, h_{I_1} \rangle, \langle f_2, h^0_{I_2} \rangle]h_{I_3}.
$$
The other two cases are handled symmetrically.

Let $N \in \N$ and suppose
$S_{t,u,v} \in \{S_j \}_{j \in \calJ}$, $f^1_{t,u} \in L^{p_1}(X_1)$, $f^2_{u,v} \in L^{p_2}(X_2)$ and $f^3_{t,v} \in L^{p_3}(X_3)$ for $t,u,v \in \{1, \dots, N\}$.
Abbreviate $\calC:=\bigcup_j \calC(S_j)$.
We need to show that for arbitrary $\epsilon_{t,u,v} \in \C$ with $|\epsilon_{t,u,v}|=1$ there holds that
\begin{equation}\label{eq:REstShift}
\begin{split}
\Big | \sum_{t,u,v=1}^N  & \epsilon_{t,u,v} \langle S_{t,u,v} [f^1_{t,u}, f^2_{u,v}], f^3_{t,v} \rangle  \Big| \\
&\lesssim  \hR(\calC)
 \| (f^1_{t,u})_{t,u}  \|_{\Rad_2(L^{p_1}(X_1))} \| (f^2_{u,v})_{u,v} \|_{\Rad_2(L^{p_2}(X_2))}\| (f^3_{t,v})_{t,v} \|_{\Rad_2(L^{p_3}(X_3))} \\
 & \sim \hR(\calC) \| F_1  \|_{L^{p_1}(\Rad_2(X_1))} \| F_2 \|_{L^{p_2}(\Rad_2(X_2))}\| F_3 \|_{L^{p_3}(\Rad_2(X_3))},
\end{split}
\end{equation}
where $F_1 \colon \R^d \to \Rad_2(X_1)$, $F_1(x)=(f_{t,u}(x))_{t,u}$, and similarly $F_2=(f^2_{u,v})_{u,v}$,
$F_3=(f^3_{t,v})_{t,v}$.
The last step was obtained using the Kahane-Khinchine inequality.
We will now construct a new shift $S$ so that
\begin{equation}\label{eq:RTrick}
\sum_{t,u,v=1}^N  \epsilon_{t,u,v} \langle S_{t,u,v} [f^1_{t,u}, f^2_{u,v}], f^3_{t,v} \rangle
=\langle S(F_1,F_2),F_3 \rangle.
\end{equation}
Then we show that
$$
|\langle S(F_1,F_2),F_3 \rangle|
\lesssim \hR(\calC) \prod _i \| F_i  \|_{L^{p_i}(\Rad_2(X_i))},
$$
from which the desired estimate \eqref{eq:REstShift} follows.

Denote the coefficients of $S_{t,u,v}$ by $a^{t,u,v}_{K,(I_i)}$.
Let $I_1,I_2,I_3,K \in \calD$ be such that $I_i^{(k_i)}=K$.
Define the operator $a_{K,(I_i)} \in \calL(\Rad_2(X_1) \times \Rad_2(X_2), \Rad_2(Y_3))$ by
$$
\langle a_{K,(I_i)}[e^1,e^2],e^3 \rangle
= \sum_{t,u,v =1}^N \epsilon_{t,u,v} \langle a^{t,u,v}_{K,(I_i)} [e^1_{t,u}, e^2_{u,v}],e^3_{t,v} \rangle,
$$
where $e^i=(e^i_{l,m})_{l,m} \in \Rad_{2}(X_i)$.
The shift $S$ is defined with these coefficients
by
$$
 S(G_1,G_2)
:= \sum_{K \in \calD} \sum_{\substack{I_i \in \calD \\ I_i^{(k_i)}=K}}
a_{K,(I_i)} [ \langle G_1, h_{I_1} \rangle, \langle G_2, h^0_{I_2} \rangle]h_{I_3},
$$
where $G_i \in L^{p_i}(\Rad_{2}(X_i))$, $i=1,2$.
We see that \eqref{eq:RTrick} is satisfied.

It remains to show that the shift $S$ is bounded, which follows by Theorem \ref{thm:MultiShifts} from the $\calR$-boundedness of the
family of coefficients $\{a_{K,(I_i)}\}_{K,(I_i)}$ (notice that $\Rad(X)$ is UMD if $X$ is). To check this, let the admissible partition
be  for example $\{\{1\}, \{2,3\}\}$.
Choose some $W \in \N$. For each $w=1, \dots, W$ let
$a_w:= a_{K(w),(I_i(w))}$ be one of the coefficients of $S$, and accordingly write $a_w^{t,u,v}:= a_{K(w),(I_i(w))}^{t,u,v}$.
Also, let  $e^1 \in \Rad_{2}(X_1)$ and $e^{i,w}=(e^{i,w}_{l,m})_{l,m} \in \Rad_{2}(X_i)$ for $i=2,3$ and $w=1, \dots, W$. Then
$$
\Big | \sum_{w=1}^W \langle a_w [e^1, e^{2,w}], e^{3,w} \rangle \Big |
= \Big | \sum_{w=1}^W \sum_{t,u,v=1}^N \epsilon_{t,u,v}\langle a_w^{t,u,v} [ e^1_{t,u}, e^{2,w}_{u,v}], e^{3,w}_{t,v} \rangle \Big |.
$$
We see that the pairs $(v,w)$ appear in $a^{t,u,v}_w$, $e^{2,w}_{u,v}$ and $e^{3,w}_{t,v}$.
Therefore, we look at the last sum as a sum over triples
$(t,u,(v,w))$. Thus, we see that
\begin{equation*}
\begin{split}
\Big | \sum_{w=1}^W  & \langle a_w [e^1, e^{2,w}], e^{3,w} \rangle \Big |
\le \hR(\calC) \| (e^1_{t,u})_{t,u} \|_{\Rad_{2}(X_1)} \\
& \times
\Big(\E \Big \| \sum_{\substack{u,v= 1, \dots, N \\ w=1, \dots, W}} \varepsilon_{u,v,w} e^{2,w}_{u,v} \Big \|_{X_2}^2 \Big)^{1/2}
\Big(\E \Big \| \sum_{\substack{t,v= 1, \dots, N \\ w=1, \dots, W}} \varepsilon_{t,v,w} e^{3,w}_{t,v} \Big \|_{X_3}^2 \Big)^{1/2}.
\end{split}
\end{equation*}
Using the fact that $X_2$ has Pisier's property $(\alpha)$, we have that
\begin{equation*}
\begin{split}
\Big(\E \Big \| \sum_{\substack{u,v= 1, \dots, N \\ w=1, \dots, W}} \varepsilon_{u,v,w} e^{2,w}_{u,v} \Big \|_{X_2}^2 \Big)^{1/2}
& \sim \Big(\E \E' \Big \| \sum_{\substack{u,v= 1, \dots, N \\ w=1, \dots, W}} \varepsilon_w \varepsilon_{u,v}' e^{2,w}_{u,v} \Big \|_{X_2}^2 \Big)^{1/2} \\
& = \Big( \E \Big \| \sum_{w=1}^W \varepsilon_w e^{2,w} \Big \|_{\Rad_{2}(X_2)}^2 \Big)^{1/2}.
\end{split}
\end{equation*}
Doing the same estimate for the term related to $X_3$ we have shown that
\begin{equation*}
\begin{split}
\Big | \sum_{w=1}^W &  \langle a_w [e^1, e^{2,w}], e^{3,w} \rangle \Big | \\
&\lesssim \| e^1 \|_{\Rad_{2}(X_1)} \| (e^{2,w})_{w=1}^W\|_{\Rad(\Rad_{2}(X_2))}
\| (e^{3,w})_{w=1}^W\|_{\Rad(\Rad_{2}(X_3))},
\end{split}
\end{equation*}
which is what we wanted to show. This concludes the proof.
\end{proof}
\section{Multi-linear multi-parameter analysis}\label{sec:multilinmultipar}
In this section we apply our operator-valued theory to prove multi-parameter estimates in our multilinear setup. Such a strategy requires
$\calR$-boundedness estimates, so in light of the previous section we will eventually have to restrict to the fundamental bilinear case. Focusing only on the essentials, we will simply prove estimates for dyadic shifts.
After this, the full paraproduct free singular integral theory in the same bilinear multi-parameter operator-valued
generality would only require the development of the corresponding representation theory. We do not anymore pursue this rather lengthy avenue here.

We define an $n$-linear $m$-parameter operator-valued dyadic shift in
$$
\R^d = \prod_{i=1}^m \R^{d_i}, \qquad d_i \ge 1.
$$
Suppose $X_1, \dots, X_n,Y_{n+1}$ are Banach spaces. Let also $\calD = \prod_{i=1}^m \calD^{d_i}$, where $\calD^{d_i}$ is a dyadic grid in $\R^{d_i}$, $i = 1, \ldots, m$. Fix the complexity $k = (k_j)_{j=1}^{n+1}$, $k_j = (k^i_j)_{i=1}^{m}$, $k^i_j \ge 0$. In what follows $\tilde h_I \in \{h_I, h_I^0\}$ if $I \in \calD^{d_i}$ for some $i$.

An $n$-linear $m$-parameter operator-valued shift $S^k = S^k_{\calD}$ has the form
$$
S^{k}(f_1, \ldots, f_n) = \sum_{K = \prod_{i=1}^m K^i \in \calD} A^k_{K}(f_1, \ldots, f_n),
$$
where
$$
A^k_K(f_1, \ldots, f_n) = \sum_{\substack{Q_1, \ldots, Q_{n+1} \in \calD  \\ Q_j^{(k_j)} = K}} a_{K, (Q_j)}
\big[  \bla f_1, \tilde h_{Q_1} \bra, \ldots, \bla f_n, \tilde h_{Q_n} \bra\big] \tilde h_{Q_{n+1}}.
$$
Here $f_j \in L^1_{\loc}(\R^d; X_j)$,
$$
a_{K,(Q_j)}=a_{K,Q_1,\dots, Q_{n+1}} \in \calL\Big(\prod_{j=1}^n X_j,Y_{n+1}\Big),
$$
$$
Q_j = \prod_{i=1}^m Q_j^i,\,\, Q_j^i \in \calD^{d_i}, \,\, Q_j^{(k_j)} := \prod_{i=1}^m (Q_j^i)^{(k_j^i)}\,\textup{ and }\, \tilde h_{Q_j} = \otimes_{i=1}^m \tilde h_{Q_j^i}.
$$
We assume that for all $K$ and the related $(Q_1, \ldots, Q_{n+1})$ we have for every
$i = 1, \ldots, m$ that in exactly two fixed positions, depending on $i$ (but on nothing else), of the tuple $(\tilde h_{Q^i_j})_{j=1}^{n+1}$ we have
cancellative Haar functions, e.g. that
$$
(\tilde h_{Q^i_j})_{j=1}^{n+1} = (h_{Q_1^i}, h_{Q_2^i}, h_{Q_3^i}^0, \ldots, h_{Q_{n+1}^i}^0).
$$
Moreover, we form the collection
\begin{equation}\label{eq:NormCoef-mpar}
\calC(S^k)=\Big\{ \frac{|K|^n}{\prod_{j=1}^{n+1} |Q_j|^{1/2}}  a_{K,(Q_j)} \colon K , Q_1, \dots, Q_{n+1} \in \calD, Q_j^{(k_j)} = K \Big\}.
\end{equation}

Define $\calD_{>t} := \calD^{d_{t+1}} \times \cdots \times \calD^{d_m}$, $t = 1, \ldots, m-1$, so that we can
write $\calD = \calD_{\le t} \times \calD_{>t}$. Define similarly e.g. $\R^{d}_{>t} = \R^{d_{t+1}} \times \cdots \times \R^{d_{m}}$,
$k_{j, >t} = (k^i_j)_{i=t+1}^{m}$.
For $K^1, Q_1^1, \ldots, Q^1_{n+1} \in \calD^{d_1}$ and functions $g_j \colon \R^d_{>1} \to X_j$ define
\begin{align*}
&S_{K^1, (Q^1_j)}(g_1, \ldots, g_n) \\
& := \sum_{K_{>1} = \prod_{i=2}^m K^i \in \calD_{>1}}
\sum_{\substack{Q_{1, >1}, \ldots, Q_{n+1, >1} \in \calD_{>1}  \\ Q_{j, >1}^{(k_{j, >1})} = K_{>1}}} a_{K, (Q_j)}
\big[  \bla g_1, \tilde h_{Q_{1, >1}} \bra, \ldots, \bla g_n, \tilde h_{Q_{n, >1}} \bra\big] \tilde h_{Q_{n+1, >1}},
\end{align*}
where $K := K^1 \times K_{>1}$, $Q_j = Q^1_j \times Q_{j, >1}$.
Notice that we can write
$$
S^{k}(f_1, \ldots, f_n) = \sum_{K^1 \in \calD^{d_1}} \sum_{\substack{Q^1_1, \ldots, Q_{n+1}^1  \in \calD^{d_1} \\ (Q_j^1)^{(k_j^1)} = K^1}}
S_{K^1, (Q^1_j)}[\langle f_1, \tilde h_{Q^1_1} \rangle, \ldots, \langle f_n, \tilde h_{Q^1_n} \rangle] \tilde h_{Q^1_{n+1}}.
$$

Let $\varpi \colon \prod_{i=1}^{n+1} X_i \to \C$ be a contraction as in \eqref{eq:Contr} and suppose $(X_1, \ldots, X_{n+1})$, where $X_{n+1} := Y_{n+1}^*$,
satisfies the $\RMF_{\varpi}$ condition.
Fix $p_j \in (1,\infty)$ so that $\sum_{j=1}^{n+1} 1/p_j = 1$
and let $\varpi_{>1} \colon \prod_{j=1}^{n+1} L^{p_j}(\R^d_{>1}; X_j) \to \C$ be as in \eqref{eq:LpXContr}. Example \ref{ex:LpRMF} says that the tuple of UMD spaces
$(L^{p_1}(\R^d_{>1}; X_1), \ldots, L^{p_{n+1}}(\R^d_{>1}; X_{n+1}))$ satisfies the  $\RMF_{\varpi_{>1}}$ condition.
Viewing $S^{k}$ as an $n$-linear operator-valued shift of complexity $(k^1_j)_{j=1}^{n+1}$ acting on functions $f_j \in L^{p_j}(\R^{d_1}; L^{p_j}(\R^{d}_{>1}; X_j)) = L^{p_j}(\R^d; X_j)$
Theorem \ref{thm:MultiShifts}
says that
\begin{align*}
\| S^k(f_1, \dots, f_n) \|_{L^{q_{n+1}}(\R^d; Y_{n+1})}
\lesssim (1+\max_j k^1_j)^{n-1}\calR_{>1} \prod_{j=1}^n \| f_j\|_{L^{p_j}(\R^d; X_j)},
\end{align*}
where
\begin{align*}
\calR_{>1} := \calR_{\varpi_{>1}}\Big( \Big\{  \frac{|K^1|}{\prod_{j=1}^{n+1} |Q^1_j|^{1/2}} S_{K^1, (Q^1_j)} \colon \prod_{j=1}^{n} L^{p_j}(\R^d_{>1}; X_j)& \to L^{p_{n+1}'}(\R^d_{>1}; Y_{n+1})\colon \\
& K^1, Q^1_j \in \calD^{d_1}, \, (Q_j^1)^{(k_j^1)} = K^1 \Big\} \Big).
\end{align*}
Now we need to revert to the bilinear setting, since as explained in Section \ref{sec:Rbound}, we do not have a suitable theory for the $\calR$-boundedness of $n$-linear shifts if $n > 2$.

\subsection{Boundedness of bilinear multi-parameter operator-valued shifts}
Suppose we have a family $\{S^k_u\}_{u \in \calU}$ of bilinear multi-parameter operator-valued shifts as above.
Suppose $X_1, X_2, Y_3$ are UMD spaces with Pisier's property $(\alpha)$. Recall that spaces of the form $L^p(\Omega; X)$ are UMD and have Pisier's property $(\alpha)$ if $X$ is
UMD and has Pisier's property $(\alpha)$. Therefore, we can iterate the above scheme using Theorem \ref{thm:hRShift} to get that given $p_1, p_2, q_3 \in (1,\infty)$ with $1/p_1 + 1/p_2 = 1/q_3$
we have
\begin{align*}
\hR(\{S_u^k \colon L^{p_1}(\R^d; X_1) \times L^{p_2}(\R^d; X_2)& \to L^{q_3}(\R^d; Y_3) \colon u \in \calU\} ) \\
&\lesssim \prod_{i=1}^m (1+\max_j k^i_j)  \hR\Big( \bigcup_{u \in \calU} \calC(S^k_u) \Big).
\end{align*}

\bibliography{OP_Multilin_7AUG}
\bibliographystyle{amsplain}

\end{document}